\documentclass[a4paper,12pt]{article}
\usepackage[utf8]{inputenc}
\usepackage[english]{babel}
\usepackage{enumitem}
\usepackage{color}
\usepackage{amsmath}
\usepackage{amssymb}
\usepackage{amsthm}
\usepackage{mathrsfs}
\usepackage{amsfonts}
\usepackage{epsfig}
\usepackage{graphicx}
\usepackage{epstopdf}
\usepackage{multirow}
\usepackage{cite}
\usepackage{caption}
\usepackage{algorithm}
\usepackage{algorithmic}

\usepackage{subfigure}
\usepackage[top=2cm,bottom=2cm,left=2cm,right=2cm]{geometry}

\usepackage[english]{babel}
\usepackage{bbding}
\usepackage{latexsym,bm}

\newtheorem{theorem}{Theorem}[section]
\newtheorem{definition}{Definition}[section]
\newtheorem{lemma}{Lemma}[section]

\newtheorem{remark}{Remark}[section]

\theoremstyle{example}
\theoremstyle{definition}
\theoremstyle{Assumption}

\allowdisplaybreaks

\begin{document}

\title{A primal-dual splitting algorithm for monotone inclusions with applications}
\author{ \sc \normalsize Changchi Huang$^{a}${\thanks{Email: cchuang@gzhu.edu.cn}},\,Jigen Peng$^{a}${\thanks{Email: jgpeng@gzhu.edu.cn}},\,Liqian Qin$^{a}${\thanks{Email: qlqmath@163.com}},\,Yuchao Tang$^{a}${\thanks{ Corresponding author. Email: hhaaoo1331@163.com}}
\\
\small $^{a}$School of Mathematics and Information Science, Guangzhou University,\\
\small Guangzhou 510006, P.R. China,\\
}
\date{}
\maketitle

	\begin{abstract}
{.\\ In this paper, we study a broad class of structured monotone inclusion problems in real Hilbert spaces. We propose a novel primal–dual splitting algorithm for solving such inclusions, which accommodates multiple monotone operators and cocoercive terms, as well as a composite monotone operator involving the linear map. The algorithm combines forward evaluations for the cocoercive components with backward resolvent steps for the monotone operators and employs a dual update for the linear composition term. It generalizes and unifies several existing methods, while requiring only a single resolvent or operator evaluation per iteration. We prove weak convergence of the iterates under standard assumptions on monotonicity and cocoercivity. Furthermore, we establish strong convergence under a mild regularity condition, such as uniform monotonicity. Numerical experiments on image deblurring and denoising problems demonstrate the efficiency and flexibility of the proposed algorithm.


 }
\end{abstract}

\newpage

\section{Introduction}

Monotone inclusion problems play a fundamental role in modern optimization theory and its applications. Such problems providing a unified framework encompassing convex minimization,
variational inequalities, equilibrium problems, and more. In particular, many tasks in signal processing, image restoration, and machine learning can be formulated as convex minimization problems and these problems can  be transformed as monotone inclusion problems in terms of finding a zero of a sum of monotone
operators. See, for example \cite{Condat2022FSP,Condat2023SIAM,Combettes2024AN}. This broad applicability has made monotone inclusion models as topic of contemporary interest in applied mathematics.

Over the past few decades, operator splitting algorithms have emerged as powerful tools for solving monotone inclusion problems. Early methods focused on splitting a sum of two monotone operators, such as the forward-backward splitting algorithm \cite{lionsandmercier1979}, the Douglas-Rachford splitting algorithm \cite{combettes2007}, the forward-backward-forward splitting algorithm \cite{Tseng2000SIAM}, and their variants \cite{lorenz2015JMIV,Bot2015AMC,Attouch2018SJO}, etc. As problem complexity grew, researchers developed more advanced splitting algorithms to handle multi-operator monotone inclusion problems involving compositions of linear operators. For example, Brice\~{n}o-Arias and Combettes \cite{briceno2011SIAM} proposed a primal-dual splitting algorithm for a sum of a maximally monotone operator and another monotone operator composed with a linear mapping. Combettes and Pesquet \cite{combettes2012} then extended this approach to a more general inclusion containing mixtures of composite, Lipschitzian, and parallel sum type operators. Around the same time, Condat \cite{condat2013} and V\~{u} \cite{vu2013ACM} independently proposed closely related primal-dual splitting algorithms for monotone inclusions arising from convex minimization. These algorithms usually referred as the Condat-Vu algorithm, which have been widely adopted in imaging and signal processing applications. For other related work, please refer to \cite{Valkonen2014,Bot2016NA,Combettes2018MP,Gao2022IP,Bang2024AMV} and references therein.

More recently, attention has focused on splitting methods capable of handling the sum of a finite number of maximally monotone operators and cocoercive operators or monotone Lipschitz operators. In particular, Arg\'{o}n-Artacho et al. \cite{ArtachoCOA2023} first proposed a class of forward-backward-type algorithms, which did not rely on reducing the problem to a two-operator inclusion in a product space. Instead, each iteration of this algorithm requires only one resolvent evaluation per set-valued operator, one forward evaluation per cocoercive operator, and two forward evaluations per monotone operator. This algorithm includes both the Davis-Yin splitting algorithm \cite{davis2015} and the resolvent splitting algorithms of \cite{MalitskyandTam2023MP} as special cases. Building on these findings, the work of \cite{ArtachoCOA2023} serves as a promising foundation for developing a more general framework that can accommodate diverse network topologies. Later, Arg\'{o}n-Artacho et al. \cite{Aragon2024Arxiv} proposed a graph-forward-backward splitting algorithm, which extended the work of \cite{BrediesSIAMJO2024} to cover the case involving a finite number of cocoercive operators. The algorithms are guided by three graphs that define variable interactions and resolvent computations. Assumptions on these graphs ensure minimal lifting and frugality: the fixed point operator acts in a minimal-dimensional space, and each resolvent and cocoercive operator is evaluated only once per iteration. In contrast, Dao et al. \cite{Dao2025Arxiv} developed a distributed splitting algorithm that splits a global task among networked nodes, each handling a single operator and communicating only with direct neighbors. The proposed algorithm in \cite{Dao2025Arxiv}  encompasses several existing methods as special cases, including the forward–backward algorithms designed for graphs \cite{Aragon2024Arxiv}, the forward–backward and forward–reflected–backward algorithms for ring networks \cite{ArtachoCOA2023}, the sequential and parallel forward–Douglas–Rachford algorithms \cite{Bredies2022SIAMJO}, the generalized forward–backward algorithm \cite{Raguet-SIAM-2013}, and various product-space formulations of the Davis–Yin algorithm, including a reduced-dimensional variant. Published online around the same time, Akerman et al. \cite{Akerman2025Arxiv} proposed a frugal splitting algorithm with minimal lifting, ensuring single-pass evaluation per operator, low memory usage, an exact fixed-point formulation, and convergence guarantees via averaged nonexpansiveness.
On the other hand, Arg\'{o}n-Artacho et al. \cite{ArtachoNA2022} proposed a primal-dual resolvent splitting algorithm with minimal lifting for finding a zero of the sum of maximally monotone operators involving compositions with bounded linear operators, which is is derived from a fixed-point characterization of the monotone inclusions. In Table \ref{table-1}, we summarize the above-mentioned related work as well as other studies from recent years on problems involving three or more monotone inclusions.

\begin{table}[h]
\centering
\caption{Monotone inclusion problems involving multiple maximally monotone operators, cocoercive operators or monotone Lipschitz continuous operators, where $A, A_1, \cdots, A_n$ are maximally monotone operators, $C, C_1, \cdots, C_m$ are cocoercive operators, $B, B_1, \cdots, B_p$ are monotone Lipschitz continuous operators, and $L$ is bounded linear operator whose adjoint is denoted by $L^{*}$. }
\begin{tabular}{c|c|c}
\hline
Monotone inclusion problems & References & Minimal lifting \\
\hline
\hline
$0\in \sum_{i=1}^{n}A_i x + B x$, $n\geq 2$ & \cite{Banert2012,Briceno2015JOTA} & No \\
$0\in \sum_{i=1}^{n}A_i x + C x$, $n\geq 2$ & \cite{Raguet-SIAM-2013,Raduet2015SIAMJIS,briceno2015Optim,Raguet2019OL} & No \\
$0\in \sum_{i=1}^{n}A_i x + L^{*}ALx + C x$, $n\geq 2$  & \cite{Tang2022JSC} & No \\
$0\in \sum_{i=1}^{n}A_i x + L^{*}ALx$, $n\geq 2$ & \cite{Yang2021IPI}  & No \\
$0\in A_1 x + A_2 x + A_3 x + Cx$ & \cite{ZongSVVA2023,Duan2025,Tang2025Opt} & Yes \\
$0\in A_1 x + A_2 x + A_3 x + Bx$ & \cite{CaoJOTA2024} & Yes \\
$0\in \sum_{i=1}^{n}A_i x + \sum_{k=1}^{n-1}C_k x$, $n\geq 2$ &  \cite{ArtachoCOA2023} & Yes \\
$0\in \sum_{i=1}^{n}A_i x + \sum_{j=1}^{n-2}B_j x$, $n\geq 3$ &  \cite{ArtachoCOA2023} & Yes \\
$0\in \sum_{i=1}^{n}A_i x + \sum_{k=1}^{m}C_k x$, $n \geq 2$, $m\geq 1$ &  \cite{Aragon2024Arxiv,Dao2025Arxiv,Akerman2025Arxiv} & Yes \\
$0\in \sum_{i=1}^{n}A_i x + L^{*}A Lx$, $n\geq 2$ & \cite{ArtachoNA2022} & Yes \\
\hline

\end{tabular}\label{table-1}

\end{table}

The purpose of this paper is to propose a fully splitting algorithm to solve the following monotone inclusion:
\begin{equation}\label{problem1}
\textrm{find }x\in \mathcal{H} \quad \textrm{such that } 0\in \sum_{i=1}^{n}A_{i}x+L^{\ast}BLx+\sum_{k=1}^{m}C_{k}x,
\end{equation}
where $n\geq 2, m \geq 1$, $A_1, \cdots, A_n :\mathcal{H}\rightarrow 2^{\mathcal{H}}$ are maximally monotone operators on a Hilbert
space $\mathcal{H}$, for each $k\in \{1, \cdots, m\}$, $C_k:\mathcal{H} \rightarrow \mathcal{H}$ is $\frac{1}{\beta_k}$-cocoercive, for some $\beta_k>0$, $L:\mathcal{H} \rightarrow \mathcal{G}$ is a bounded linear operator from $\mathcal{H}$ to Hilbert space $\mathcal{G}$ with adjoint operator $L^{*}$, and $B:\mathcal{G} \rightarrow 2^{\mathcal{G}}$ is maximally monotone operator.

Our main contributions can be summarized as follows:

(i) We develop a novel primal-dual splitting algorithm for the general monotone inclusion (\ref{problem1}). The algorithm efficiently handles the presence of multiple operator components by combining forward steps for cocoercive terms, backward resolvent steps for monotone operators, and a dual update mechanism for the composite monotone component. Importantly, it maintains a low per-iteration computational cost, requiring only one resolvent or operator evaluation per component per iteration.

(ii) We rigorously prove the convergence of the proposed algorithm. Under standard assumptions, the iterative sequence is shown to converge weakly to a solution of (\ref{problem1}). Moreover, we establish strong convergence of the iterates by assuming an additional uniform monotonicity condition on one of the maximally monotone operators.

(iii) Extensive numerical experiments on image deblurring and denoising problems demonstrate that the proposed algorithm achieves competitive performance compared with existing methods.

The remainder of this paper is organized as follows. In Section 2, we review the necessary background on monotone operator theory and convex analysis. Section 3 introduces our new splitting algorithm in detail and states the main convergence theorems. Furthermore, we demonstrate an application of the proposed algorithm to a structured convex minimization problem. Section 4 presents numerical experiments on image deblurring and denoising problems to demonstrate the performance of the proposed algorithm. Finally, we give conclusions of the paper and discuss possible directions for future research.

\section{Preliminaries\label{sec:pre}}

Throughout this paper, let $\mathcal{H}$ be a real Hilbert space, i.e., a real inner-product space that is complete with respect to the induced norm. We denote the inner product by $\langle\cdot,\cdot\rangle$ and the induced norm by $\|\cdot\|$. We abbreviate strong convergence in $\mathcal{H}$ with $\rightarrow$ and we use $\rightharpoonup$ for weak convergence. We use $\mathbb{R}_{++}$ to denote the set of all positive real numbers (strictly positive real numbers), i.e., real numbers greater than zero, and $\mathbb{R}_+$ to denote the set of non-negative real numbers (real numbers greater than or equal to zero).

Let \( A : \mathcal{H} \to 2^{\mathcal{H}} \) be a set-valued operator. We denote by \( \text{dom}\, A = \{ x \in \mathcal{H} : Ax \neq \varnothing \} \) its domain, by \( \text{zer}\, A = \{ x \in \mathcal{H} : 0 \in Ax \} \) its set of zeros, by \( \text{ran}\, A = \{ u \in \mathcal{H} : \exists x \in \mathcal{H},\, u \in Ax \} \) its range, by \( \text{gra}\, A = \{ (x, u) \in \mathcal{H} \times \mathcal{H} : u \in Ax \} \) its graph, and by \( A^{-1} : \mathcal{H} \to 2^{\mathcal{H}},\, u \mapsto \{ x \in \mathcal{H} : u \in Ax \} \) its inverse.

\begin{definition}[\cite{bauschkebook2017},Monotone operator]\label{def:monotone-operator}
Let \( A : \mathcal{H} \to 2^{\mathcal{H}} \) be a set-valued operator. The operator \( A \) is said to be
\begin{itemize}
    \item[\emph{(i)}]  monotone, if
    \[
    \langle x - y, u - v \rangle \geq 0, \quad \forall \ (x, u), (y, v) \in \text{gra}\, A.
    \]
 Furthermore, $A$ is said to be maximally monotone, if there exists no monotone operator \( A' : \mathcal{H} \to 2^{\mathcal{H}} \) such that \( \text{gra} A' \) properly contains \( \text{gra} A \),
    \item[\emph{(ii)}] uniformly monotone with modulus \( \phi_A : \mathbb{R}_+ \to [0, +\infty] \), if \( \phi_A \) is increasing, vanishes only at \( 0 \), and
    \[
    \langle x - y, u - v \rangle \geq \phi_A\bigl(\|x - y\|\bigr), \quad \forall \ (x, u), (y, v) \in \text{gra} A,
    \]
    \item[\emph{(iii)}] \( \beta \)-strongly monotone with \( \beta \in \mathbb{R}_{++} \), if it is uniformly monotone with modulus \( \phi_A : \mathbb{R}_+ \to [0, +\infty] \), \( \phi_A(t) = \beta t^2 \), i.e.,
    \[
    \langle x - y, u - v \rangle \geq \beta \|x - y\|^2, \quad \forall \ (x, u), (y, v) \in \text{gra} A.
    \]
\end{itemize}
\end{definition}

The resolvent of an operator $A: \mathcal{H} \to 2^{\mathcal{H}}$ with parameter $\gamma >0$ is defined by $J_{\gamma A} = (I + \gamma A)^{-1}$.

\begin{definition}[\cite{bauschkebook2017}]
Let \( T : \mathcal{H} \to {\mathcal{H}} \) be a single-valued operator. The operator \( T \) is said to be
\begin{itemize}
\item[\emph{(i)}] L-Lipschitz continous, if
$$
\|Tx - Ty\| \leq L \|x-y\|, \quad  \forall x,y \in \mathcal{H}.
$$
In particular, $L=1$, $T$ is said to be nonexpansive.

\item[\emph{(ii)}] $\alpha$-averaged, $\alpha \in (0,1)$, if there exists an nonexpansive operator $R$, such that
$$
T = (1-\alpha)I + \alpha R.
$$
Equivalently, $T$ is $\alpha$-averaged if and only if
 \[
 \|T x - T y\|^2 \leq \|x - y\|^2 - \frac{1 - \alpha}{\alpha} \| (I - T)x - (I - T)y \|^2, \quad \forall x, y \in \mathcal{H}.
 \]

\item[\emph{(iii)}] $\beta$-cocoercive, for some $\beta >0$, if
$$
\langle x -y, Tx - Ty\rangle \geq \beta \| Tx - Ty \|^2, \quad  \forall x,y \in \mathcal{H}.
$$

\end{itemize}

\end{definition}

We begin by recalling several preliminary results from convex analysis. Let
$
f: \mathcal{H} \rightarrow (-\infty, +\infty].
$
The effective domain of \( f \) is defined by
\[
\mathrm{dom}\,f = \{ x \in \mathcal{H} \mid f(x) < +\infty \}.
\]
The function \( f \) is said to be proper if \( \mathrm{dom}\,f \neq \emptyset \). We denote by \( \Gamma_0(\mathcal{H}) \) the class of proper, lower semi-continuous (lsc), convex functions mapping from \( \mathcal{H} \) to \( (-\infty, +\infty] \).

Let \( f \in \Gamma_0(\mathcal{H}) \). The subdifferential of \( f \) is defined as
\[
\partial f : \mathcal{H} \to 2^{\mathcal{H}}, \quad x \mapsto \left\{ v \in \mathcal{H} \mid f(y) \geq f(x) + \langle v, y - x \rangle,\ \forall y \in \mathcal{H} \right\}.
\]

\begin{definition}[\cite{bauschkebook2017}, Proximity Operator]
Let \( f \in \Gamma_0(\mathcal{H}) \). The proximity operator of \( f \) with parameter \( \lambda > 0 \) is defined by
\[
\mathrm{prox}_{\lambda f}(u) = \arg\min_{x \in \mathcal{H}} \left\{ \frac{1}{2\lambda} \|x - u\|^2 + f(x) \right\}.
\]
\end{definition}

The resolvent operator of \( \lambda \partial f \) coincides with the proximity operator of \( \lambda f \); that is,
\[
\mathrm{prox}_{\lambda f} = J_{\lambda \partial f}.
\]

In particular, when \( f(x) = \delta_C(x) \), where \( \delta_C \) is the indicator function of a set \( C \subseteq \mathcal{H} \), the proximity operator reduces to the projection onto \( C \).

The following Opial lemma is a fundamental tool when applied to prove the weak convergence of iterative sequence in Hilbert space.

\begin{lemma}[\cite{bauschkebook2017}]
Let $C$ be a nonempty subset of $\mathcal{H}$ and $\{x_k\}$ be a sequence in $\mathcal{H}$ such that

\emph{(i)} for every $x\in C$, $\lim_{k\rightarrow +\infty}\|x_k - x\|$ exists;

\emph{(ii)} every sequential weak cluster point of $\{x_k\}$ lies in $C$;

Then $\{x_k\}$ converges weakly to a point in $C$.

\end{lemma}

Finally, we review the definition of the minimax-concave (MC) penalty function; its relationship with the Huber function can be found in, e.g., \cite{Selesnick-2020-JMIV}.

\begin{definition}[MC Penalty]
The scalar minimax-concave (MC) penalty
\[
\psi_a : \mathbb{R} \to \mathbb{R}
\]
with parameter \(a > 0\) is defined as
\[
\psi_a(x) =
\begin{cases}
|x| - \dfrac{a}{2}x^2, & |x| \leq 1/a \\[1em]
\dfrac{1}{2a}, & |x| \geq 1/a.
\end{cases}
\]
For \(a = 0\), the MC penalty is defined as \(\psi_0(x) = |x|\).
\end{definition}

\section{Primal-dual splitting algorithm and convergence analysis}

In this section, we introduce a primal-dual splitting algorithm and analyze its convergence. We then apply it to solve a general class of convex minimization problems. We begin by recalling the monotone inclusion corresponding to the primal problem:
\begin{equation}\label{primal problem}
\textrm{find}\quad x\in \mathcal{H}\quad\textrm{such that}\quad 0\in \sum_{i=1}^{n}A_{i}x+L^{\ast}BLx+\sum_{k=1}^{n-1}C_{k}x,
\end{equation}
together with its dual problem:
\begin{equation}\label{dual problem}
\textrm{find}\quad u\in \mathcal{G}\quad\textrm{such that}\quad 0\in -L\left(\sum_{i=1}^{n}A_{i}+\sum_{k=1}^{n-1}C_k\right)^{-1}(-L^{\ast}u)+B^{-1}u.
\end{equation}

\begin{remark}

To construct a fully splitting algorithm for the considered monotone inclusion problem, we introduce the assumption $m=n-1$ in (\ref{problem1}). It is worth noting that this assumption has also been adopted in \cite{ArtachoCOA2023}. For the case where $m\neq n-1$, we will demonstrate that the results of this paper can still be applied through an appropriate transformation.

\emph{(i)}
If $m > n-1$, then the sum of two or more cocoercive operators remains cocoercive. Therefore, we may equivalently regard the collection of multiple cocoercive operators as a single cocoercive operator, and all conclusions presented in this paper continue to apply. Indeed, let
\[
C = \sum_{k=1}^m C_k,\qquad
\beta = \Big(\sum_{k=1}^m \beta_k \Big)^{-1},\qquad
\alpha_k = \beta\,\beta_k,
\]
so that $\sum_{k=1}^m \alpha_k = 1$. Then we have
\begin{align*}
\langle Cx - Cy,\, x-y \rangle
&= \Big\langle \sum_{k=1}^m C_k x - \sum_{k=1}^m C_k y,\, x-y \Big\rangle \\
&\ge \sum_{k=1}^m \frac{1}{\beta_k}\,\|C_k x - C_k y\|^2 \\
&= \beta \sum_{k=1}^m \alpha_k \left\| \frac{1}{\alpha_k}(C_k x - C_k y) \right\|^2 \\
&\ge \beta \left\| \sum_{k=1}^m (C_k x - C_k y) \right\|^2.
\end{align*}
Hence, it follows that
\[
\langle Cx - Cy,\, x-y \rangle
\ge \beta \| Cx - Cy \|^2,
\]
which shows that $C$ is $\beta$-cocoercive.

\emph{(ii)} If $m<n-1$, we may consider the special case in which $C_{m+1}=\cdots = C_{n-1}=0$. In this situation, the conclusions established in this paper still remain valid.

\end{remark}

\subsection{Main algorithm and convergence analysis}

In this subsection, we first present the main algorithm and then establish its convergence.

Let $\mathcal{P}$ and $\mathcal{D}$ denote the solution sets of (\ref{primal problem}) and (\ref{dual problem}), respectively. We define the set $\bm{Z}$ as:
\begin{equation}\label{the definition of Z}
\bm{Z}:= \left\{(x,u)\in \mathcal{H}\times \mathcal{G}:-L^{\ast}u \in \sum_{i=1}^{n}A_{i}x+\sum_{k=1}^{n-1}C_{k}x, \quad \textrm{and} \quad u\in B(Lx) \right\}.
\end{equation}

Based on the definitions of $\mathcal{P}$, $\mathcal{D}$, and $\mathcal{Z}$, we have
\begin{equation}
\begin{aligned}
\exists\; x\in \mathcal{P} &\Leftrightarrow (\exists \;x\in\mathcal{H})\quad 0\in \sum_{i=1}^{n}A_{i}x+L^{\ast}BLx+\sum_{k=1}^{n-1}C_{k}x\\
&\Leftrightarrow (\exists \;(x,u)\in\mathcal{H}\times\mathcal{G})
\left\{
\begin{aligned}
&-L^{\ast}u\in \sum_{i=1}^{n}A_{i}x+\sum_{k=1}^{n-1}C_{k}x\\
&u\in B(Lx)
\end{aligned}
\right.\\
&\Leftrightarrow (\exists \;(x,u)\in\mathcal{H}\times\mathcal{G})
\left\{
\begin{aligned}
&x\in \left(\sum_{i=1}^{n}A_{i}+\sum_{k=1}^{n-1}C_k\right)^{-1}(-L^{\ast}u) \\
&Lx\in B^{-1}u
\end{aligned}
\right.\\
&\Leftrightarrow (\exists \;u\in\mathcal{G})\quad 0\in -L\left(\sum_{i=1}^{n}A_{i}+\sum_{k=1}^{n-1}C_k\right)^{-1}(-L^{\ast}u)+B^{-1}u\\
&\Leftrightarrow \exists \;u\in \mathcal{D}.
\end{aligned}
\end{equation}
Therefore, it follows that $\mathcal{P}\neq \emptyset \Leftrightarrow \bm{Z}\neq \emptyset \Leftrightarrow \mathcal{D}\neq \emptyset$.

\renewcommand{\algorithmicrequire}{\textbf{Input:}}
\renewcommand{\algorithmicensure}{\textbf{Output:}}
\begin{algorithm}[htb]
\caption{Primal-dual splitting algorithm for solving (\ref{primal problem})-(\ref{dual problem}).}
\label{alg1}
\begin{algorithmic}[1]
\REQUIRE
\vskip 2mm
 Choose $\lambda\in (0,1)$, $\alpha\in (0,\frac{2}{\beta})$, and $\gamma\in \left(0,\frac{1-\frac{1}{2}\alpha\beta}{\alpha\|L\|^2}\right)$, where $\beta = \max\{\beta_k\}_{k=1}^{n-1}$. For any given $\bm{z}^0=(z_{1}^{0},\cdots,z_{n-1}^{0})\in \mathcal{H}^{n-1}$
 and $v^{0}\in \mathcal{G}$. For $k= 0, 1,2, \cdots$, compute
\begin{equation}\label{the main algorithm--1}
\dbinom{\bm{z}^{k+1}}{v^{k+1}}= \dbinom{\bm{z}^{k}}{v^{k}} + \lambda\left(
\begin{matrix}
x_{2}^{k}-x_{1}^{k}\\
x_{3}^{k}-x_{2}^{k}\\
\vdots\\
x_{n}^{k}-x_{n-1}^{k}\\
\gamma(y^{k}-Lx_{n}^{k})
\end{matrix}
\right)
\end{equation}
with
\begin{equation}\label{the main algorithm--2}
\left\{
\begin{aligned}
& x_{1}^{k}=J_{ \alpha A_{1}}(z_{1}^{k})\\
& x_{i}^{k}=J_{ \alpha A_{i}}(z_{i}^{k} + x_{i-1}^{k} - z_{i-1}^{k} - \alpha C_{i-1}x_{i-1}^{k}), \;\forall i\in [2,n-1]\\
& x_{n}^{k}=J_{ \alpha A_{n}}(x_{1}^{k} + x_{n-1}^{k} - z_{n-1}^{k}-\alpha L^{\ast}(\gamma Lx_{1}^{k}-v^{k})-\alpha C_{n-1}x_{n-1}^{k})\\
& y^{k}=J_{\frac{B}{\gamma}}\left(L(x_{1}^{k} + x_{n}^{k}) - \frac{v^{k}}{\gamma}\right)
\end{aligned}
\right.
\end{equation}
Stop when a given stopping criterion is met.
\end{algorithmic}
\end{algorithm}

\begin{remark}

\emph{(i)} Let $B=0$, and $L=0$, Algorithm \ref{alg1} reduces to
\begin{equation}\label{algorithm1-special1}
\begin{aligned}
& \bm{z}^{k+1}  = \bm{z}^{k} + \lambda\left(
\begin{matrix}
x_{2}^{k}-x_{1}^{k}\\
x_{3}^{k}-x_{2}^{k}\\
\vdots\\
x_{n}^{k}-x_{n-1}^{k}\\
\end{matrix}
\right)\\
 \textrm{with} & \\
& \left\{
\begin{aligned}
& x_{1}^{k}=J_{ \alpha A_{1}}(z_{1}^{k})\\
& x_{i}^{k}=J_{ \alpha A_{i}}(z_{i}^{k} + x_{i-1}^{k} - z_{i-1}^{k} - \alpha C_{i-1}x_{i-1}^{k}), \;\forall i\in [2,n-1]\\
& x_{n}^{k}=J_{ \alpha A_{n}}(x_{1}^{k} + x_{n-1}^{k} - z_{n-1}^{k}-\alpha C_{n-1}x_{n-1}^{k}),\\
\end{aligned}
\right.
\end{aligned}
\end{equation}
which recovers the splitting algorithm proposed in \cite{ArtachoCOA2023}.

\emph{(ii)} Let $C_k =0$, for each $k\in \{1, \cdots, n-1\}$,  Algorithm \ref{alg1} becomes
\begin{equation}\label{algorithm1-special2}
\dbinom{\bm{z}^{k+1}}{v^{k+1}}= \dbinom{\bm{z}^{k}}{v^{k}} + \lambda\left(
\begin{matrix}
x_{2}^{k}-x_{1}^{k}\\
x_{3}^{k}-x_{2}^{k}\\
\vdots\\
x_{n}^{k}-x_{n-1}^{k}\\
\gamma(y^{k}-Lx_{n}^{k})
\end{matrix}
\right)
\end{equation}
\quad   with
\begin{equation*}
\left\{
\begin{aligned}
& x_{1}^{k}=J_{ \alpha A_{1}}(z_{1}^{k})\\
& x_{i}^{k}=J_{ \alpha A_{i}}(z_{i}^{k} + x_{i-1}^{k} - z_{i-1}^{k}), \;\forall i\in [2,n-1]\\
& x_{n}^{k}=J_{ \alpha A_{n}}(x_{1}^{k} + x_{n-1}^{k} - z_{n-1}^{k}-\alpha L^{\ast}(\gamma Lx_{1}^{k}-v^{k}))\\
& y^{k}=J_{\frac{B}{\gamma}}\left(L(x_{1}^{k} + x_{n}^{k}) - \frac{v^{k}}{\gamma}\right).
\end{aligned}
\right.
\end{equation*}
When $\alpha =1$, (\ref{algorithm1-special2}) coincides with the primal-dual splitting algorithm introduced by \cite{ArtachoNA2022}.
Therefore, (\ref{algorithm1-special2}) can be viewed as a parameterized extension of the splitting algorithms in \cite{ArtachoNA2022}.

\end{remark}

To prove the convergence of Algorithm \ref{alg1}, we introduce an operator $T := T_1 \times T_2:\mathcal{H}^{n-1}\times\mathcal{G}\rightarrow\mathcal{H}^{n-1}\times\mathcal{G}$, which is defined by
\begin{equation}\label{fixed point algorithm--1}
T\dbinom{\bm{z}}{v} = \dbinom{T_1 \bm{z}}{T_2 v} = \dbinom{\bm{z}}{v} + \lambda\left(
\begin{matrix}
x_{2}-x_{1}\\
x_{3}-x_{2}\\
\vdots\\
x_{n}-x_{n-1}\\
\gamma(y-Lx_{n})
\end{matrix}
\right)
\end{equation}
where $(\bm{x},y)= (x_{1},\cdots,x_{n-1},y)\in \mathcal{H}^{n-1}\times\mathcal{G}$ depends on $(\bm{z},v)= (z_{1},\cdots,z_{n-1},v)\in \mathcal{H}^{n-1}\times\mathcal{G}$ and is described by
\begin{equation}\label{fixed point algorithm--2}
\left\{
\begin{aligned}
& x_{1}=J_{ \alpha A_{1}}(z_{1})\\
& x_{i}=J_{ \alpha A_{i}}(z_{i} + x_{i-1} - z_{i-1}- \alpha C_{i-1}x_{i-1}^{k}), \;\forall i\in [2,n-1]\\
& x_{n}=J_{ \alpha A_{n}}(x_{1} + x_{n-1} - z_{n-1}- \alpha L^{\ast}(\gamma Lx_{1}-v)-\alpha C_{n-1}x_{n-1})\\
& y=J_{\frac{B}{\gamma}}\left(L(x_{1} + x_{n}) - \frac{v}{\gamma}\right)
\end{aligned}
\right.
\end{equation}
It is easy to see that the sequence $\{(\bm{z}^{k},v^{k})\}$ generated by Algorithms \ref{alg1} satisfies
$$
(\bm{z}^{k+1}, v^{k+1}) = T(\bm{z}^{k}, v^k) = (T_1 \bm{z}^{k}, T_2 v^k ), \quad \forall k\geq 0.
$$


The following lemma establishes a connection between the solution sets of the primal and dual problems (\ref{primal problem})-(\ref{dual problem}) and the set of fixed points of the operator $T$.

\begin{lemma}\label{the lemma of two nonempty solution set}
Let $n\geq 2$ and $\lambda,\gamma >0$. Then the following statements hold:

\emph{(i)} If $(\bar{x},\bar{u})\in \bm{Z}$, then there exists $\bar{\bm{z}}=(\bar{z}_{1},\cdots,\bar{z}_{n-1})\in \mathcal{H}^{n-1}$ such that $(\bar{\bm{z}},\gamma L\bar{x}-\bar{u})\in \emph{Fix}\;T$.

\emph{(ii)} If $(\bar{z}_1,\cdots,\bar{z}_{n-1},\bar{v})\in \emph{Fix}\;T$, then $(J_{\alpha A_1}(\bar{z}_1),\gamma L\bar{x}-\bar{v})\in \bm{Z}$.
Hence, $\emph{Fix}\;T\neq \emptyset \Leftrightarrow \bm{Z}\neq \emptyset$.
\end{lemma}

\begin{proof}
(i) Let $(\bar{x},\bar{u})\in \bm{Z}$ and take $a_{i}\in \mathcal{H}$ such that  $a_{i}\in A_{i}\bar{x}$ for all $i\in [1,n]$. It follows from the definition of $\bm{Z}$ that $\bar{u}\in B(L\bar{x})$ and $-L^{\ast}\bar{u}-\sum_{k=1}^{n-1}C_k\bar{x}=\sum_{i=1}^{n}a_{i}$. We further define the vectors $(\bar{z}_1,\cdots,\bar{z}_{n-1},\bar{v})\in \mathcal{H}^{n-1}\times \mathcal{G}$ by
\begin{equation}
\left\{
\begin{aligned}
& \bar{z}_{1}:=\bar{x}+\alpha a_{1}\in (I+ \alpha A_{1})(\bar{x})\\
& \bar{z}_{i}:= \alpha a_{i}+\bar{z}_{i-1} + \alpha C_{i-1}\bar{x} \in (\text{Id}+ \alpha A_{i})(\bar{x})-\bar{x}+\bar{z}_{i-1}
+ \alpha C_{i-1}\bar{x}, \;\forall i\in [2,n-1]\\
& \bar{v}:=\gamma L\bar{x}-\bar{u}\in (\gamma I -B)(L\bar{x}),
\end{aligned}
\right.
\end{equation}
from which we obtain that $\bar{x}=J_{ \alpha A_{1}}(\bar{z}_{1})$ and $\bar{x}=J_{ \alpha A_{i}}(\bar{z}_{i} + \bar{x} - \bar{z}_{i-1} - \alpha C_{i-1}\bar{x})$ for all $i\in [2,n-1]$. Moreover, we deduce that
\begin{equation}
\begin{aligned}
& \quad 2\bar{x}-\bar{z}_{n-1}- \alpha L^{\ast}(\gamma L\bar{x}-\bar{v})- \alpha C_{n-1}\bar{x} \\
& = 2\bar{x}-\bar{z}_{n-1}- \alpha L^{\ast}\bar{u}- \alpha C_{n-1}\bar{x}\\
&= \bar{x}+\alpha a_{n}+\bar{x}-\bar{z}_{n-1}+\alpha \sum_{i=1}^{n-1}a_{i}+ \alpha \sum_{k=1}^{n-2}C_{k}\bar{x}\\
&= \bar{x}+\alpha a_{n}+\bar{x}-\bar{z}_{n-1}+\sum_{i=2}^{n-1}(\bar{z}_{i}-\bar{z}_{i-1})+\bar{z}_{1}-\bar{x}\\
& \in (\text{Id}+ \alpha A_{n})(\bar{x}).
\end{aligned}
\end{equation}
Altogether, it yields
\begin{equation}
\left\{
\begin{aligned}
& \bar{x}=J_{ \alpha A_{1}}(\bar{z}_{1})\\
& \bar{x}=J_{ \alpha A_{i}}(\bar{z}_{i} + \bar{x} - \bar{z}_{i-1} - \alpha C_{i-1}\bar{x}), \;\forall i\in [2,n-1]\\
& \bar{x}=J_{ \alpha A_{n}}(2\bar{x} - \bar{z}_{n-1}- \alpha L^{\ast}(\gamma L\bar{x}-\bar{v})-\alpha C_{n-1}\bar{x})\\
& L\bar{x}=J_{\frac{B}{\gamma}}\left(2L\bar{x} - \frac{\bar{v}}{\gamma}\right),
\end{aligned}
\right.
\end{equation}
which implies that $(\bar{\bm{z}},\gamma L\bar{x}-\bar{u})\in \text{Fix}\;T$.

(ii) Let $(\bar{z}_1,\cdots,\bar{z}_{n-1},\bar{v})\in \text{Fix}\;T$ and $\bar{x}=J_{ \alpha A_{1}}(\bar{z}_{1})$. It follows from (\ref{fixed point algorithm--1}) that $\bar{y}=L\bar{x}$ and $\bar{x}_{i}=\bar{x}$ for all $i\in [1,n]$. In this way, (\ref{fixed point algorithm--2}) can be rewritten in the following form:
\begin{equation}\label{the inclusion of fixed point}
\left\{
\begin{aligned}
& \bar{z}_{1}-\bar{x}\in \alpha A_{1}\bar{x}\\
& \bar{z}_{i}-\bar{z}_{i-1} - \alpha C_{i-1}\bar{x}\in \alpha A_{i}\bar{x}, \;\forall i\in [2,n-1]\\
& \bar{x}-\bar{z}_{n-1}- \alpha L^{\ast}(\gamma L\bar{x}-\bar{v})- \alpha C_{n-1}\bar{x}\in { \alpha A_{n}}\bar{x}\\
& \gamma L\bar{x}-\bar{v}\in B(L\bar{x}),
\end{aligned}
\right.
\end{equation}
Summing the first $n$ inclusions in (\ref{the inclusion of fixed point}) and assuming $\bar{u}:=\gamma L\bar{x}-\bar{v}$, we have
\begin{equation}
\left\{
\begin{aligned}
& -L^{\ast}\bar{u} \in \sum_{i=1}^{n}A_{i}\bar{x}+\sum_{k=1}^{n-1}C_{k}\bar{x}\\
& \bar{u}\in B(L\bar{x}),
\end{aligned}
\right.
\end{equation}
which implies that $(J_{\alpha A_1}(\bar{z}_1),\gamma L\bar{x}-\bar{v})\in \bm{Z}$. Hence, the result follows from (i) and (ii).
\end{proof}

In what follows, we present a technical lemma concerning the nonexpansive properties of the fixed-point operator $T$. We begin by defining a scalar product on the real Hilbert space $\mathcal{H}^{n-1}\times \mathcal{G}$ as
\begin{equation}\label{new inner product}
\langle ({z}_1,\cdots,{z}_{n-1},{v}),(\bar{z}_1,\cdots,\bar{z}_{n-1},\bar{v})\rangle_{\frac{\alpha}{\gamma}}:=\sum_{i=1}^{n-1}\langle z_{i},\bar{z}_{i}\rangle_{\mathcal{H}}+ \frac{\alpha}{\gamma} \langle v,\bar{v}\rangle_{\mathcal{G}},
\end{equation}
for all $({z}_1,\cdots,{z}_{n-1},{v}),(\bar{z}_1,\cdots,\bar{z}_{n-1},\bar{v})\in \mathcal{H}^{n-1}\times \mathcal{G}$.

\begin{lemma}\label{the lemma of nonexpansion}
Let $(\bm{\bar{z}},\bar{v})=(\bar{z}_1,\cdots,\bar{z}_{n-1},\bar{v}) \in \mathcal{H}^{n-1}\times \mathcal{G}$, and $(\bm{z},{v})=({z}_1,\cdots,{z}_{n-1},{v})\in \mathcal{H}^{n-1}\times \mathcal{G}$. Then, we have
\begin{equation}\label{the inequality of nonexpansion}
\begin{aligned}
& \quad \|T(\bm{z},v)-T(\bm{\bar{z}},\bar{v})\|_{\frac{\alpha}{\gamma}}^{2}  \\
&\leq \|(\bm{z},v)-(\bm{\bar{z}},\bar{v})\|_{\frac{\alpha}{\gamma}}^{2}
- \frac{1-\lambda-\frac{\alpha\beta}{2}}{\lambda}
\|(\emph{I}-T_1)\bm{z}-(\emph{I}-T_1)\bm{\bar{z}}\|^2 \\
&\quad
- \frac{1-\lambda}{\lambda}\frac{\alpha}{\gamma}
\|(\emph{I}-T_2)v-(\emph{I}-T_2)\bar{v}\|^2,
\end{aligned}
\end{equation}
where $\|\cdot\|_{\frac{\alpha}{\gamma}}$ denotes the norm induced by the product (\ref{new inner product}) and $\beta = \max\{\beta_k\}_{k=1}^{n-1}$. In particular, if $\alpha\in (0,\frac{2}{\beta}), \gamma\in (0,\frac{1-\frac{1}{2}\alpha\beta}{\alpha\|L\|^2})$ and $\lambda\in (0,1-\frac{\alpha\beta}{2})$, the operator $T$ is $\theta$-averaged nonexpansive with
\[
\theta=\frac{\lambda}{1-\frac{\alpha\beta}{2}}\in(0,1).
\]
Furthermore, if one of the operators $A_i$ is uniformly monotone for some $i\in [1,n]$, then we obtain
\begin{equation}\label{the inequality of nonexpansion2}
\begin{aligned}
& \quad \|T(\bm{z},v)-T(\bm{\bar{z}},\bar{v})\|_{\frac{\alpha}{\gamma}}^{2}  \\
&\leq \|(\bm{z},v)-(\bm{\bar{z}},\bar{v})\|_{\frac{\alpha}{\gamma}}^{2}
- \frac{1-\lambda-\frac{\alpha\beta}{2}}{\lambda}
\|(\emph{I}-T_1)\bm{z}-(\emph{I}-T_1)\bm{\bar{z}}\|^2 \\
&\quad
- \frac{1-\lambda}{\lambda}\frac{\alpha}{\gamma}
\|(\emph{I}-T_2)v-(\emph{I}-T_2)\bar{v}\|^2
-2\lambda\phi_{A_i}(\|x_i - \bar{x}_{i}\|).
\end{aligned}
\end{equation}
\end{lemma}

\begin{proof}
Consider the vectors $(x_{1},\cdots,x_{n},y)\in \mathcal{H}^{n}\times\mathcal{G}$ and $(\bar{x}_{1},\cdots,\bar{x}_{n},\bar{y})\in \mathcal{H}^{n}\times\mathcal{G}$ obtained from (\ref{fixed point algorithm--2}) using $(\bm{z},{v})$ and $(\bm{\bar{z}},\bar{v})$, respectively. For convenience, we denote $(\bm{z}^{+},{v}^{+})=T(\bm{z},{v})$ and $(\bm{\bar{z}}^{+},\bar{v}^{+})=T(\bm{\bar{z}},\bar{v})$. Since $z_{1}-x_{1}\in \alpha A_{1}(x_{1})$ and $\bar{z}_{1}-\bar{x}_{1}\in \alpha A_{1}(\bar{x}_{1})$, the monotonicity of $A_{1}$ implies
\begin{equation}\label{the monotonicity of A1}
0\leq \langle (z_{1}-x_{1}) - (\bar{z}_{1}-\bar{x}_{1}), x_{1}-\bar{x}_{1}\rangle.
\end{equation}

In addition, for all $i\in [2,n-1]$, it holds that $z_{i}+x_{i-1}-z_{i-1}-x_{i}- \alpha C_{i-1}(x_{i-1})\in \alpha A_{i}x_{i}$ and $\bar{z}_{i}+\bar{x}_{i-1}-\bar{z}_{i-1}-\bar{x}_{i}- \alpha C_{i-1}\bar{x}_{i-1}\in \alpha A_{i}\bar{x}_{i}$. Moreover, it follows from the monotonicity of $A_{i}$ that
\begin{equation}\label{the monotonicity of Ai}
\begin{aligned}
0&\leq \langle (z_{i}+x_{i-1}-z_{i-1}-x_{i}- \alpha C_{i-1}x_{i-1}) - (\bar{z}_{i}+\bar{x}_{i-1}-\bar{z}_{i-1}-\bar{x}_{i}- \alpha C_{i-1}\bar{x}_{i-1}), x_{i}-\bar{x}_{i}\rangle\\
&= \langle (z_{i}-x_{i}) - (\bar{z}_{i}-\bar{x}_{i}), x_{i}-\bar{x}_{i}\rangle - \langle (z_{i-1}-x_{i-1}) - (\bar{z}_{i-1}-\bar{x}_{i-1}), x_{i}-\bar{x}_{i} \rangle\\
&~~~~- \alpha \langle C_{i-1}x_{i-1} - C_{i-1}\bar{x}_{i-1}, x_{i}-\bar{x}_{i}\rangle .
\end{aligned}
\end{equation}
From the second-to-last line in (\ref{fixed point algorithm--2}), we have $x_{1}+x_{n-1}-z_{n-1}-x_{n}-\alpha L^{\ast}(\gamma Lx_{1}-v)- \alpha C_{n-1}x_{n-1}\in \alpha A_{n}x_{n}$ and $\bar{x}_{1}+\bar{x}_{n-1}-\bar{z}_{n-1}-\bar{x}_{n}- \alpha L^{\ast}(\gamma L\bar{x}_{1}-\bar{v})- \alpha C_{n-1}\bar{x}_{n-1}\in \alpha A_{n}\bar{x}_{n}$. Moreover, the monotonicity of $A_{n}$ implies
\begin{equation}\label{the monotonicity of An}
\begin{aligned}
0\leq \;& \langle x_{1}+x_{n-1}-z_{n-1}-x_{n}-\alpha L^{\ast}(\gamma Lx_{1}-v)-\alpha C_{n-1}x_{n-1}, x_{n}-\bar{x}_{n}\rangle \\
& - \langle \bar{x}_{1}+\bar{x}_{n-1}-\bar{z}_{n-1}-\bar{x}_{n}- \alpha L^{\ast}(\gamma L\bar{x}_{1}-\bar{v})- \alpha C_{n-1}\bar{x}_{n-1}, x_{n}-\bar{x}_{n}\rangle\\
= \;& \langle (x_{n-1}-z_{n-1})-(\bar{x}_{n-1}-\bar{z}_{n-1}), x_{n}-\bar{x}_{n}\rangle + \langle (x_{1}-\bar{x}_{1})-({x}_{n}-\bar{x}_{n}), x_{n}-\bar{x}_{n}\rangle\\
& - \alpha \langle \gamma(Lx_{1}-L\bar{x}_{1})-(v-\bar{v}), Lx_{n}-L\bar{x}_{n}\rangle - \alpha \langle C_{n-1}x_{n-1}-C_{n-1}\bar{x}_{n-1}, x_{n}-\bar{x}_{n}\rangle .
\end{aligned}
\end{equation}
Finally, the last line in (\ref{fixed point algorithm--2}) gives $\gamma L(x_{1}+x_{n})-v-\gamma y\in By$ and $\gamma L(\bar{x}_{1}+\bar{x}_{n})-\bar{v}-\gamma \bar{y}\in B\bar{y}$. By the monotonicity of $B$, we have
\begin{equation}\label{the monotonicity of B}
\begin{aligned}
0\leq \langle (\gamma L(x_{1}+x_{n})-v-\gamma y) - (\gamma L(\bar{x}_{1}+\bar{x}_{n})-\bar{v}-\gamma \bar{y}), y-\bar{y}\rangle .
\end{aligned}
\end{equation}
By summing (\ref{the monotonicity of A1})-(\ref{the monotonicity of B}), we obtain the inequality
\begin{equation}\label{the inequality--1}
\begin{aligned}
0\leq &\sum_{i=1}^{n-1}\langle (x_{i}-x_{i+1}) - (\bar{x}_{i}-\bar{x}_{i+1}),z_{i}-\bar{z}_{i} \rangle +\sum_{i=1}^{n-1}\langle (x_{i}-\bar{x}_{i}) - (x_{i+1}-\bar{x}_{i+1}),\bar{x}_{i}-x_{i} \rangle \\
&+ \langle (x_{1}-\bar{x}_{1}) - (x_{n}-\bar{x}_{n}),{x}_{n}-\bar{x}_{n} \rangle + \alpha \langle (Lx_{n}-L\bar{x}_{n}) - (y-\bar{y}),{v}-\bar{v} \rangle \\
&+ \alpha\gamma \langle (L(x_{1}+x_{n}) - L(\bar{x}_{1}+\bar{x}_{n}))-(y-\bar{y}), y-\bar{y}\rangle - \alpha\gamma\langle Lx_{1}-L\bar{x}_{1},Lx_{n}-L\bar{x}_{n}\rangle \\
&-\alpha\sum_{i=1}^{n-1}\langle C_{i}x_{i}-C_{i}\bar{x}_{i}, x_{i+1}-\bar{x}_{i+1}\rangle .
\end{aligned}
\end{equation}
The first term in (\ref{the inequality--1}) can be written as
\begin{equation}\label{the first term}
\begin{aligned}
& \quad \sum_{i=1}^{n-1}\langle (x_{i}-x_{i+1}) - (\bar{x}_{i}-\bar{x}_{i+1}),z_{i}-\bar{z}_{i} \rangle \\
&=\frac{1}{\lambda}\sum_{i=1}^{n-1}\langle (z_{i}-z_{i}^{+}) - (\bar{z}_{i}-\bar{z}_{i}^{+}),{z}_{i}-\bar{z}_{i} \rangle \\
&=\frac{1}{\lambda}\langle (\bm{z}-\bm{z}^{+}) - (\bm{\bar{z}}-\bm{\bar{z}}^{+}),\bm{z}-\bm{\bar{z}} \rangle \\
&=\frac{1}{2\lambda}(\|(\bm{z}-\bm{z}^{+}) - (\bm{\bar{z}}-\bm{\bar{z}}^{+})\|^2 - \|\bm{z}^{+} - \bm{\bar{z}}^{+}\|^2 + \|\bm{z}-\bm{\bar{z}}\|^2).
\end{aligned}
\end{equation}
The second term in (\ref{the inequality--1}) can be equivalently expressed as
\begin{equation}\label{the second term}
\begin{aligned}
& \quad \sum_{i=1}^{n-1}\langle (x_{i}-\bar{x}_{i}) - (x_{i+1}-\bar{x}_{i+1}),\bar{x}_{i}-x_{i} \rangle \\
&=\frac{1}{2}\sum_{i=1}^{n-1}(\|x_{i+1}-\bar{x}_{i+1}\|^2 - \|{x}_{i}-\bar{x}_{i}\|^2 - \|(x_{i}-x_{i+1}) - (\bar{x}_{i}-\bar{x}_{i+1})\|^2) \\
&=\frac{1}{2}(\|x_{n}-\bar{x}_{n}\|^2 - \|{x}_{1}-\bar{x}_{1}\|^2 - \frac{1}{{\lambda}^2}\sum_{i=1}^{n-1}\|(z_{i}-z_{i}^{+}) - (\bar{z}_{i}-\bar{z}_{i}^{+})\|^2) \\
&=\frac{1}{2}(\|x_{n}-\bar{x}_{n}\|^2 - \|{x}_{1}-\bar{x}_{1}\|^2 - \frac{1}{{\lambda}^2}\|(\bm{z}-\bm{z}^{+}) - (\bm{\bar{z}}-\bm{\bar{z}}^{+})\|^2).
\end{aligned}
\end{equation}
The third term in (\ref{the inequality--1}) yields
\begin{equation}\label{the third term}
\begin{aligned}
& \quad \langle (x_{1}-\bar{x}_{1}) - (x_{n}-\bar{x}_{n}),{x}_{n}-\bar{x}_{n} \rangle\\
&=\frac{1}{2}(\|x_{1}-\bar{x}_{1}\|^2 - \|x_{n}-\bar{x}_{n}\|^2 - \|(x_{1}-\bar{x}_{1}) - (x_{n}-\bar{x}_{n})\|^2).
\end{aligned}
\end{equation}
The fourth term in (\ref{the inequality--1}) becomes
\begin{equation}\label{the fourth term}
\begin{aligned}
& \quad \langle (Lx_{n}-L\bar{x}_{n}) - (y-\bar{y}),{v}-\bar{v} \rangle \\
&=\frac{1}{\gamma\lambda} \langle ({v}-{v}^{+}) - (\bar{v}-\bar{v}^{+}),{v}-\bar{v} \rangle \\
&=\frac{1}{2\gamma\lambda}(\|({v}-{v}^{+}) - (\bar{v}-\bar{v}^{+})\|^2 - \|{v}^{+} - {\bar{v}}^{+}\|^2 + \|{v}-\bar{v}\|^2).
\end{aligned}
\end{equation}
The fifth term in (\ref{the inequality--1}) can be estimated as
\begin{equation}\label{the fifth term}
\begin{aligned}
& \quad \gamma \langle (L(x_{1}+x_{n}) - L(\bar{x}_{1}+\bar{x}_{n}))-(y-\bar{y}), y-\bar{y}\rangle \\
&=\gamma\langle Lx_{1} - L\bar{x}_{1}, y-\bar{y}\rangle + \gamma\langle (Lx_{n} - L\bar{x}_{n})-(y-\bar{y}), y-\bar{y}\rangle \\
&\leq \frac{\gamma}{2}(\|Lx_{1} - L\bar{x}_{1}\|^2+\|y-\bar{y}\|^2) + \frac{\gamma}{2}(\|Lx_{n} - L\bar{x}_{n}\|^2 -\|(Lx_{n} - L\bar{x}_{n})-(y-\bar{y})\|^2-\|y-\bar{y}\|^2) \\
&=\frac{\gamma}{2}\|Lx_{1} - L\bar{x}_{1}\|^2 + \frac{\gamma}{2}\|Lx_{n} - L\bar{x}_{n}\|^2 - \frac{1}{2\gamma\lambda^2}\|({v}-{v}^{+}) - (\bar{v}-\bar{v}^{+})\|^2 .
\end{aligned}
\end{equation}
By the Lipschitz property of $L$, the second last term in (\ref{the inequality--1}) can be expressed as
\begin{equation}\label{the second last term}
\begin{aligned}
& \quad -\gamma \langle  Lx_{1}-L\bar{x}_{1},Lx_{n}-L\bar{x}_{n}\rangle \\
& =\frac{\gamma}{2}(\|L(x_{1}-x_{n})-L(\bar{x}_{1}-\bar{x}_{n})\|^2 - \|Lx_{1}-L\bar{x}_{1}\|^2 - \|Lx_{n}-L\bar{x}_{n}\|^2)\\
&\leq \frac{\gamma\|L\|^2}{2}\|(x_{1}-x_{n})-(\bar{x}_{1}-\bar{x}_{n})\|^2
-\frac{\gamma}{2}\|Lx_{1}-L\bar{x}_{1}\|^2
-\frac{\gamma}{2}\|Lx_{n}-L\bar{x}_{n}\|^2 .
\end{aligned}
\end{equation}
Finally, using the cocoercivity property of the operator $C_i$, the last term in (\ref{the inequality--1}) can be estimated as
\begin{equation}\label{the last term}
\begin{aligned}
& \quad -\sum_{i=1}^{n-1}\langle C_{i}x_{i}-C_{i}\bar{x}_{i}, x_{i+1}-\bar{x}_{i+1}\rangle \\
&=\sum_{i=1}^{n-1}\langle C_{i}x_{i}-C_{i}\bar{x}_{i}, (\bar{x}_{i+1}-\bar{x}_{i})-(x_{i+1}-x_i)\rangle \\
&~~~~+\sum_{i=1}^{n-1}\langle C_{i}x_{i}-C_{i}\bar{x}_{i}, \bar{x}_{i}-x_{i}\rangle \\
& \leq \sum_{i=1}^{n-1}\langle C_{i}x_{i}-C_{i}\bar{x}_{i}, (\bar{x}_{i+1}-\bar{x}_{i})-(x_{i+1}-x_i)\rangle
-\sum_{i=1}^{n-1} \frac{1}{\beta_i} \| C_{i}x_{i}-C_{i}\bar{x}_{i} \|^2 \\
&\leq \frac{\beta}{4}\sum_{i=1}^{n-1}\| (\bar{x}_{i+1}-\bar{x}_{i})-(x_{i+1}-x_i)\|^2
+ \frac{1}{\beta}\sum_{i=1}^{n-1}\| C_{i}x_{i}-C_{i}\bar{x}_{i}\|^2 \\
&~~~~-\frac{1}{\beta}\sum_{i=1}^{n-1}\| C_{i}x_{i}-C_{i}\bar{x}_{i}\|^2 \\
&= \frac{\beta}{4}\sum_{i=1}^{n-1}\| (\bar{x}_{i+1}-\bar{x}_{i})-(x_{i+1}-x_i)\|^2 \\
&= \frac{\beta}{4\lambda^2}\|(\bm{z}-\bm{z}^{+})-(\bm{\bar{z}}-\bm{\bar{z}}^{+})\|^2.
\end{aligned}
\end{equation}
Substituting (\ref{the first term})-(\ref{the last term}) into the inequality (\ref{the inequality--1}) and multiplying by $2\lambda$, it follows that
\begin{equation}\label{the inequality--2}
\begin{aligned}
0\leq &\|(\bm{z}-\bm{z}^{+}) - (\bm{\bar{z}}-\bm{\bar{z}}^{+})\|^2 - \|\bm{z}^{+} - \bm{\bar{z}}^{+}\|^2 + \|\bm{z}-\bm{\bar{z}}\|^2 \\
&- \frac{1}{{\lambda}}\|(\bm{z}-\bm{z}^{+}) - (\bm{\bar{z}}-\bm{\bar{z}}^{+})\|^2 - \lambda\|(x_{1}-\bar{x}_{1}) - (x_{n}-\bar{x}_{n})\|^2 \\
& + \frac{\alpha}{\gamma}(\|({v}-{v}^{+}) - (\bar{v}-\bar{v}^{+})\|^2 - \|{v}^{+} - {\bar{v}}^{+}\|^2 + \|{v}-\bar{v}\|^2)\\
& -\frac{\alpha}{\gamma\lambda}\|({v}-{v}^{+}) - (\bar{v}-\bar{v}^{+})\|^2
+ {\alpha\lambda\gamma\|L\|^2}\|(x_{1}-x_{n})-(\bar{x}_{1}-\bar{x}_{n})\|^2 \\
& + \frac{\alpha\beta}{2\lambda}\|(\bm{z}-\bm{z}^{+})-(\bm{\bar{z}}-\bm{\bar{z}}^{+})\|^2.
\end{aligned}
\end{equation}
Combining the two terms involving $\|(x_{1}-x_{n})-(\bar{x}_{1}-\bar{x}_{n})\|^2$ in (\ref{the inequality--2}), we get
\[
-\lambda\|(x_{1}-x_{n})-(\bar{x}_{1}-\bar{x}_{n})\|^2
+\alpha\lambda\gamma\|L\|^2\|(x_{1}-x_{n})-(\bar{x}_{1}-\bar{x}_{n})\|^2
\]
\[
=
-\lambda(1-\alpha\gamma\|L\|^2)\|(x_{1}-x_{n})-(\bar{x}_{1}-\bar{x}_{n})\|^2.
\]
Thus, after rearrangement, inequality (\ref{the inequality--2}) yields
\begin{equation}
\begin{aligned}
\quad & \|\bm{z}^{+}-\bm{\bar{z}}^{+}\|^2 + \frac{\alpha}{\gamma}\|{v}^{+} - {\bar{v}}^{+}\|^2  \\
&+ \frac{1-\lambda-\frac{\alpha\beta}{2}}{\lambda}
\|(\bm{z}-\bm{z}^{+})-(\bm{\bar{z}}-\bm{\bar{z}}^{+})\|^{2} \\
&+ \frac{\alpha}{\gamma}\frac{1-\lambda}{\lambda}
\|({v}-{v}^{+})-({\bar{v}}-{\bar{v}}^{+})\|^{2} \\
&+ \lambda(1-\alpha\gamma\|L\|^2)
\|(x_{1}-x_{n})-(\bar{x}_{1}-\bar{x}_{n})\|^2 \\
&\leq \|\bm{z}-\bm{\bar{z}}\|^2 + \frac{\alpha}{\gamma}\|v-\bar{v}\|^2.
\end{aligned}
\end{equation}
Since
\[
\gamma\in \left(0,\frac{1-\frac{1}{2}\alpha\beta}{\alpha\|L\|^2}\right)
\]
and $\alpha\in(0,\frac{2}{\beta})$, we have
\[
1-\alpha\gamma\|L\|^2>0.
\]
Hence,
\[
\lambda(1-\alpha\gamma\|L\|^2)
\|(x_{1}-x_{n})-(\bar{x}_{1}-\bar{x}_{n})\|^2\geq 0.
\]
Dropping this nonnegative term from the left-hand side gives
\begin{equation}
\begin{aligned}
\quad & \|\bm{z}^{+}-\bm{\bar{z}}^{+}\|^2 + \frac{\alpha}{\gamma}\|{v}^{+} - {\bar{v}}^{+}\|^2  \\
&+ \frac{1-\lambda-\frac{\alpha\beta}{2}}{\lambda}
\|(\bm{z}-\bm{z}^{+})-(\bm{\bar{z}}-\bm{\bar{z}}^{+})\|^{2} \\
&+ \frac{\alpha}{\gamma}\frac{1-\lambda}{\lambda}
\|({v}-{v}^{+})-({\bar{v}}-{\bar{v}}^{+})\|^{2} \\
&\leq \|\bm{z}-\bm{\bar{z}}\|^2 + \frac{\alpha}{\gamma}\|v-\bar{v}\|^2.
\end{aligned}
\end{equation}
Equivalently,
\begin{equation}
\begin{aligned}
& \quad \|T(\bm{z},v)-T(\bm{\bar{z}},\bar{v})\|_{\frac{\alpha}{\gamma}}^{2}  \\
&\leq \|(\bm{z},v)-(\bm{\bar{z}},\bar{v})\|_{\frac{\alpha}{\gamma}}^{2}
- \frac{1-\lambda-\frac{\alpha\beta}{2}}{\lambda}
\|(\emph{I}-T_1)\bm{z}-(\emph{I}-T_1)\bm{\bar{z}}\|^2 \\
&\quad
- \frac{1-\lambda}{\lambda}\frac{\alpha}{\gamma}
\|(\emph{I}-T_2)v-(\emph{I}-T_2)\bar{v}\|^2,
\end{aligned}
\end{equation}
which implies that the inequality (\ref{the inequality of nonexpansion}) holds.

Moreover, since
\[
\frac{1-\lambda-\frac{\alpha\beta}{2}}{\lambda}
\leq
\frac{1-\lambda}{\lambda},
\]
we further have
\begin{equation}
\begin{aligned}
& \quad \|T(\bm{z},v)-T(\bm{\bar{z}},\bar{v})\|_{\frac{\alpha}{\gamma}}^{2}  \\
&\leq \|(\bm{z},v)-(\bm{\bar{z}},\bar{v})\|_{\frac{\alpha}{\gamma}}^{2}
- \frac{1-\lambda-\frac{\alpha\beta}{2}}{\lambda}
\|(\emph{I}-T)(\bm{z},v)-(\emph{I}-T)(\bm{\bar{z}},\bar{v})\|_{\frac{\alpha}{\gamma}}^{2}.
\end{aligned}
\end{equation}
Let
\[
\theta=\frac{\lambda}{1-\frac{\alpha\beta}{2}}.
\]
Since $\lambda\in(0,1-\frac{\alpha\beta}{2})$, we have $\theta\in(0,1)$. Moreover,
\[
\frac{1-\theta}{\theta}
=
\frac{1-\frac{\alpha\beta}{2}-\lambda}{\lambda}
=
\frac{1-\lambda-\frac{\alpha\beta}{2}}{\lambda}.
\]
Therefore,
\begin{equation}
\begin{aligned}
& \quad \|T(\bm{z},v)-T(\bm{\bar{z}},\bar{v})\|_{\frac{\alpha}{\gamma}}^{2}  \\
&\leq \|(\bm{z},v)-(\bm{\bar{z}},\bar{v})\|_{\frac{\alpha}{\gamma}}^{2}
- \frac{1-\theta}{\theta}
\|(\emph{I}-T)(\bm{z},v)-(\emph{I}-T)(\bm{\bar{z}},\bar{v})\|_{\frac{\alpha}{\gamma}}^{2}.
\end{aligned}
\end{equation}
This proves that $T$ is $\theta$-averaged nonexpansive.

Furthermore, if one of the operators $A_i$ is uniformly monotone for some $i\in [1,n]$, then, by the definition of uniform monotonicity, the corresponding inequality among (\ref{the monotonicity of A1}), (\ref{the monotonicity of Ai}), and (\ref{the monotonicity of An}) can be strengthened by adding the term $\phi_{A_i}(\|x_i - \bar{x}_{i}\|)$ on the left-hand side. Repeating the above argument, after multiplying by $2\lambda$, we obtain the additional term $2\lambda\phi_{A_i}(\|x_i - \bar{x}_{i}\|)$ on the left-hand side of the descent inequality. Therefore, we get
\begin{equation*}
\begin{aligned}
& \quad \|T(\bm{z},v)-T(\bm{\bar{z}},\bar{v})\|_{\frac{\alpha}{\gamma}}^{2}  \\
&\leq \|(\bm{z},v)-(\bm{\bar{z}},\bar{v})\|_{\frac{\alpha}{\gamma}}^{2}
- \frac{1-\lambda-\frac{\alpha\beta}{2}}{\lambda}
\|(\emph{I}-T_1)\bm{z}-(\emph{I}-T_1)\bm{\bar{z}}\|^2 \\
&\quad
- \frac{1-\lambda}{\lambda}\frac{\alpha}{\gamma}
\|(\emph{I}-T_2)v-(\emph{I}-T_2)\bar{v}\|^2
-2\lambda\phi_{A_i}(\|x_i - \bar{x}_{i}\|),
\end{aligned}
\end{equation*}
which yields the inequality (\ref{the inequality of nonexpansion2}).
\end{proof}

\begin{theorem}\label{theorem--1}
Let the sequences $(\bm{z}^{k},v^k)_{k\in N}=(z_{1}^{k},\cdots,z_{n-1}^{k},v^k)_{k\in N}$, and $(\bm{x}^{k},y^k)_{k\in N}=(x_{1}^{k},\cdots,x_{n}^{k},y^k)_{k\in N}$ be generated by Algorithm \ref{alg1}. Assume that
\[
\alpha\in \left(0,\frac{2}{\beta}\right),\quad
\gamma\in \left(0,\frac{1-\frac{1}{2}\alpha\beta}{\alpha\|L\|^2}\right),\quad
\lambda\in \left(0,1-\frac{\alpha\beta}{2}\right),
\]
where $\beta=\max\{\beta_k\}_{k=1}^{n-1}$. Set
\[
\theta=\frac{\lambda}{1-\frac{\alpha\beta}{2}}.
\]
Then, the following assertions hold:

\emph{(i)} The sequence $(\bm{z}^{k},v^k)_{k\in N}$ converges weakly to $(\bm{\bar{z}},\bar{v})=(\bar{z}_1,\cdots,\bar{z}_{n-1},\bar{v})\in \emph{Fix}\; T$.

\emph{(ii)} The sequence $(\bm{x}^{k},y^k)_{k\in N}$ converges weakly to $(\bm{\bar{x}},L\bar{x})=(\bar{x},\cdots,\bar{x},L\bar{x})$, where $\bar{x}$ belongs to $\mathcal{P}$.

\emph{(iii)} For every $i\in [1,n]$, the sequences $(\gamma Lx_{i}^k-v^k)_{k\in N}$ converges weakly to $\gamma L\bar{x}-\bar{v}\in \mathcal{D}$.

\emph{(iv)} $
\| (\bm{z}^k, v^k) - (\bm{z}^{k+1}, v^{k+1}) \|_{\frac{\alpha}{\gamma}} = o\left(\frac{\theta}{\sqrt{\sigma_k}}\right),
$
where \( \sigma_k = \theta(1-\theta)(k+1) \).

\emph{(v)} If there exists an index $i\in [1,n]$ such that the operator $A_i$ is uniformly monotone,  then the sequence $\{x_{i}^{k}\}$ converges strongly to $\bar{x}$, where $\bar{x}$ belongs to $\mathcal{P}$.
\end{theorem}

\begin{proof}
(i)The sequence in (\ref{the main algorithm--1}) can be viewed as a fixed point iteration
\begin{equation}
(\bm{z}^{k+1}, v^{k+1}) = T(\bm{z}^{k}, v^k), \quad \forall k\geq 0.
\end{equation}
Note that $\lambda\in \left(0,1-\frac{\alpha\beta}{2}\right)$, $\alpha\in (0,\frac{2}{\beta})$, and $\gamma\in (0,\frac{1-\frac{1}{2}\alpha\beta}{\alpha\|L\|^2})$. By Lemma \ref{the lemma of nonexpansion}, it follows that the operator $T$ is $\theta$-averaged nonexpansive with
\[
\theta=\frac{\lambda}{1-\frac{\alpha\beta}{2}}\in(0,1).
\]
Furthermore, since $Z\neq \emptyset$ and by Lemma \ref{the lemma of two nonempty solution set}(i), we have $\text{Fix}\;T\neq \emptyset$. Therefore, by Theorem 5.15 of \cite{bauschkebook2017}, The sequence $(\bm{z}^{k},v^k)_{k\in N}$ converges weakly to $(\bm{\bar{z}},\bar{v})\in \text{Fix}\; T$, and $\lim_{k\rightarrow \infty}\|(\bm{z}^{k+1},v^{k+1})-(\bm{z}^{k},v^{k})\|_{\frac{\alpha}{\gamma}}=0$.

(ii) From (i), we know that the sequence $(\bm{z}^{k},v^k)_{k\in N}$ is bounded. By the Lipschitz continuity of $C$, the boundedness of $L$, and the nonexpansivity of the resolvents, it follows that $(\bm{x}^{k},y^k)_{k\in N}$ is also bounded. Moreover, since $\lim_{k\rightarrow \infty}\|(\bm{z}^{k+1},v^{k+1})-(\bm{z}^{k},v^{k})\|_{\frac{\alpha}{\gamma}}=0$, together with (\ref{the main algorithm--1}), we obtain
\begin{equation}
\lim_{k\rightarrow \infty}\|y^k-Lx_{n}^{k}\|=0 \; \text{and} \lim_{k\rightarrow \infty}\|x_{i+1}^{k}-x_{i}^{k}\|=0 ,\; \text{for all}\; i\in [1,n-1].
\end{equation}
By the definition of resolvent operator, (\ref{the main algorithm--2}) can be written as the following inclusion
\begin{equation}\label{the last inclusion}
\footnotesize
\left(
\begin{matrix}
x_{1}^{k}-x_{n}^{k}\\
x_{2}^{k}-x_{n}^{k}\\
\vdots\\
x_{n-1}^{k}-x_{n}^{k}\\
x_{1}^{k}-x_{n}^{k}+ \alpha \gamma L^{\ast}(Lx_{n}^{k}-y^k)+ \alpha\sum_{i=1}^{n-1}c_{i+1}^k\\
y^{k}- \alpha Lx_{n}^{k}
\end{matrix}
\right)
\in (R+S)
\left(
\begin{matrix}
z_{1}^{k}-x_{1}^{k}\\
(z_{2}^{k}-x_{2}^{k})-(z_{1}^{k}-x_{1}^{k}) + \alpha c_2^k\\
\vdots\\
(z_{n-1}^{k}-x_{n-1}^{k})-(z_{n-2}^{k}-x_{n-2}^{k}) + \alpha c_{n-1}^k\\
x_{n}^{k}\\
\gamma(L(x_{1}^{k}+x_{n}^{k})-y^{k})-v^{k},
\end{matrix}
\right)
\end{equation}
where $c_{i}^{k} = C_{i-1}(x_{i}^{k}) - C_{i-1}(x_{i-1}^{k})$, and the operators $R$, $S$ are defined by

$\bm{R}=\begin{bmatrix}
\left(\alpha A_{1}\right)^{-1}& 0 &0 &\cdots  & 0& 0      \\
0& \left(\alpha(A_{2} + C_{1})\right)^{-1}& 0 &\cdots & 0& 0 \\
\vdots & 0& \ddots &\vdots & \vdots& \vdots \\
0& \vdots &0 & \left(\alpha(A_{n-1} + C_{n-2})\right)^{-1} & 0& 0 \\
0 & 0 &\cdots & 0 & \alpha(A_{n}+C_{n-1}) & 0\\
0 & 0 &\cdots & 0 & 0 &   B^{-1}
\end{bmatrix}$
,

$\bm{S}=\begin{bmatrix}
0& 0 &\cdots &0  & -\text{Id}& 0      \\
0& 0& \cdots &0 & -\text{Id}& 0 \\
\vdots & \vdots& \ddots &\vdots & \vdots& \vdots \\
0& 0 &\cdots & 0 & -\text{Id}& 0 \\
\text{Id} & \text{Id} &\cdots & \text{Id} & 0 & \alpha L^{\ast}\\
0 & 0 &\cdots & 0 & -\alpha L & 0
\end{bmatrix}$
, respectively.\\
By the definition of the operators $A_{i},1\leq i \leq n$, $B$, $C_{i}, 1\leq i \leq n-1$, we know that $\bm{R}$ is maximally monotone. Since $\bm{S}$ is a skew symmetric linear operator, the operator $\bm{S}$ is monotone and Lipschitz. Thus, we know from Corollary 25.5 of \cite{bauschkebook2017} that $\bm{R}+\bm{S}$ is maximally monotone. Hence, its graph is closed in the weak-strong topology on $\mathcal{H}^{n}\times \mathcal{G}$.

Now, let $(\bm{\bar{x}},\bar{y})$ be a weak sequential cluster point of $(\bm{x}^k,y^k)_{k\in N}$, where $\bm{\bar{x}}=(\bar{x},\cdots,\bar{x})\in \mathcal{H}^{n}$ and $\bar{y}=L\bar{x}$. Taking the limit along a subsequence of $(\bm{x}^k,y^k)_{k\in N}$, (\ref{the last inclusion}) yields
\begin{equation}
\left\{
\begin{aligned}
& \bar{z}_{1}-\bar{x}\in \alpha A_{1}\bar{x} \\
& \bar{z}_{i}-\bar{z}_{i-1}\in \alpha(A_{i}+C_{i-1})(\bar{x}), \;\forall i\in [2,n-1]\\
& \bar{x}-\bar{z}_{n-1}- \alpha L^{\ast}(\gamma L\bar{x}-\bar{v})\in \alpha(A_{n}+ C_{n-1})(\bar{x})\\
& \gamma L\bar{x}-\bar{v}\in B(L\bar{x}),
\end{aligned}
\right.
\end{equation}
from which we deduce that $(\bar{x},\gamma L\bar{x}-\bar{v})\in Z$, where $\bar{x}=J_{\alpha A_1}(\bar{z}_{1})$. Hence, $(\bm{\bar{x}},\bar{y})$ is the unique weak sequential cluster point of $(\bm{x}^k,y^k)_{k\in N}$, which implies that the result holds.

(iii) From (i)-(ii), we can obtain that the sequences $(\gamma Lx_{i}^k-v^k)_{k\in N}$ converges weakly to $\gamma L\bar{x}-\bar{v}$ for every $i\in [1,n]$, and from $(\bar{x},\gamma L\bar{x}-\bar{v})\in Z$ , we know that $\gamma L\bar{x}-\bar{v}\in \mathcal{D}$.

(iv) According to Lemma \ref{the lemma of nonexpansion}, the operator \( T \) is \( \theta \)-averaged and nonexpansive, where
\[
\theta=\frac{\lambda}{1-\frac{\alpha\beta}{2}}.
\]
Applying Theorem 3.1 in~\cite{Matsushita2017} to the Krasnosel'skii--Mann iteration associated with \( T \), we obtain
\[
\| (\bm{z}^k, v^k) - (\bm{z}^{k+1}, v^{k+1}) \|_{\frac{\alpha}{\gamma}} = o\left(\frac{\theta}{\sqrt{\sigma_k}}\right),
\]
where \( \sigma_k = \theta(1-\theta)(k+1) \).

(v) Let $(\bm{\bar{z}},\bar{v})\in \emph{Fix}\;T$ be the weak limit obtained in (i). Applying inequality (\ref{the inequality of nonexpansion2}) with $(\bm{z},v)=(\bm{z}^{k},v^k)$ and $(\bm{\bar{z}},\bar{v})\in \emph{Fix}\;T$, we obtain
\[
\begin{aligned}
&\|(\bm{z}^{k+1},v^{k+1})-(\bm{\bar{z}},\bar{v})\|_{\frac{\alpha}{\gamma}}^2 \\
&\leq
\|(\bm{z}^{k},v^{k})-(\bm{\bar{z}},\bar{v})\|_{\frac{\alpha}{\gamma}}^2
-2\lambda\phi_{A_i}(\|x_i^k-\bar{x}\|).
\end{aligned}
\]
Hence,
\[
2\lambda\sum_{k=0}^{+\infty}\phi_{A_i}(\|x_i^k-\bar{x}\|)
\leq
\|(\bm{z}^{0},v^{0})-(\bm{\bar{z}},\bar{v})\|_{\frac{\alpha}{\gamma}}^2
<+\infty.
\]
Therefore,
\[
\phi_{A_i}(\|x_i^k-\bar{x}\|)\rightarrow 0.
\]
Since $\phi_{A_i}$ is increasing and vanishes only at $0$, it follows that
\[
\|x_i^k-\bar{x}\|\rightarrow 0.
\]
Thus, the sequence $\{x_{i}^{k}\}$ converges strongly to $\bar{x}$, where $\bar{x}$ belongs to $\mathcal{P}$.
\end{proof}

In the following, we consider a more general composite monotone inclusion than (\ref{problem1}) as follows:

\begin{equation}\label{primal problem-1}
\textrm{find}\quad x\in \mathcal{H}\quad\textrm{such that}\quad 0\in \sum_{i=1}^{n}A_{i}x+\sum_{j=1}^{m}L_{j}^{\ast}B_{j}L_{j}x+\sum_{k=1}^{n-1}C_{k}x,
\end{equation}
together with its dual
\begin{equation}\label{dual problem-1}
\begin{aligned}
\textrm{find}\quad & (u_{1},\cdots,u_{m})\in \mathcal{G}_{1}\times\cdots\times\mathcal{G}_{m} \quad  \\
& \textrm{such that}\quad
(\exists \;x\in\mathcal{H})
\left\{
\begin{aligned}
&-\sum_{j=1}^{m}L_{j}^{\ast}u_{j}\in \sum_{i=1}^{n}A_{i}x+\sum_{k=1}^{n-1}C_{k}x \\
&u_{j}\in B_{j}(L_{j}x),\quad j=1,\cdots,m,
\end{aligned}
\right.\\
\end{aligned}
\end{equation}
where for any $j=1, \cdots, m$, $B_{j}:\mathcal{G}_{j}\rightarrow 2^{\mathcal{G}_{j}}$ is maximally monotone on Hilbert space $\mathcal{G}_{j}$, and $L_{j}:\mathcal{H}\rightarrow \mathcal{G}_{j}$ is bounded linear operator with adjoint operator $L_{j}^{*}$, $\{A_i\}_{i=1}^{n}$ and $\{C_k\}_{k=1}^{n-1}$ are the same as (\ref{problem1}).  With the light of a standard product space reformulation, (\ref{primal problem-1}) can be viewed as a special instance of (\ref{primal problem}), which leads to the form of Algorithm \ref{alg2} and corresponding convergence theorem.

\renewcommand{\algorithmicrequire}{\textbf{Input:}}
\renewcommand{\algorithmicensure}{\textbf{Output:}}
\begin{algorithm}[htb]
\caption{Primal-dual splitting algorithm for solving monotone inclusions (\ref{primal problem-1}) and (\ref{dual problem-1}).}
\label{alg2}
\begin{algorithmic}[1]
\REQUIRE
\vskip 2mm
Let $\lambda\in (0,1)$, $\alpha\in (0,\frac{2}{\beta})$, and $\gamma\in \left(0,\frac{1-\frac{1}{2}\alpha\beta}{\alpha \sum_{j=1}^{m}\|L_j\|^2} \right)$, where $\beta = \max\{\beta_k\}_{k=1}^{n-1}$. For any given $\bm{z}^0=(z_{1}^{0},\cdots,z_{n-1}^{0})\in \mathcal{H}^{n-1}$, $\bm{v}^{0}=(v_{1}^{0},\cdots,v_{m}^{0})\in \mathcal{G}_{1}\times\cdots\times\mathcal{G}_{m}$ and for every $k\geq 0$, iterate
\begin{equation}\label{the main algorithm--3}
\dbinom{\bm{z}^{k+1}}{\bm{v}^{k+1}}= \dbinom{\bm{z}^{k}}{\bm{v}^{k}} + \lambda\left(
\begin{matrix}
x_{2}^{k}-x_{1}^{k}\\
x_{3}^{k}-x_{2}^{k}\\
\vdots\\
x_{n}^{k}-x_{n-1}^{k}\\
\gamma(y_{1}^{k}-L_{1}x_{n}^{k})\\
\vdots\\
\gamma(y_{m}^{k}-L_{m}x_{n}^{k})
\end{matrix}
\right)
\end{equation}
with
\begin{equation}\label{the main algorithm--4}
\left\{
\begin{aligned}
& x_{1}^{k}=J_{ \alpha A_{1}}(z_{1}^{k}),\\
& x_{i}^{k}=J_{ \alpha A_{i}}(z_{i}^{k} + x_{i-1}^{k} - z_{i-1}^{k} - \alpha C_{i-1}x_{i-1}^{k}), \forall i\in [2,n-1],\\
& x_{n}^{k}=J_{ \alpha A_{n}}(x_{1}^{k} + x_{n-1}^{k} - z_{n-1}^{k}- \alpha \sum_{j=1}^{m}L_{j}^{\ast}(\gamma L_{j}x_{1}^{k}-v_{j}^{k})- \alpha C_{n-1}x_{n-1}^{k}),\\
& y_{j}^{k}=J_{\frac{B_{j}}{\gamma}}\left(L_{j}(x_{1}^{k} + x_{n}^{k}) - \frac{v_{j}^{k}}{\gamma}\right), \forall j\in [1,m].
\end{aligned}
\right.
\end{equation}
Stop when a given stopping criterion is met.
\end{algorithmic}
\end{algorithm}

\begin{theorem}\label{theorem2}
Let the sequences $(\bm{z}^{k},\bm{v}^k)_{k\in N}=(z_{1}^{k},\cdots,z_{n-1}^{k},v_{1}^k,\cdots,v_{m}^{k})_{k\in N}$, and $(\bm{x}^{k},\bm{y}^k)_{k\in N}=(x_{1}^{k},\cdots,x_{n}^{k},y_{1}^k,\cdots,y_{m}^k)_{k\in N}$ be generated by Algorithm \ref{alg2}. Assume that
\[
\alpha\in \left(0,\frac{2}{\beta}\right),\quad
\gamma\in \left(0,\frac{1-\frac{1}{2}\alpha\beta}{\alpha\sum_{j=1}^{m}\|L_j\|^2}\right),\quad
\lambda\in \left(0,1-\frac{\alpha\beta}{2}\right),
\]
where $\beta=\max\{\beta_k\}_{k=1}^{n-1}$. Set
\[
\theta=\frac{\lambda}{1-\frac{\alpha\beta}{2}}.
\]
Then, the following assertions hold:

\emph{(i)} The sequence $(\bm{z}^{k},\bm{v}^k)_{k\in N}$ converges weakly to $(\bm{\bar{z}},\bar{\bm{v}})=(\bar{z}_1,\cdots,\bar{z}_{n-1},\bar{v}_1,\cdots,\bar{v}_{m})\in \mathcal{H}^{n-1}\times\mathcal{G}_{1}\times\cdots\times\mathcal{G}_{m}$.

\emph{(ii)} The sequence $(\bm{x}^{k},\bm{y}^k)_{k\in N}$ converges weakly to $(\bm{\bar{x}},\bm{L}\bar{x})=(\bar{x},\cdots,\bar{x},L_1\bar{x},\cdots,L_m\bar{x})$, where $\bar{x}$ is a solution of the primal inclusion (\ref{primal problem-1}).

\emph{(iii)} For every $i\in [1,n]$, the sequences $(\gamma L_1 x_{i}^k-v_{1}^k,\cdots,\gamma L_m x_{i}^k-v_{m}^k)_{k\in N}$ converges weakly to $(\gamma L_1 \bar{x}-\bar{v}_{1},\cdots,\gamma L_m \bar{x}-\bar{v}_{m})$, which solves the dual inclusion (\ref{dual problem-1}).

\emph{(iv)} $
\| (\bm{z}^k, \bm{v}^k) - (\bm{z}^{k+1}, \bm{v}^{k+1}) \|_{\frac{\alpha}{\gamma}} = o\left(\frac{\theta}{\sqrt{\sigma_k}}\right),
$
where \( \sigma_k = \theta(1-\theta)(k+1) \).

\emph{(v)} If there exists an index $i\in [1,n]$ such that the operator $A_i$ is uniformly monotone,  then the sequence $\{x_{i}^{k}\}$ converges strongly to $\bar{x}$, where $\bar{x}$ belongs to $\mathcal{P}$.
\end{theorem}

\begin{proof}
Define the product space $\bm{G}=\mathcal{G}_{1}\times\cdots\times\mathcal{G}_{m}$, which equipped with the inner product and the associated norm as follows
$$
\langle \bm{y}, \bm{z} \rangle_{\bm{G}} = \sum_{j=1}^{m}\langle y_j, z_j \rangle, \textrm{ and } \|\bm{y}\|_{\bm{G}} = \sqrt{\sum_{j=1}^{m}\|y_j\|^2},
$$
where $\bm{y}=(y_1, \cdots, y_m)\in \bm{G}$ and $\bm{z} = (z_1, \cdots, z_m)\in \bm{G}$.
Let $\bm{B}:\bm{G}\rightarrow 2^{\bm{G}}$ is $\bm{y} \mapsto \times_{j=1}^{m}B_{j}y_j$, and the bounded linear operator $\bm{L}:\mathcal{H}\rightarrow \bm{G} : x \mapsto (L_1 x, \cdots, L_m x)$.  Moreover, $\|\bm{L}\|^2=\sum_{j=1}^{m}\|L_{j}\|^2$ and its adjoint operator is $\bm{L}^{*}:\bm{G}\rightarrow \mathcal{H}: \bm{y}\mapsto \sum_{j=1}^{m}L_{j}^{*}y_j$, for any $\bm{y} = (y_1, \cdots, y_m)\in \bm{G}$. Therefore, the monotone inclusion (\ref{primal problem-1}) is equivalent to
$$
0\in \sum_{i=1}^{n}A_i x + \bm{L}^{*}\bm{B}\bm{L}x + \sum_{k=1}^{n-1}C_k x.
$$
Therefore, the results follow by replacing $B$ with the operator $\bm{B}$ and $L$ with the linear operator $\bm{L}$ in Theorem \ref{theorem--1}. In particular, the condition
\[
\gamma\in \left(0,\frac{1-\frac{1}{2}\alpha\beta}{\alpha\sum_{j=1}^{m}\|L_j\|^2}\right)
\]
is exactly the condition
\[
\gamma\in \left(0,\frac{1-\frac{1}{2}\alpha\beta}{\alpha\|\bm{L}\|^2}\right)
\]
in the product space. Moreover, the averagedness parameter is still
\[
\theta=\frac{\lambda}{1-\frac{\alpha\beta}{2}}.
\]
Hence, all assertions follow from Theorem \ref{theorem--1}.

\end{proof}

\subsection{Applications to convex minimization problems}

In this subsection, we consider the following convex composite minimization problem:
\begin{equation}\label{convex-composite}
\min_{x\in \mathcal{H}} \sum_{i=1}^{n}g_i (x) + \sum_{k=1}^{n-1}f_{k}(x) + \sum_{j=1}^{m}h_{j}(L_j x),
\end{equation}
where $g_1, \cdots, g_n:\mathcal{H}\rightarrow (-\infty, +\infty]$ are proper, lsc, and convex functions. For each $j\in \{1, \cdots, m\}$, $h_j:\mathcal{G}_{j}\rightarrow (-\infty, +\infty]$ is a proper, lsc, and convex functions, and for each $k\in \{1, \cdots n-1\}$, $f_k:\mathcal{H}\rightarrow (-\infty,+\infty)$ is convex and differentiable with $\beta_k$-Lipshictz continuous gradient.  For each $j\in \{1, \cdots, m\}$, $L_j:\mathcal{H}\rightarrow \mathcal{G_j}$ is a nonzero bounded linear operator. By the Fenchel-Rockafellar duality theorem, the dual problem of (\ref{convex-composite}) is
\begin{equation}\label{convex-composite-dual}
\max_{y_1 \in \mathcal{G}_1, \cdots  y_m \in \mathcal{G}_m} - \left (\sum_{k=1}^{n-1}f_{k} + \sum_{i=1}^{n}g_i \right )^{*} \left(-\sum_{j=1}^{m}L_{j}^{*}y_j \right ) - \sum_{j=1}^{m}h_{j}^{*}(y_j),
\end{equation}

\begin{theorem}\label{main-convex-minization3}
Consider the convex minimization problem (\ref{convex-composite}) and its dual (\ref{convex-composite-dual}). Let $\lambda\in (0,1)$, $\alpha\in (0,\frac{2}{\beta})$, and $\gamma\in \left(0,\frac{1-\frac{1}{2}\alpha\beta}{\alpha \sum_{j=1}^{m}\|L_j\|^2} \right)$, where $\beta = \max\{\beta_k\}_{k=1}^{n-1}$. For any given $\bm{z}^0=(z_{1}^{0},\cdots,z_{n-1}^{0})\in \mathcal{H}^{n-1}$, $\bm{v}^{0}=(v_{1}^{0},\cdots,v_{m}^{0})\in \mathcal{G}_{1}\times\cdots\times\mathcal{G}_{m}$ and for every $k\geq 0$, iterate
\begin{equation}\label{the main algorithm--32}
\dbinom{\bm{z}^{k+1}}{\bm{v}^{k+1}}= \dbinom{\bm{z}^{k}}{\bm{v}^{k}} + \lambda\left(
\begin{matrix}
x_{2}^{k}-x_{1}^{k}\\
x_{3}^{k}-x_{2}^{k}\\
\vdots\\
x_{n}^{k}-x_{n-1}^{k}\\
\gamma(y_{1}^{k}-L_{1}x_{n}^{k})\\
\vdots\\
\gamma(y_{m}^{k}-L_{m}x_{n}^{k})
\end{matrix}
\right)
\end{equation}
with
\begin{equation}\label{the main algorithm--42}
\left\{
\begin{aligned}
& x_{1}^{k}= prox_{ \alpha g_{1}}(z_{1}^{k}),\\
& x_{i}^{k}= prox_{ \alpha g_{i}}(z_{i}^{k} + x_{i-1}^{k} - z_{i-1}^{k} - \alpha \nabla f_{i-1}x_{i-1}^{k}), \;\forall i\in [2,n-1],\\
& x_{n}^{k}= prox_{ \alpha g_{n}}(x_{1}^{k} + x_{n-1}^{k} - z_{n-1}^{k}- \alpha \sum_{j=1}^{m}L_{j}^{\ast}(\gamma L_{j}x_{1}^{k}-v_{j}^{k})- \alpha \nabla f_{n-1}x_{n-1}^{k}),\\
& y_{j}^{k}= prox_{\frac{h_{j}}{\gamma}}\left(L_{j}(x_{1}^{k} + x_{n}^{k}) - \frac{v_{j}^{k}}{\gamma}\right), \forall j\in [1,m].
\end{aligned}
\right.
\end{equation}
 Then, the following assertions hold:

\emph{(i)} The sequence $(\bm{z}^{k},\bm{v}^k)_{k\in N}$ converges weakly to $(\bm{\bar{z}},\bar{\bm{v}})=(\bar{z}_1,\cdots,\bar{z}_{n-1},\bar{v}_1,\cdots,\bar{v}_{m})\in \mathcal{H}^{n-1}\times\mathcal{G}_{1}\times\cdots\times\mathcal{G}_{m}$.

\emph{(ii)} The sequence $(\bm{x}^{k},\bm{y}^k)_{k\in N}$ converges weakly to $(\bm{\bar{x}},\bm{L}\bar{x})=(\bar{x},\cdots,\bar{x},L_1\bar{x},\cdots,L_m\bar{x})$, where $\bar{x}$ is a solution of (\ref{convex-composite}).

\emph{(iii)} For every $i\in [1,n]$, the sequences $(\gamma L_1x_{i}^k-v_{1}^k,\cdots,\gamma L_m x_{i}^k-v_{m}^k)_{k\in N}$ converges weakly to $(\gamma L_1 \bar{x}-\bar{v}_{1},\cdots,\gamma L_m \bar{x}-\bar{v}_{m})$, which solves the dual problem (\ref{convex-composite-dual}).

\emph{(iv)} $
\| (\bm{z}^k, \bm{v}^k) - (\bm{z}^{k+1}, \bm{v}^{k+1}) \|_{\frac{\alpha}{\gamma}} = o\left(\frac{\lambda}{\sqrt{\sigma_k}}\right),
$
where \( \sigma_k = \lambda(1-\lambda)(k+1) \).

\emph{(v)} If there exists an index $i\in [1,n]$ such that $\partial g_i$ is uniformly monotone, then the sequence $\{x_{i}^{k}\}$ converges strongly to $\bar{x}$, where $\bar{x}$ is a solution of (\ref{convex-composite}).
\end{theorem}

\begin{proof}
By the first order optimality condition, (\ref{convex-composite}) can be recast as (\ref{primal problem-1}) with
$$
A_i = \partial g_i, C_k = \nabla f_{k}, \textrm{ and } B_j = \partial h_{j}.
$$
Therefore, the convergence theorem is a direct application of Theorem \ref{theorem2}.
\end{proof}

\begin{remark}
In \cite{Tang2022JSC}, the authors proposed a primal-dual splitting algorithm for solving \eqref{convex-composite} and \eqref{convex-composite-dual}. The method treats the finite sum of convex differentiable functions $\sum_{k=1}^{n-1}f_{k}$ as a whole and adopts a parallelizable primal-dual scheme,
in which local proximal steps are computed independently across blocks and then aggregated
through a global synchronization using traditional product-space techniques.

In contrast, the proposed algorithm~\eqref{the main algorithm--32}--\eqref{the main algorithm--42} employs a
sequential chain-structured splitting, updating variables in a dependent order with reduced
dual storage and without requiring global reduction. As a result, it is both simpler and more memory-efficient. To illustrate this, we compare the proposed algorithm with \cite{Tang2022JSC} on a concise image deblurring problem in the next section.
\end{remark}


\section{Numerical experiments}

In this section, we present two numerical experiments to evaluate the proposed algorithms. The first addresses a constrained image deblurring problem using a nuclear norm–total variation model with box constraints. The second tackles image denoising via the MC-TV model, where a convex–nonconvex reformulation yields a convex problem under certain conditions. All experiments are carried out on a laptop running Windows 7 with MATLAB R2016a, equipped with an Intel Core i7-6700 processor (3.40 GHz) and 4 GB of RAM. We use the peak signal-to-noise ratio (PSNR) and the structural similarity index (SSIM)~\cite{wangzhou} to evaluate the quality of the restored images. They are defined as
\begin{equation*}\label{N3}
\mathrm{PSNR} = 20\log_{10}\!\left(\frac{255\sqrt{mn}}{\|x-\widetilde{x}\|_{F}}\right),
\end{equation*}
and
\begin{equation*}\label{N4}
\mathrm{SSIM} = \frac{(2\mu_{1}\mu_{2}+c_{1})(2\sigma_{12}+c_{2})}
{(\mu_{1}^{2}+\mu_{2}^{2}+c_{1})(\sigma_{1}^{2}+\sigma_{2}^{2}+c_{2})},
\end{equation*}
where $\|\cdot\|_{F}$ denotes the Frobenius norm, $x\in \mathbb{R}^{m\times n}$ is the original image, and $\widetilde{x}\in \mathbb{R}^{m\times n}$ is the restored image. The constants $c_{1}>0$ and $c_{2}>0$ are small positive numbers for stability. The quantities $\mu_{1}$ and $\mu_{2}$ denote the mean values of $x$ and $\widetilde{x}$, respectively; $\sigma_{1}^{2}$ and $\sigma_{2}^{2}$ are their corresponding variances; and $\sigma_{12}$ is the covariance between $x$ and $\widetilde{x}$. We use the following condition as the stopping criterion:
$$
\frac{\| \bm{z}^{k+1} - \bm{z}^{k} \|}{\|\bm{z}^{k}\|} \leq 10^{-5}.
$$

The test images used in the experiments are presented in Figure \ref{test-images}.
\begin{figure}[H]
     \setlength{\abovecaptionskip}{-3pt}
  \centering
  \subfigcapskip=-10pt
  \makeatletter
    \subfigure[]{
        \scalebox{0.35}{\includegraphics{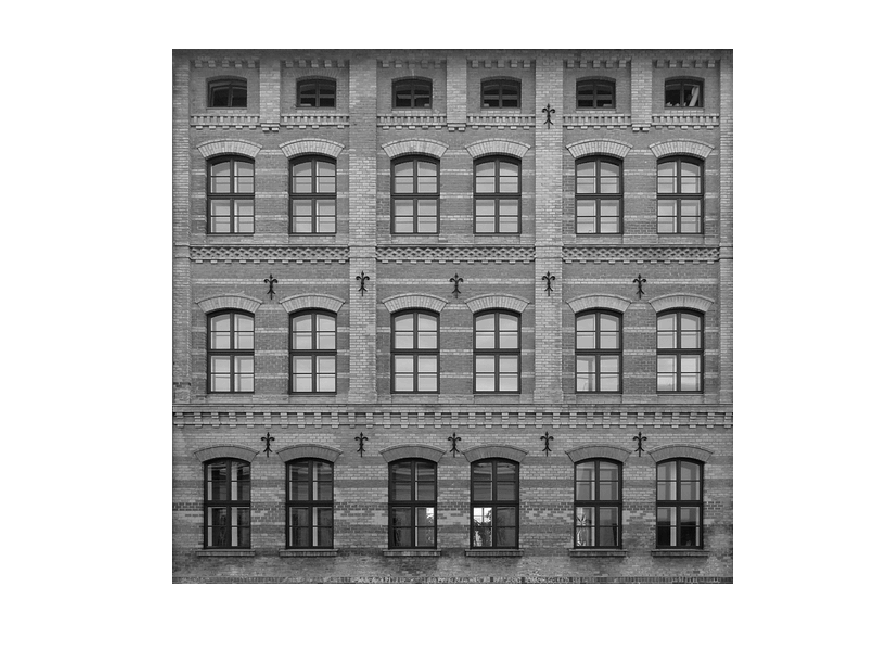}}
    }\hspace{-35pt}
    \subfigure[]{
        \scalebox{0.35}{\includegraphics{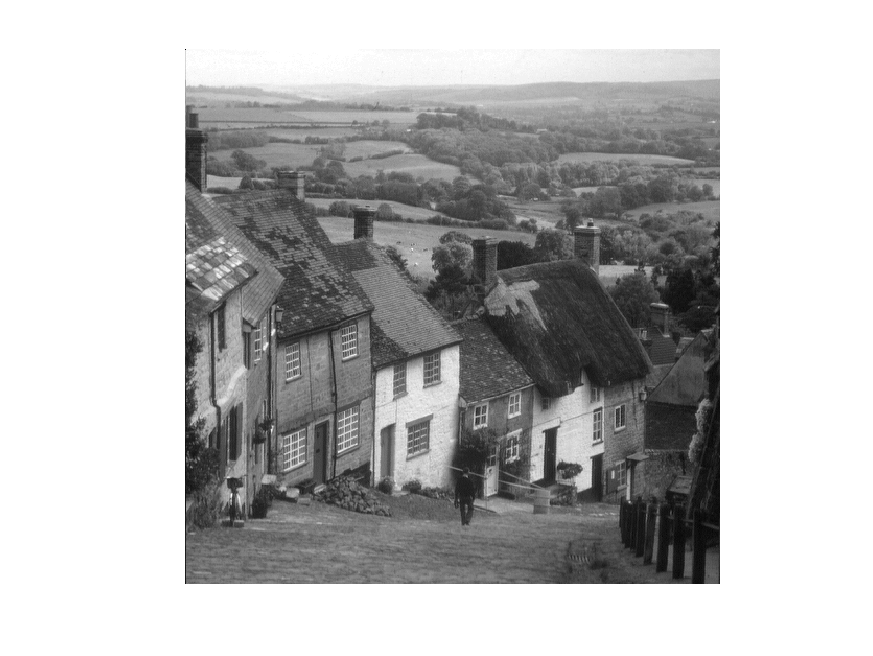}}
    }\hspace{-35pt}
        \subfigure[]{
        \scalebox{0.35}{\includegraphics{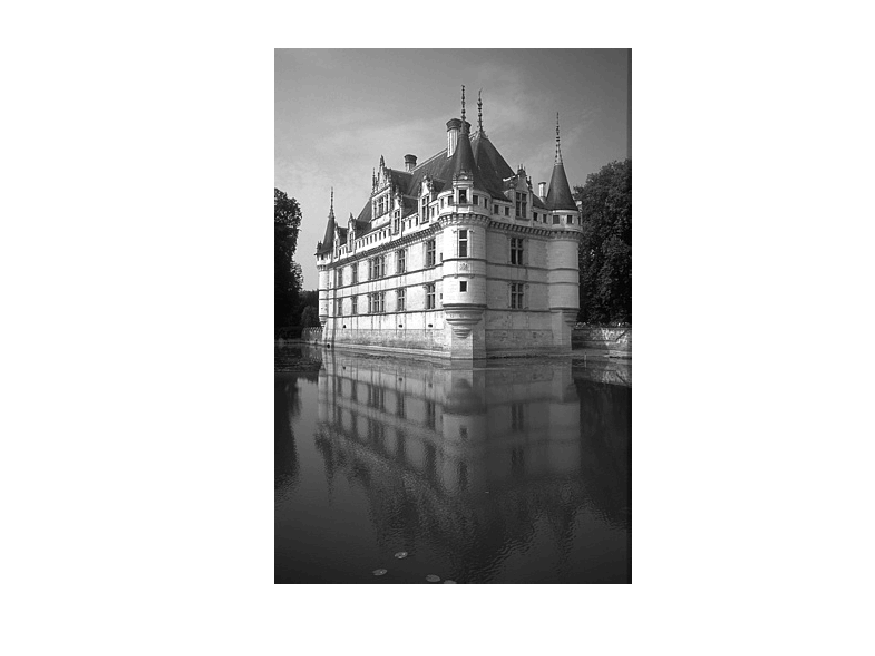}}
    }
    \caption{Test images (pixel intensity range:$0$-$255$). (a) Building, size $493\times 517$. (b) Goldhill, size $512\times 512$. (c) Castle, size $481\times 321$.}
    \label{test-images}
\end{figure}

\subsection{Image deblurring task}

In this subsection, we assess the performance of the proposed algorithms through numerical experiments on a standard image deblurring task. In detail, we consider the following constrained nuclear norm--total variation joint regularization model, which was proposed in \cite{Tang2022JSC}:
\begin{equation}\label{constrained-l2-nuclear-tv}
\begin{aligned}
\min_{x} \ & \tfrac{1}{2}\|Ax-b\|^{2}  + \mu_1 \|x\|_{TV} + \mu_2 \|x\|_{*}, \\
\text{s.t.} \ & x \in C.
\end{aligned}
\end{equation}
The set \(C\) represents the box constraint corresponding to valid pixel intensity bounds, typically defined as
\[
C = \{x \in \mathbb{R}^{m\times n} \mid 0 \le x_{ij} \le 255, 1\leq i \leq m, 1\leq j \leq n\},
\]
with \(0\) and \(255\) being the lower and upper intensity limits, respectively.
This constraint ensures that the reconstructed image remains physically meaningful.
Introducing the indicator function $\delta_{C}$, problem~\eqref{constrained-l2-nuclear-tv} can be equivalently written as
\begin{equation}\label{constrained-l2-nuclear-tv2}
\min_{x} \ \tfrac{1}{2}\|Ax-b\|^{2} + \delta_{C}(x) + \mu_1 \|x\|_{TV} + \mu_2 \|x\|_{*} .
\end{equation}

The total variation term can be expressed as
\[
\|x\|_{TV} = \varphi(Lx),
\]
where $\varphi$ is a convex function and $L$ is a first-order difference operator (see, e.g., \cite{Micchelliandshen2011}). Thus, \eqref{constrained-l2-nuclear-tv2} is a special case of~\eqref{convex-composite} with
\[
n=2, m=1, g_{1}(x) = \delta_{C}(x), g_{2}(x) = \mu_2 \|x\|_{*}, h_{1}(x) = \mu_1 \varphi(x), f_{1}(x)=\tfrac{1}{2}\|Ax-b\|^{2}.
\]
Therefore, the iterative scheme~\eqref{the main algorithm--32} can be directly applied to solve~\eqref{constrained-l2-nuclear-tv2}. In particular, the Lipschitz constant of $\nabla f_{1}$ is equal to $1$, i.e., $\beta =1$.

Each image is synthetically degraded by first applying a blur kernel, followed by the addition of zero-mean Gaussian noise with standard deviation $\sigma_{g}$. Two types of blur kernels are considered: a $9 \times 9$ uniform blur and a $9 \times 9$ Gaussian blur with a standard deviation of $2$.

\subsubsection{Parameter Sensitivity Analysis}

To systematically investigate the influence of the key parameters of the proposed algorithm, i.e., $\lambda$, $\alpha$, and $\gamma$, we evaluate the reconstruction quality (PSNR), convergence behavior (number of iterations), and computational cost (CPU time in seconds) under different parameter configurations. Specifically, $\lambda \in \{0.1, 0.3, 0.5, 0.7, 0.8, 0.9\}$, $\alpha \in \{0.1, 0.5, 1.0, 1.5, 1.9\}$, and $\gamma$ is normalized as $\gamma/\gamma_{\max}(\alpha) \in [0.1, 0.9]$, where $\gamma_{\max}(\alpha) = \frac{1-0.5\alpha}{8\alpha}$. We select the Building image in Figure \ref{test-images} as the test image, and add a $9 \times 9$ uniform blur and zero-mean Gaussian noise with standard deviation $\sigma_{g}=10$ to the image. The corresponding results are shown in Figures \ref{parameter-psnr}, \ref{parameter-iter}, and \ref{parameter-CPU}, respectively.

\begin{figure}[H]
     \setlength{\abovecaptionskip}{0pt}
  \centering
  \makeatletter
    \subfigure[$\lambda = 0.1$]{
        \scalebox{0.3}{\includegraphics{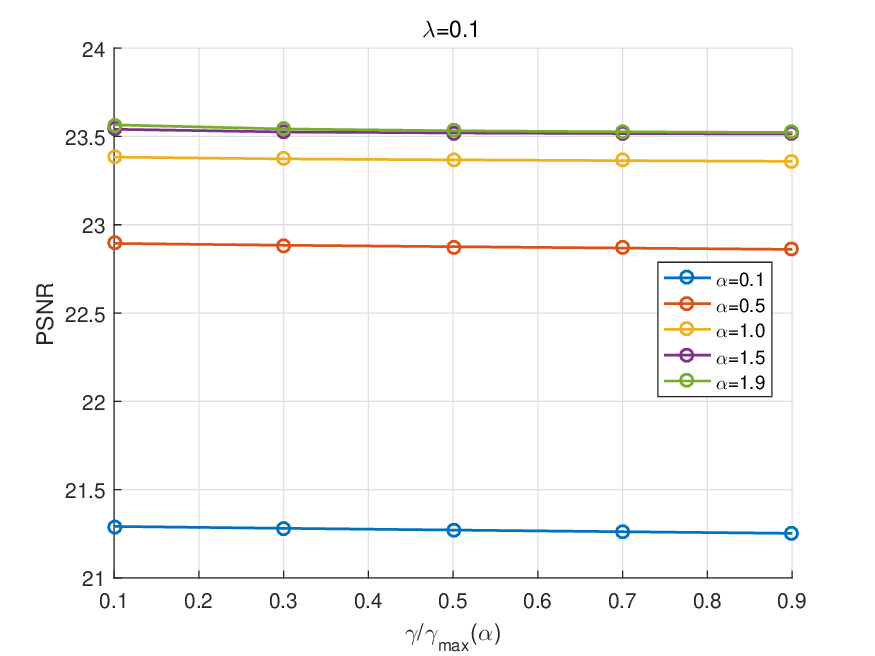}}
    } \hspace{-20pt}
      \subfigure[$\lambda = 0.3$]{
        \scalebox{0.3}{\includegraphics{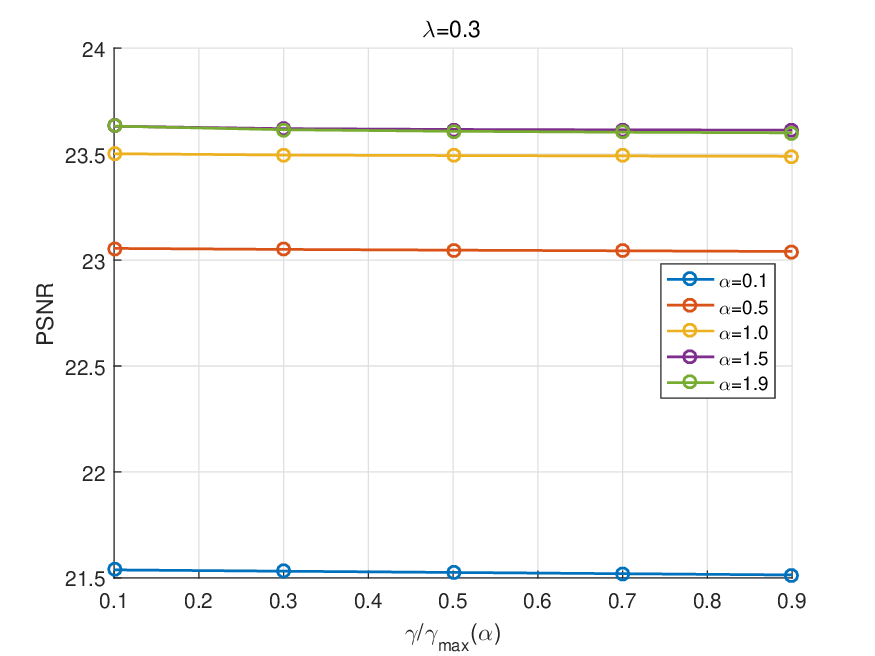}}
    }\hspace{-20pt}
    \subfigure[$\lambda = 0.5$]{
        \scalebox{0.3}{\includegraphics{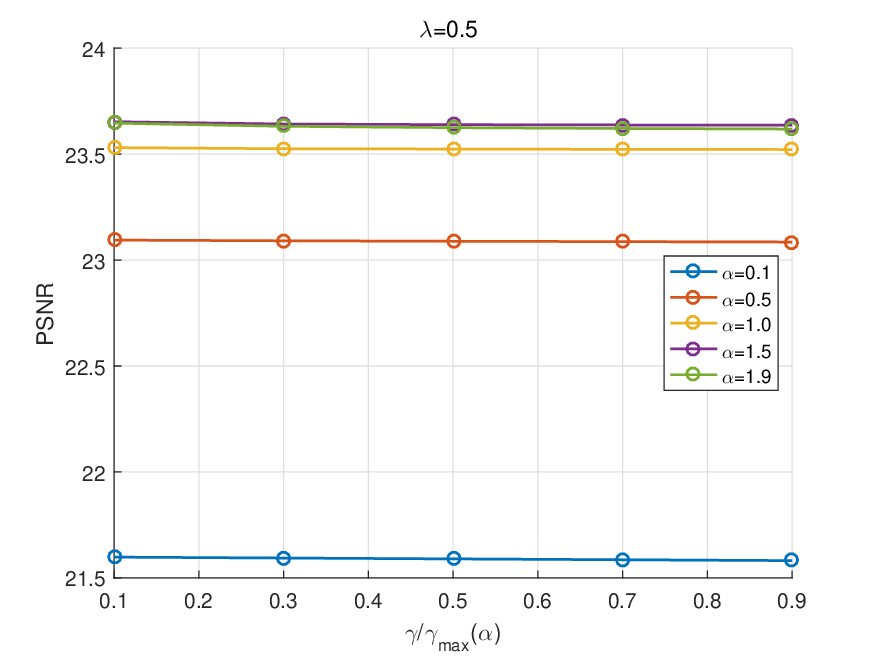}}
    }  \\
    \subfigure[$\lambda = 0.7$]{
        \scalebox{0.3}{\includegraphics{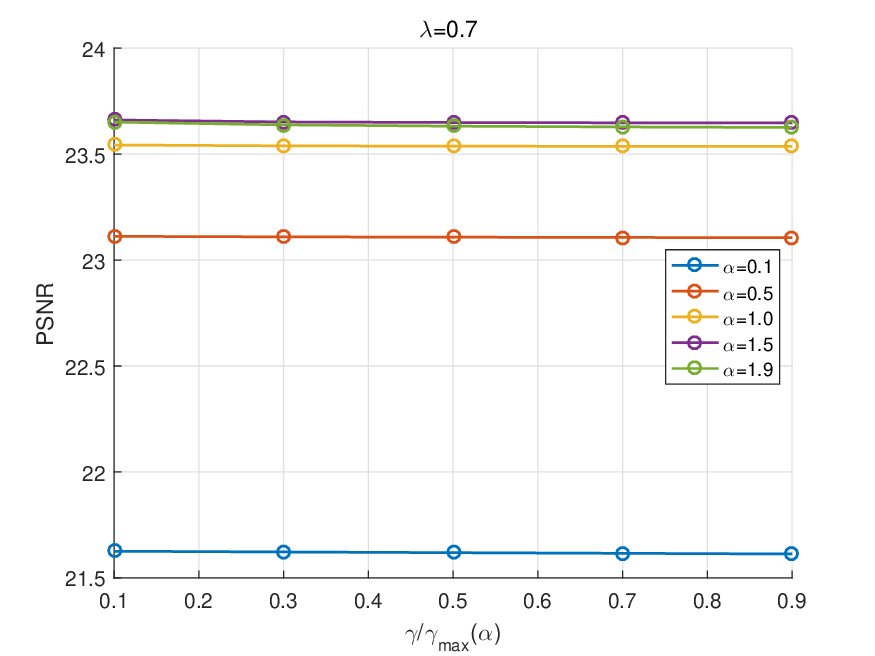}}
    } \hspace{-20pt}
      \subfigure[$\lambda = 0.8$]{
        \scalebox{0.3}{\includegraphics{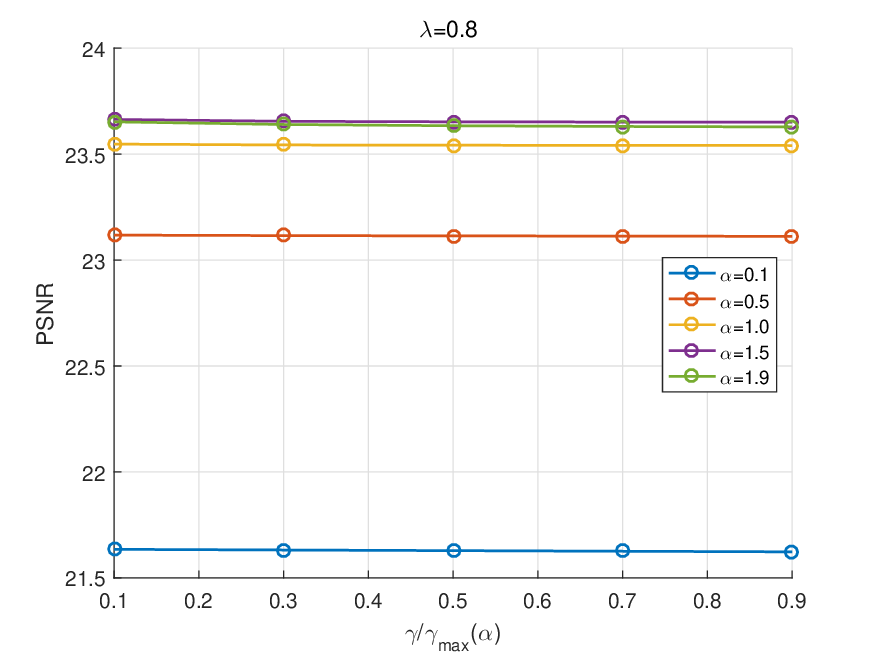}}
    }\hspace{-20pt}
     \subfigure[$\lambda = 0.9$]{
        \scalebox{0.3}{\includegraphics{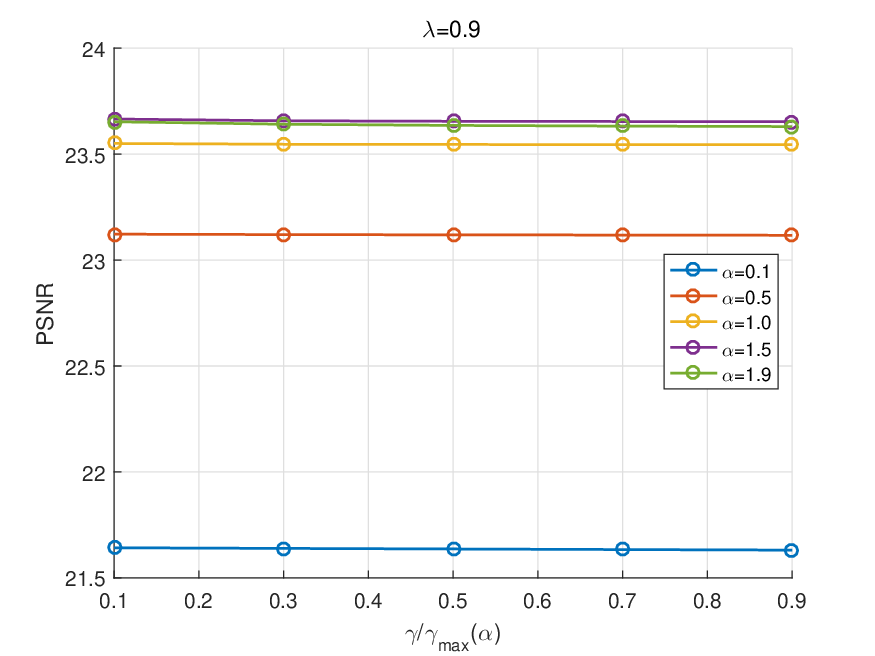}}
    }\\
    \caption{Influence of $\alpha$ and $\gamma$ on PSNR for different $\lambda$ settings.}
    \label{parameter-psnr}
\end{figure}

\begin{figure}[H]
     \setlength{\abovecaptionskip}{0pt}
  \centering
  \makeatletter
    \subfigure[$\lambda = 0.1$]{
        \scalebox{0.3}{\includegraphics{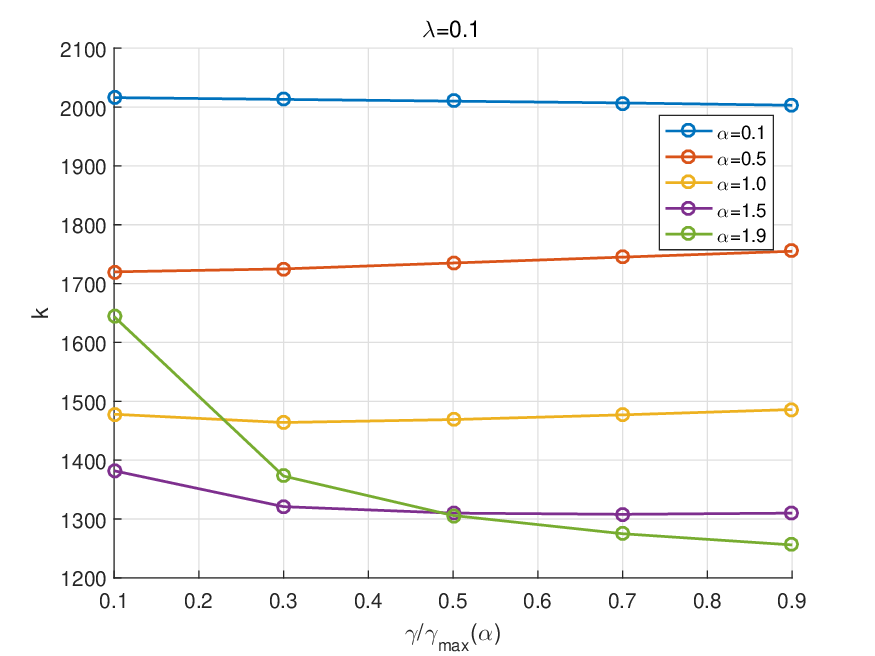}}
    } \hspace{-20pt}
      \subfigure[$\lambda = 0.3$]{
        \scalebox{0.3}{\includegraphics{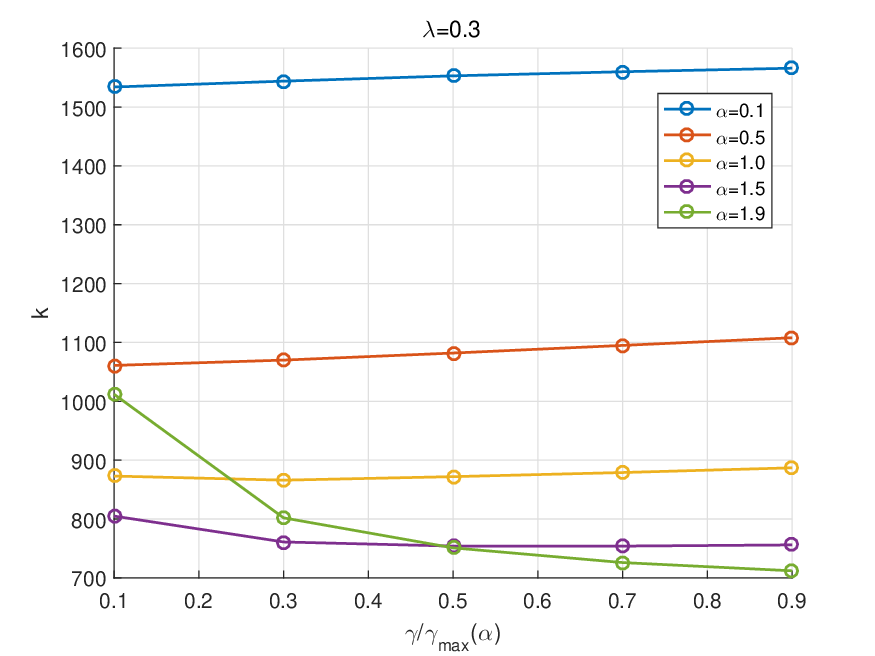}}
    }\hspace{-20pt}
    \subfigure[$\lambda = 0.5$]{
        \scalebox{0.3}{\includegraphics{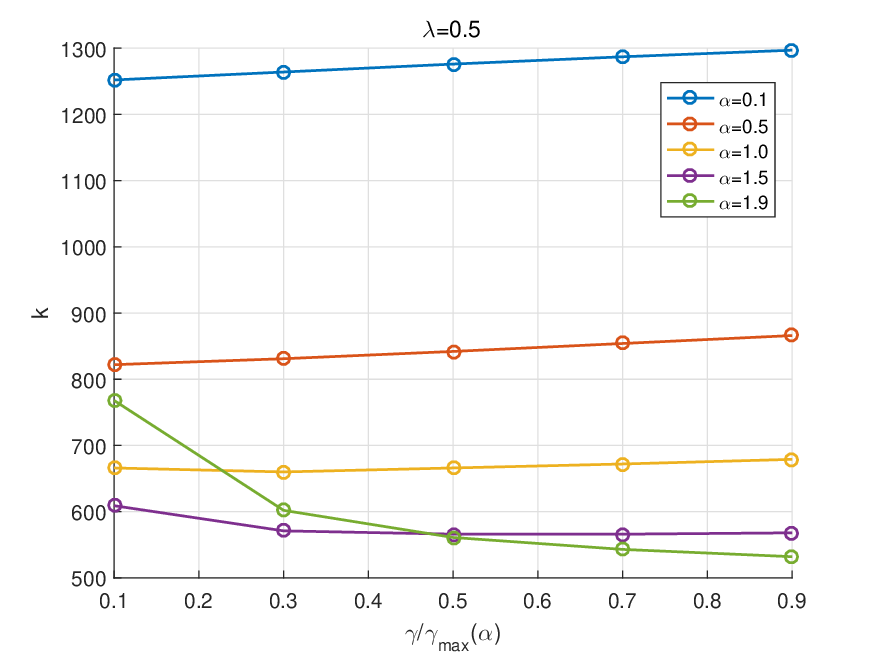}}
    }  \\
    \subfigure[$\lambda = 0.7$]{
        \scalebox{0.3}{\includegraphics{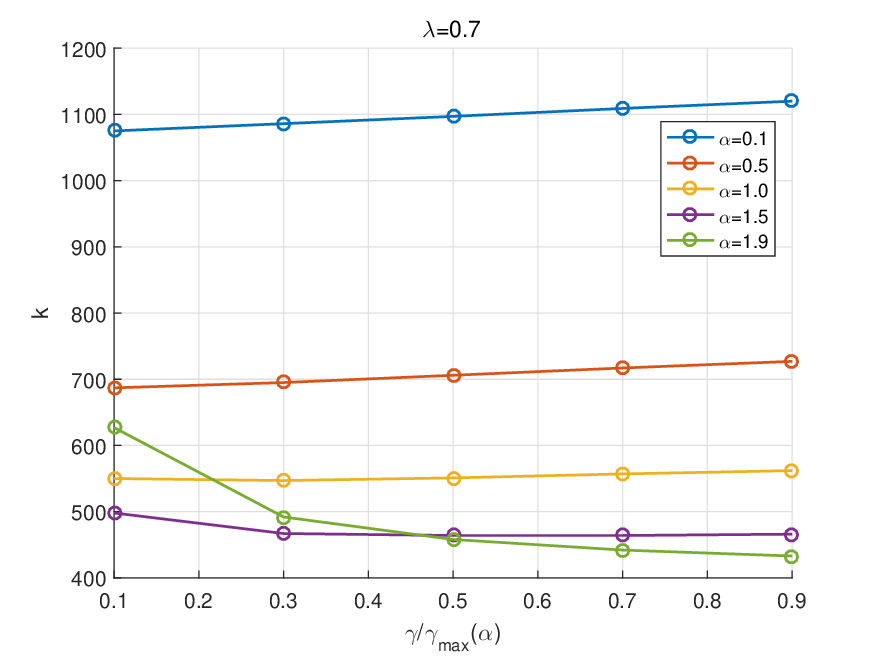}}
    } \hspace{-20pt}
      \subfigure[$\lambda = 0.8$]{
        \scalebox{0.3}{\includegraphics{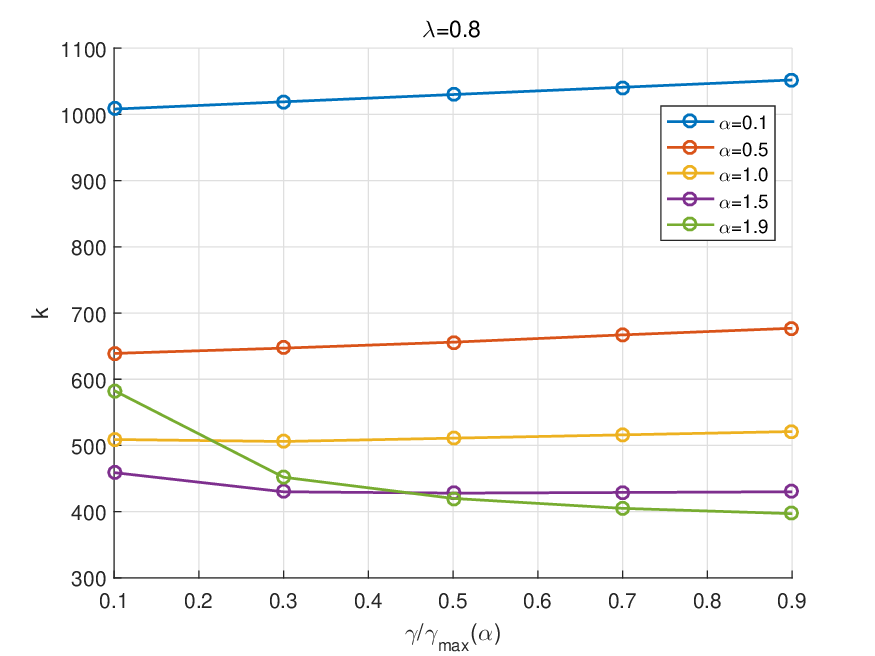}}
    }\hspace{-20pt}
     \subfigure[$\lambda = 0.9$]{
        \scalebox{0.3}{\includegraphics{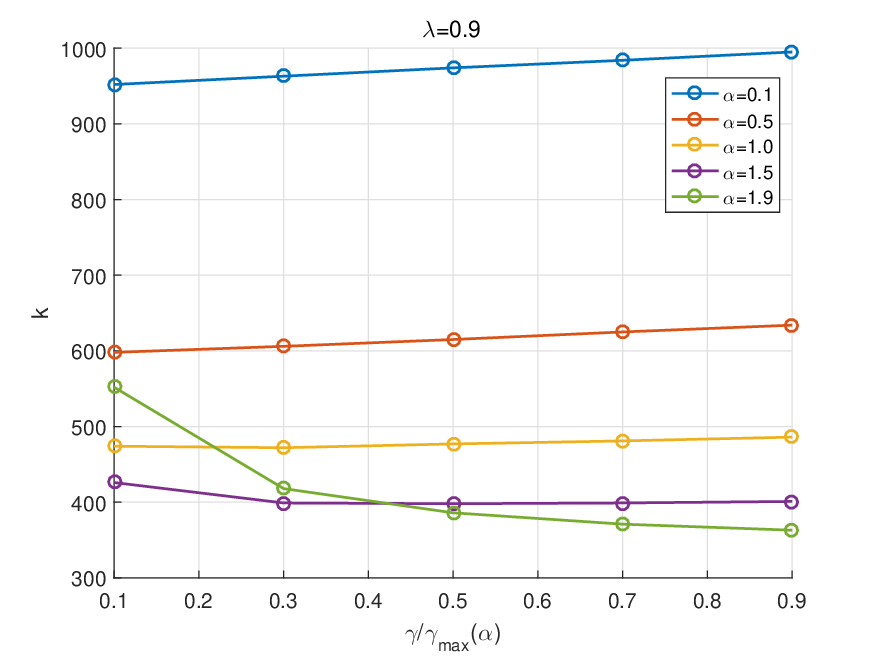}}
    }\\
    \caption{Influence of $\alpha$ and $\gamma$ on number of iterations $k$ for different $\lambda$ settings.}
    \label{parameter-iter}
\end{figure}

\begin{figure}[H]
     \setlength{\abovecaptionskip}{0pt}
  \centering
  \makeatletter
    \subfigure[$\lambda = 0.1$]{
        \scalebox{0.3}{\includegraphics{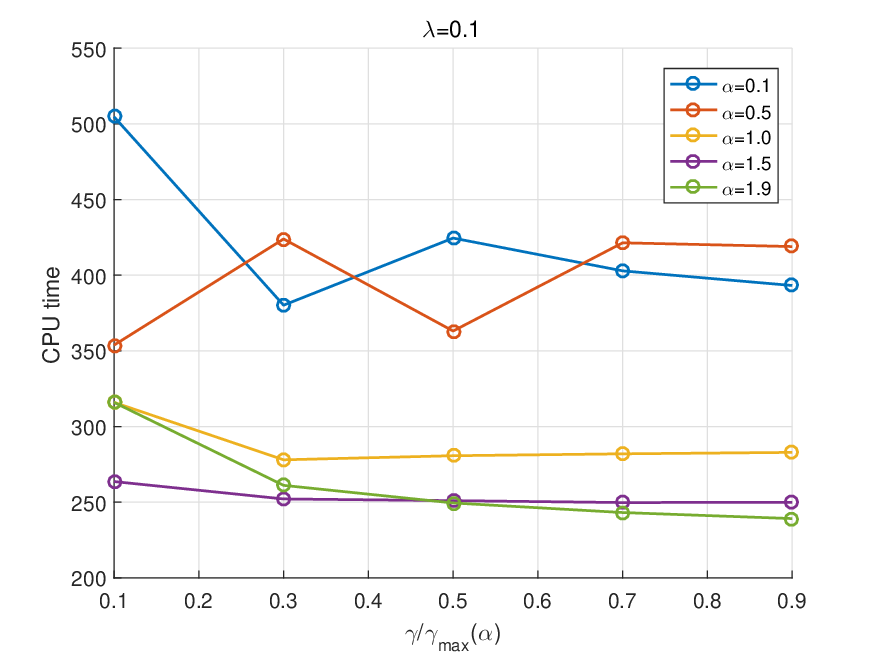}}
    } \hspace{-20pt}
      \subfigure[$\lambda = 0.3$]{
        \scalebox{0.3}{\includegraphics{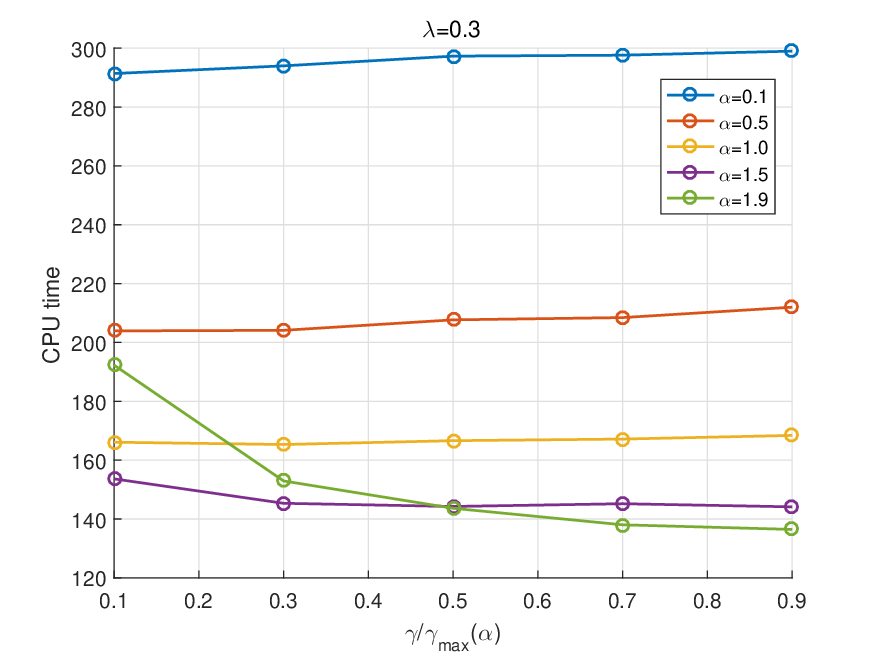}}
    }\hspace{-20pt}
    \subfigure[$\lambda = 0.5$]{
        \scalebox{0.3}{\includegraphics{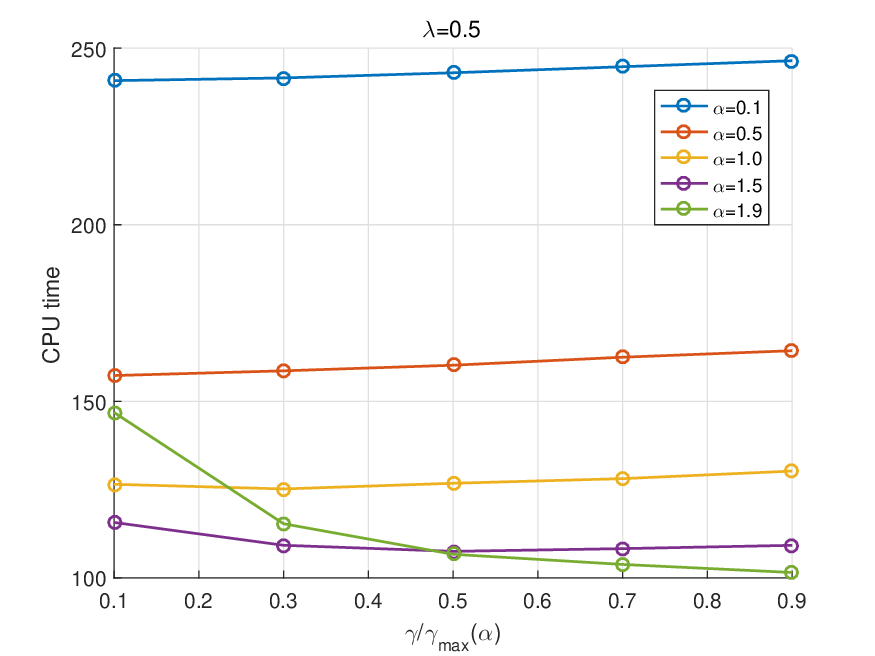}}
    }  \\
    \subfigure[$\lambda = 0.7$]{
        \scalebox{0.3}{\includegraphics{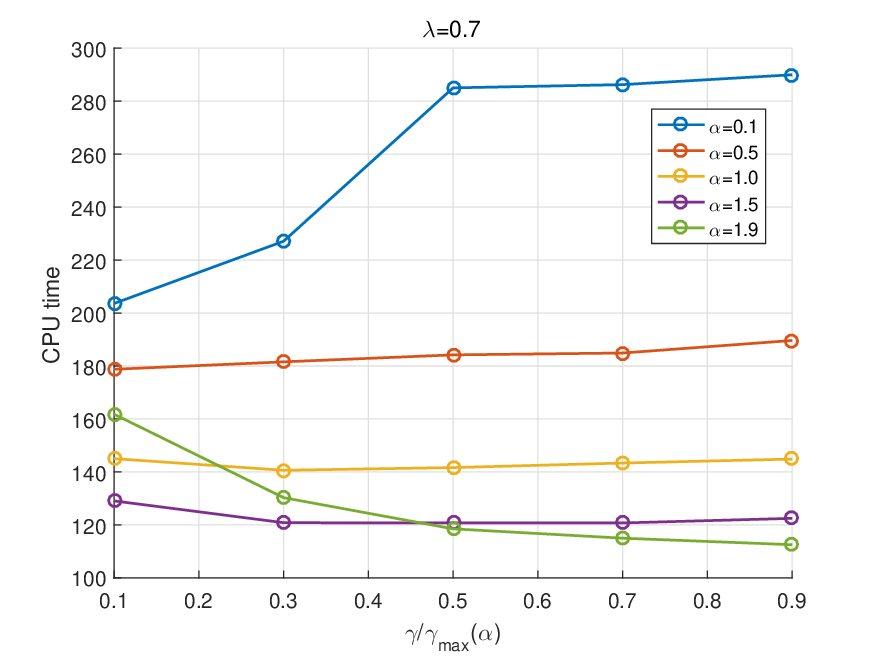}}
    } \hspace{-20pt}
      \subfigure[$\lambda = 0.8$]{
        \scalebox{0.3}{\includegraphics{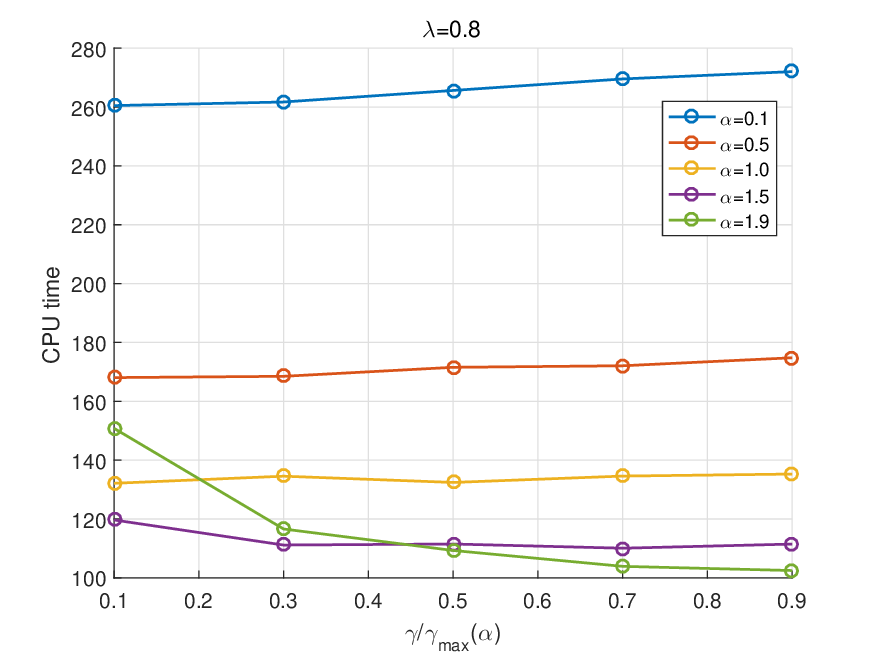}}
    }\hspace{-20pt}
     \subfigure[$\lambda = 0.9$]{
        \scalebox{0.3}{\includegraphics{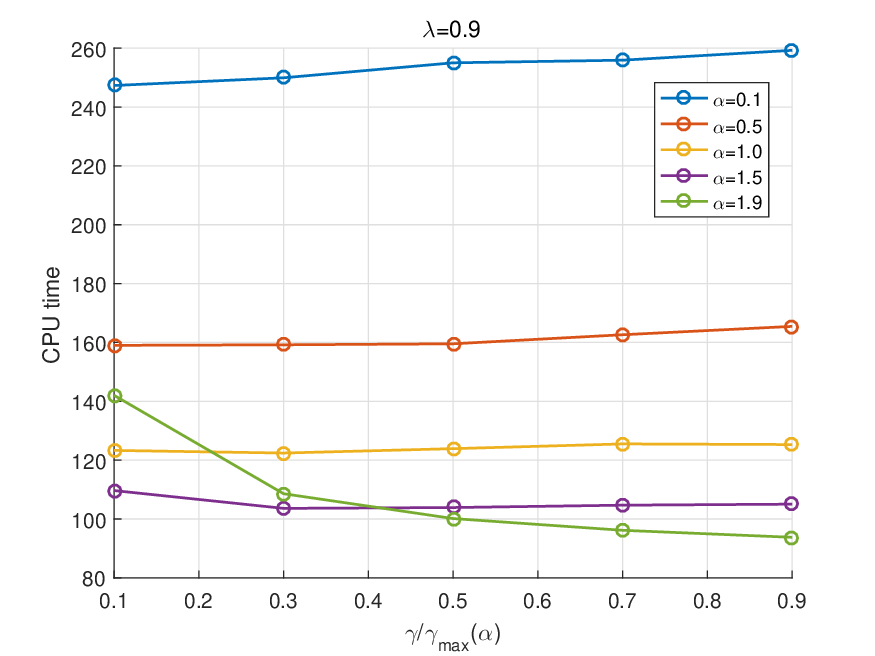}}
    }\\
    \caption{Influence of $\alpha$ and $\gamma$ on CPU time in seconds for different $\lambda$ settings.}
    \label{parameter-CPU}
\end{figure}

\noindent \textbf{Influence of $\gamma$.} From the PSNR curves, it can be observed that varying $\gamma/\gamma_{\max}(\alpha)$ within the admissible range introduces only negligible changes in reconstruction quality. For all combinations of $\lambda$ and $\alpha$, the PSNR variations remain within a very narrow band (typically less than 0.03 dB). This indicates that the proposed algorithm is insensitive to the parameter $\gamma$ in terms of reconstruction accuracy. A similar trend can be observed in the convergence curves. The number of iterations and CPU time remain nearly constant as $\gamma$ increases. Slight acceleration is observed only when $\alpha$ takes relatively large values, but the improvement is marginal. These observations suggest that $\gamma$ primarily affects numerical stability rather than reconstruction performance. Within the feasible range, the algorithm demonstrates strong robustness to $\gamma$, simplifying practical parameter tuning.

\noindent \textbf{Influence of $\alpha$.} Compared with $\gamma$, the parameter $\alpha$ plays a significantly more critical role. The PSNR consistently increases as $\alpha$ increases. When $\alpha$ grows from $0.1$ to $1.5$, a clear performance gain is observed. However, the improvement becomes saturated when $\alpha$ exceeds $1.5$, as the difference between $\alpha=1.5$ and $\alpha=1.9$ is marginal. The impact of $\alpha$ on convergence is even more pronounced. As $\alpha$ increases, the number of iterations decreases significantly. For example, when $\lambda=0.1$, the iteration count drops from approximately $2000$ ($\alpha=0.1$) to around $1200$ ($\alpha=1.9$). The CPU time exhibits the same trend, confirming that $\alpha$ is a dominant factor governing convergence efficiency. This phenomenon suggests that a larger $\alpha$ improves the conditioning of the optimization problem, thereby accelerating convergence.

\noindent \textbf{Influence of $\lambda$.} The parameter $\lambda$ also affects algorithm performance, though in a different manner.
As $\lambda$ increases from $0.1$ to $0.9$: PSNR exhibits a slight but consistent improvement (approximately $0.1$--$0.2$ dB); the number of iterations decreases significantly; the CPU time correspondingly reduces.
Notably, for fixed $\alpha$, the iteration count at $\lambda=0.9$ is nearly half of that at $\lambda=0.1$. This indicates that a larger $\lambda$ strengthens the regularization constraint, which stabilizes the optimization process and accelerates convergence. Therefore, $\lambda$ has a more substantial influence on convergence speed than on reconstruction quality.

Based on the above observations, the following conclusions can be drawn: The proposed algorithm is highly robust to $\gamma$ within the feasible range.
The parameter $\alpha$ is the most influential factor, affecting both reconstruction quality and convergence speed.
 Increasing $\lambda$ significantly accelerates convergence while slightly improving PSNR.
Performance saturation is observed when $\alpha$ exceeds approximately $1.5$.

From a practical perspective, the combination $\alpha \in [1.5, 1.9]$ and $\lambda \in [0.7, 0.9]$ provides a favorable trade-off between reconstruction accuracy and computational efficiency.

%

\subsubsection{Numerical results and discussions}

We select the regularization parameters $\mu_1$ and $\mu_2$ in (\ref{constrained-l2-nuclear-tv}) to maximize the PSNR of the restored images. The chosen parameter values for different test images of (a)-(c) in Figure \ref{test-images} and blur kernels are listed in Table~\ref{regularizaton-1}.

\begin{table}[h]
\centering
\caption{Selected values of the regularization parameters for the image restoration model~(\ref{constrained-l2-nuclear-tv}).}
\begin{tabular}{c|c|ccccc}
\hline
\multirow{2}[1]{*}{Image} & \multirow{2}[1]{*}{Kernel} & \multicolumn{2}{c}{$\sigma_g = 10$} &  & \multicolumn{2}{c}{$\sigma_g = 20$} \\
\cline{3-4} \cline{6-7}
 & & $\mu_1$ & $\mu_2$ & & $\mu_1$ & $\mu_2$ \\
\hline
\multirow{2}[1]{*}{Building}
& Uniform  & $0.1$ & $37$  &   & $0.2$   & $109.7$ \\
& Gaussian & $0.4$ & $24.9$ &   & $0.5$ & $125.5$ \\
\hline
\multirow{2}[1]{*}{Goldhill}
& Uniform  & $0.5$ & $22$  &   & $1.5$ & $62$ \\
& Gaussian & $0.6$   & $25$    &   & $2$ & $68$ \\
\hline
\multirow{2}[1]{*}{Castle}
& Uniform  & $0.5$ & $6.6$  &   & $1.1$ & $37.5$ \\
& Gaussian & $0.7$   & $9.5$    &   & $1.9$ & $43.5$ \\
\hline
\end{tabular}
\label{regularizaton-1}
\end{table}

We compare the proposed algorithm with the primal-dual algorithm of \cite{Tang2022JSC}, which is referred to as PFDR. The iterative parameters of the proposed algorithm are set as: $\lambda = 0.8$, $\alpha = 1.5$, and $\gamma = 0.0188$. The parameters for  PFDR are chosen as specified \cite{Tang2022JSC}.  Table \ref{table-results} reports the quantitative comparisons under different kernels, images, and noise levels. The proposed algorithm consistently achieves marginally higher PSNR and SSIM values than PFDR \cite{Tang2022JSC}, indicating more stable restoration quality. More importantly, it substantially reduces the number of iterations required for convergence in most scenarios. For example, on the Building image with a Uniform kernel and $\sigma_g=10$, the iteration count decreases from $779$ to $430$. Although a slight increase in iterations is observed in a few cases, the proposed method still yields improved image quality. Overall, these results confirm that the proposed algorithm attains better restoration performance with enhanced convergence efficiency, thereby outperforming the PFDR method. These findings indicate that our approach not only delivers better restoration quality but also converges more efficiently. To provide a more intuitive comparison of the two algorithms, Figures \ref{psnr1}, \ref{psnr2}, and \ref{psnr3} present the objective function values and PSNR curves with respect to the number of iterations under different noise levels.  It can be observed that the proposed method achieves a significantly faster decrease in the objective function, converging to a stable solution within a small number of iterations, whereas PFDR exhibits a slower convergence behavior. Meanwhile, the PSNR curves indicate that the proposed algorithm provides a more rapid improvement in image quality during the early iterations and consistently attains higher or at least comparable final PSNR values. This advantage is maintained across different noise types and noise intensities, demonstrating strong robustness of the proposed method. Furthermore, Figures \ref{restored1} and \ref{restored2} show the images restored by the two algorithms.

\begin{table}[h]
\centering
\footnotesize
\caption{Numerical results for solving (\ref{constrained-l2-nuclear-tv}) in terms of PSNR (dB), SSIM, number of iterations, and CPU time (seconds).}
\begin{tabular}{c|c|c|ccc}
\hline
\multirow{2}[1]{*}{Kernel} & \multirow{2}[1]{*}{Image}&  \multirow{2}[1]{*}{Noise level} & Input  &  PFDR \cite{Tang2022JSC}   & Proposed algorithm  \\
& & & PSNR/SSIM & PSNR/SSIM/Iter/CPU &   PSNR/SSIM/Iter/CPU  \\
\hline
\hline
\multirow{6}[1]{*}{Uniform} & \multirow{2}[1]{*}{Building}
 & $\sigma_g = 10$ & $19.6706/0.2350$ & $23.5273$/$0.6026$/$779$/$263.3$ &  $23.6505$/$0.6100$/$430$/$111.4$ \\
& & $\sigma_g = 20$  & 18.1314/0.1500 & $22.3512$/$0.5184$/$481$/$166.1$  &    $22.4533$/$0.5252$/$264$/$67.3$   \\
\cline{2-6}
& \multirow{2}[1]{*}{Goldhill}
 & $\sigma_g = 10$ & $23.3278/0.3428$ & $26.9840$/$0.6462$/$341$/$140.3$ &  $27.0185$/$0.6469$/$238$/$68.8$ \\
& & $\sigma_g = 20$  & $20.3375/0.1768$ & $26.0737$/$0.6031$/$332$/$124.2$  &    $26.1113$/$0.6038$/$370$/$114.6$   \\
\cline{2-6}
& \multirow{2}[1]{*}{Castle}
 & $\sigma_g = 10$ & $21.3462/0.3296$ & $24.4988$/$0.7155$/$501$/$94.8$ &  $24.4956$/$0.7042$/$258$/$33.2$ \\
& & $\sigma_g = 20$  & $19.2419/0.1539$ & $23.5216$/$0.6841$/$298$/$56.1$  &    $23.5625$/$0.6740$/$293$/$41.1$   \\
\hline
\multirow{6}[1]{*}{Gaussian} & \multirow{2}[1]{*}{Building}
 & $\sigma_g = 10$ & $21.0123/0.3727$ & $23.0711$/$0.5728$/$799$/$330.1$ &  $23.0959$/$0.5750$/$395$/$103.3$ \\
& & $\sigma_g = 20$  & $19.0324$/0.2485 & $22.4699$/$0.5254$/$420$/$135.5$  &    $22.5317$/$0.5305$/$238$/$69.5$   \\
\cline{2-6}
& \multirow{2}[1]{*}{Goldhill}
 & $\sigma_g = 10$ & $24.5985/0.4310$ & $27.8662$/$0.6921$/$339$/$94.7$ &  $27.8934$/$0.6924$/$256$/$55.7$ \\
& & $\sigma_g = 20$  & $20.9215/0.2307$ & $26.8681$/$0.6440$/$348$/$101.6$  &    $26.8998$/$0.6444$/$426$/$95.6$   \\
\cline{2-6}
& \multirow{2}[1]{*}{Castle}
 & $\sigma_g = 10$ & $22.5061/0.3903$ & $24.6871$/$0.7396$/$551$/$110.7$ &  $24.6843$/$0.7341$/$327$/$44.8$ \\
& & $\sigma_g = 20$  & $19.9043/0.1968$ & $24.0225$/$0.7092$/$321$/$67.6$  &    $24.0392$/$0.7020$/$383$/$51.2$   \\
\hline
\end{tabular}\label{table-results}

\end{table}

\begin{figure}[H]
     \setlength{\abovecaptionskip}{0pt}
  \centering
  \makeatletter
    \subfigure[Uniform, $\sigma_g = 10$]{
        \scalebox{0.4}{\includegraphics{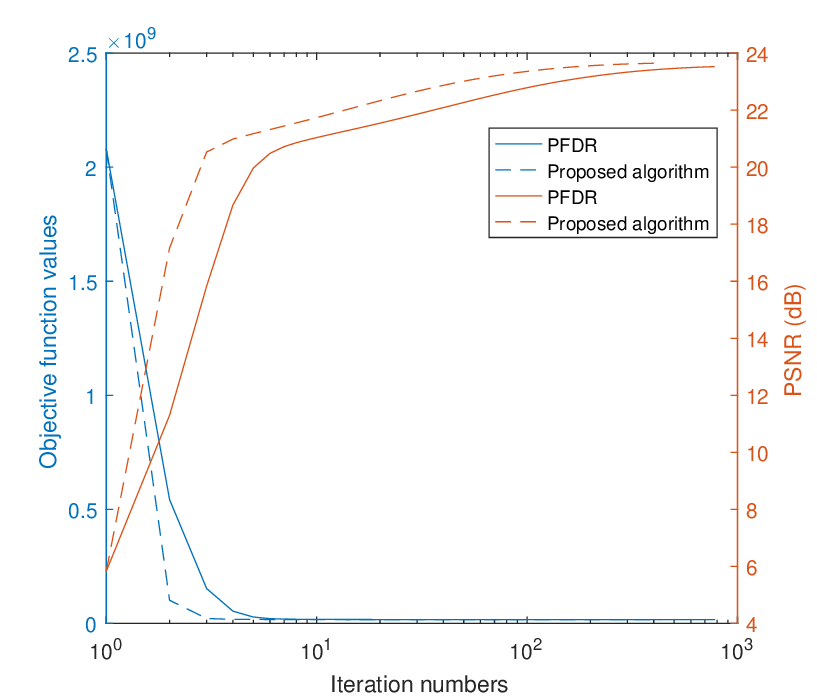}}
    } \hspace{-10pt}
      \subfigure[Uniform, $\sigma_g = 20$]{
        \scalebox{0.4}{\includegraphics{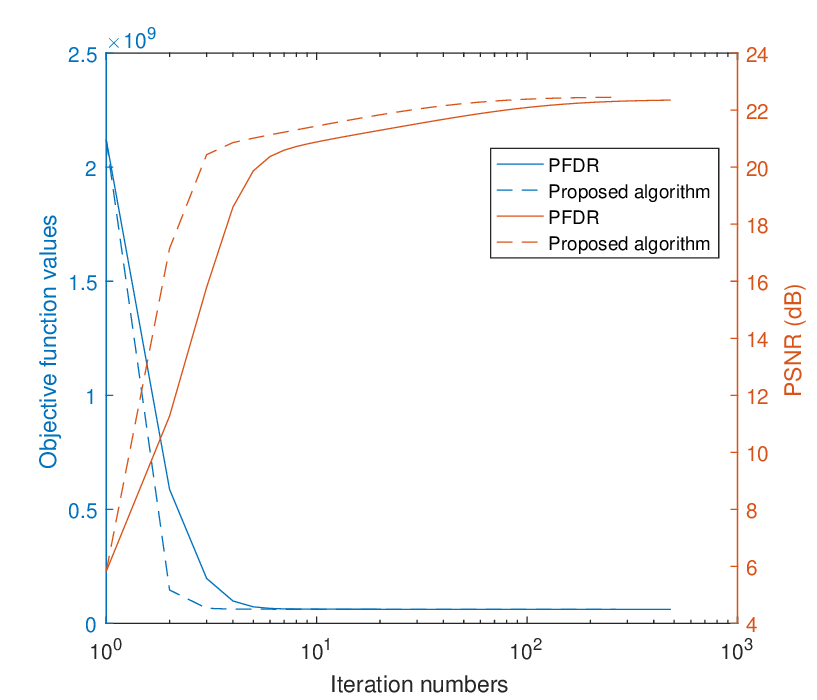}}
    } \\
    \subfigure[Gaussian, $\sigma_g = 10$]{
        \scalebox{0.4}{\includegraphics{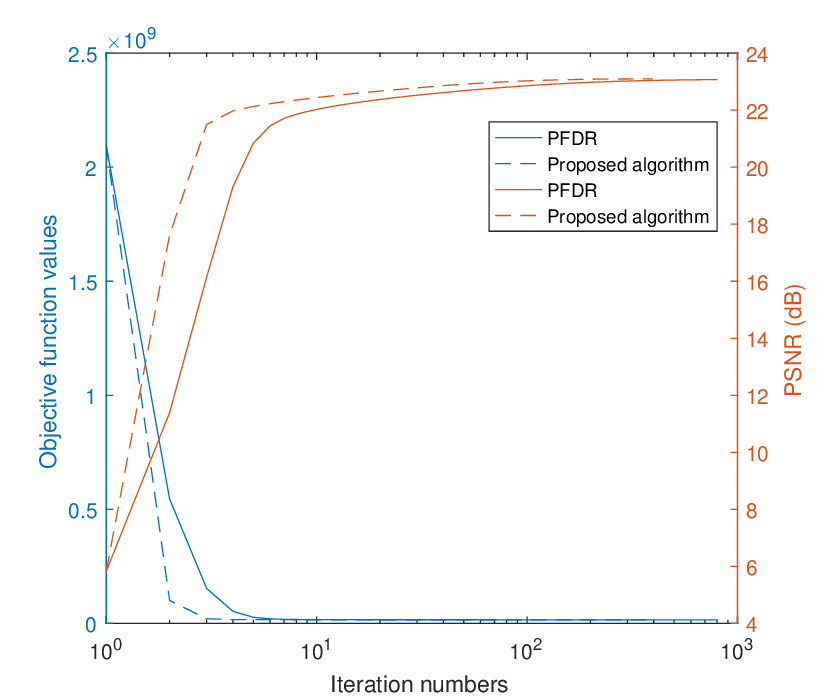}}
    } \hspace{-10pt}
      \subfigure[Gaussian, $\sigma_g = 20$]{
        \scalebox{0.4}{\includegraphics{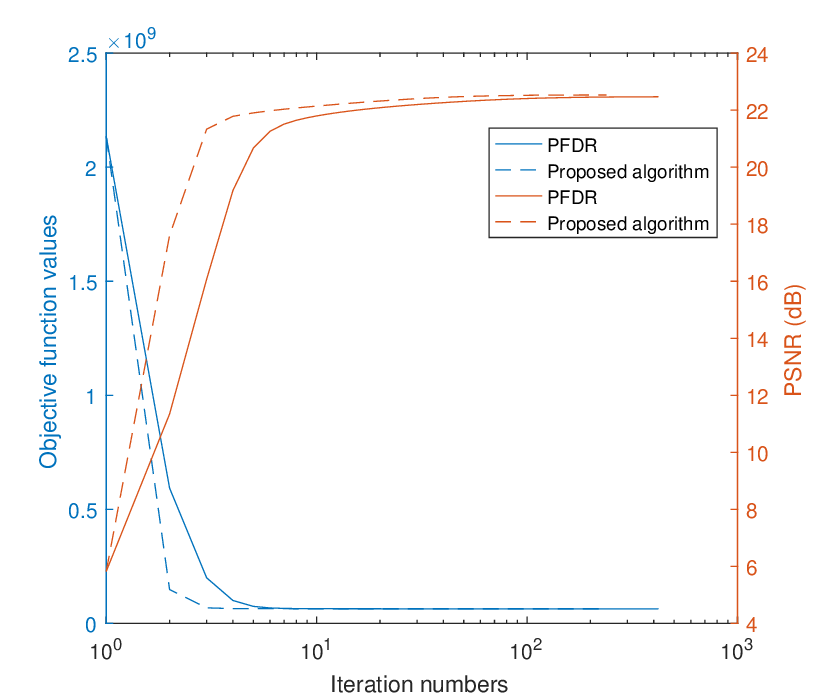}}
    }\\
    \caption{Objective function values and PSNR versus the number of iterations for the test image ``Building".}
    \label{psnr1}
\end{figure}

\begin{figure}[H]
     \setlength{\abovecaptionskip}{0pt}
  \centering
  \makeatletter
    \subfigure[Uniform, $\sigma_g = 10$]{
        \scalebox{0.4}{\includegraphics{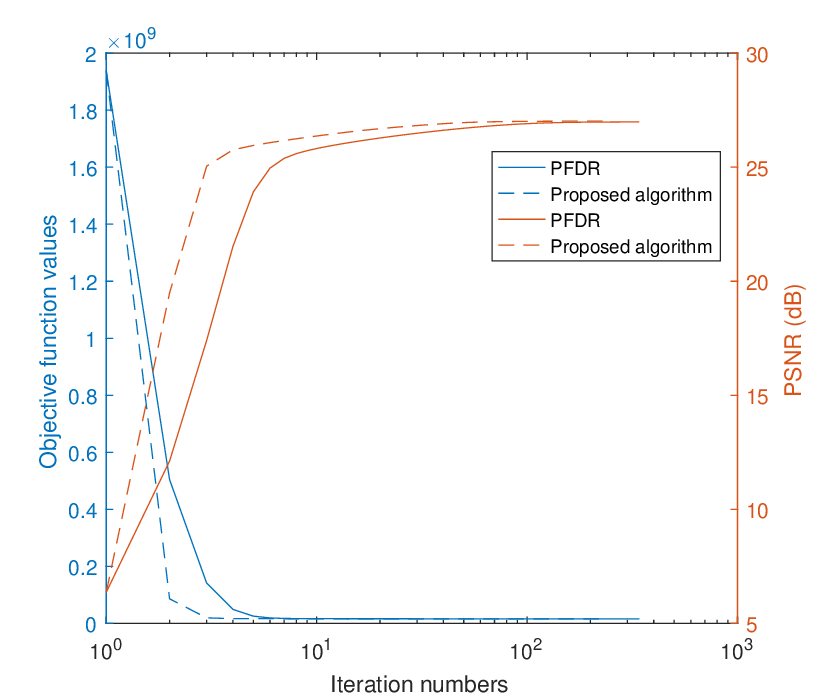}}
    } \hspace{-10pt}
      \subfigure[Uniform, $\sigma_g = 20$]{
        \scalebox{0.4}{\includegraphics{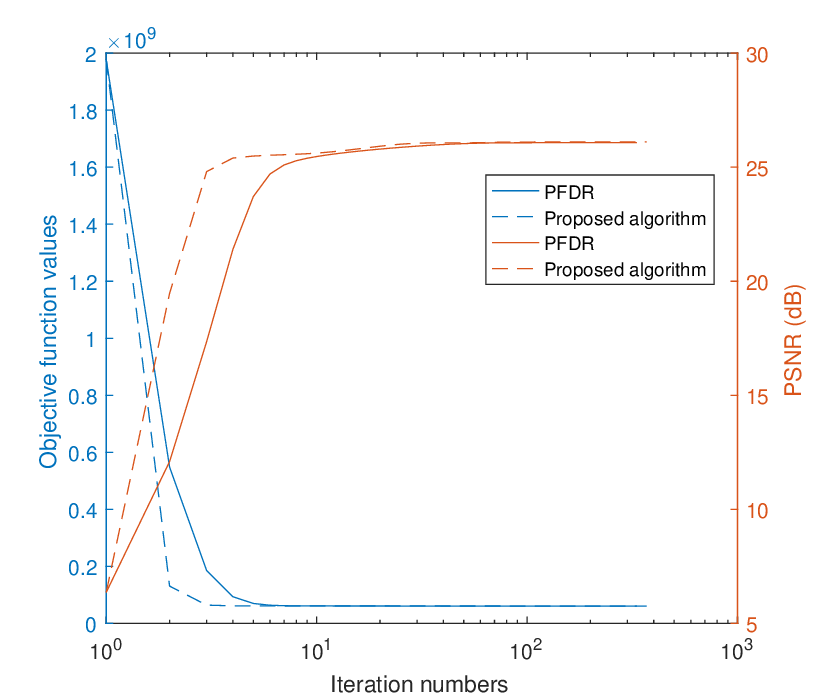}}
    } \\
    \subfigure[Gaussian, $\sigma_g = 10$]{
        \scalebox{0.4}{\includegraphics{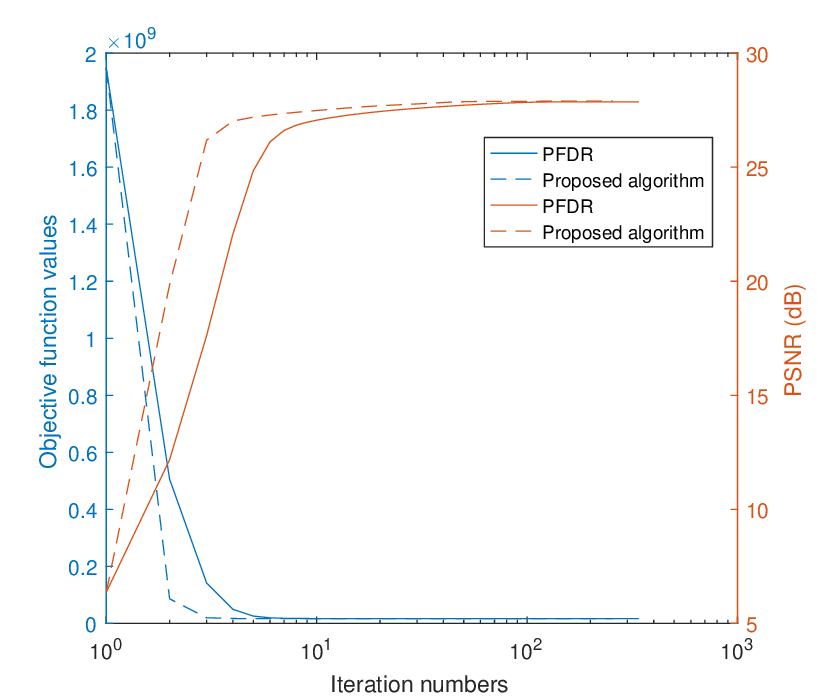}}
    } \hspace{-10pt}
      \subfigure[Gaussian, $\sigma_g = 20$]{
        \scalebox{0.4}{\includegraphics{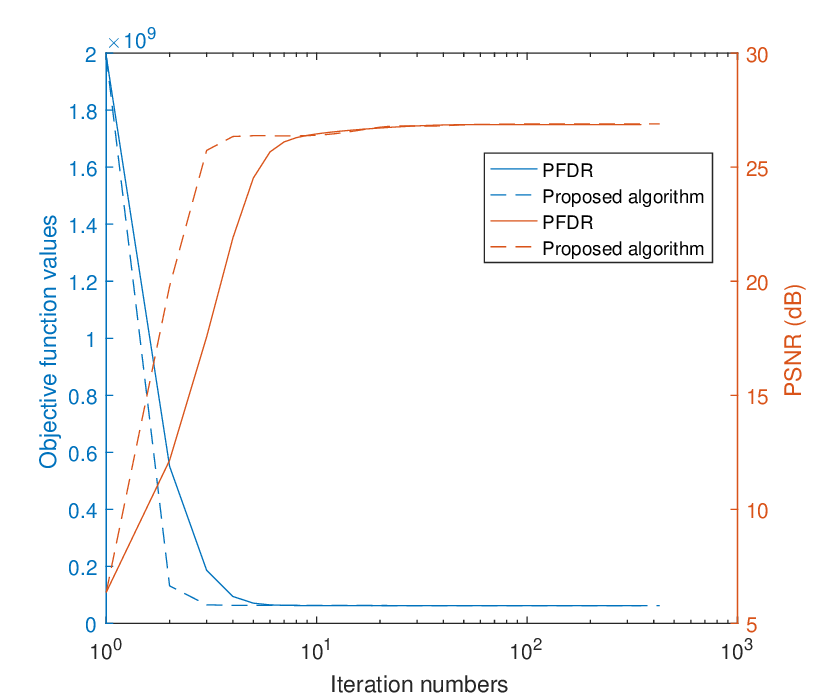}}
    }\\
    \caption{Objective function values and PSNR versus the number of iterations for the test image ``Goldhill".}
    \label{psnr2}
\end{figure}

\begin{figure}[H]
     \setlength{\abovecaptionskip}{0pt}
  \centering
  \makeatletter
    \subfigure[Uniform, $\sigma_g = 10$]{
        \scalebox{0.4}{\includegraphics{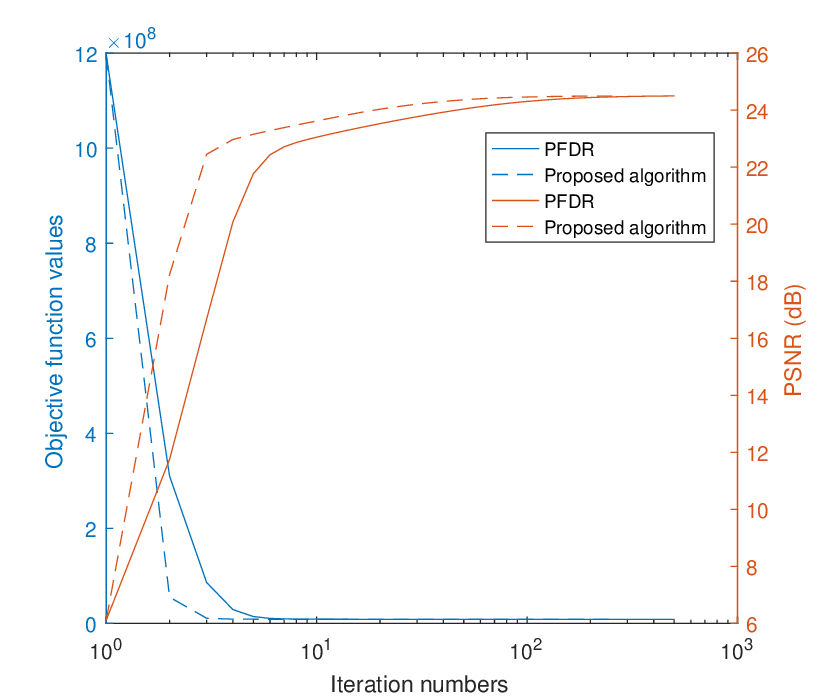}}
    } \hspace{-10pt}
      \subfigure[Uniform, $\sigma_g = 20$]{
        \scalebox{0.4}{\includegraphics{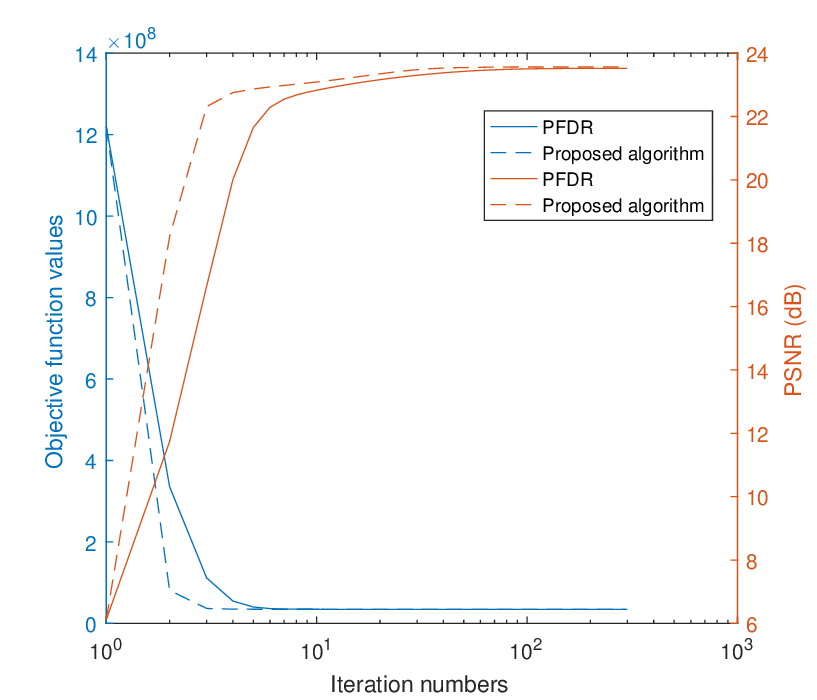}}
    } \\
    \subfigure[Gaussian, $\sigma_g = 10$]{
        \scalebox{0.4}{\includegraphics{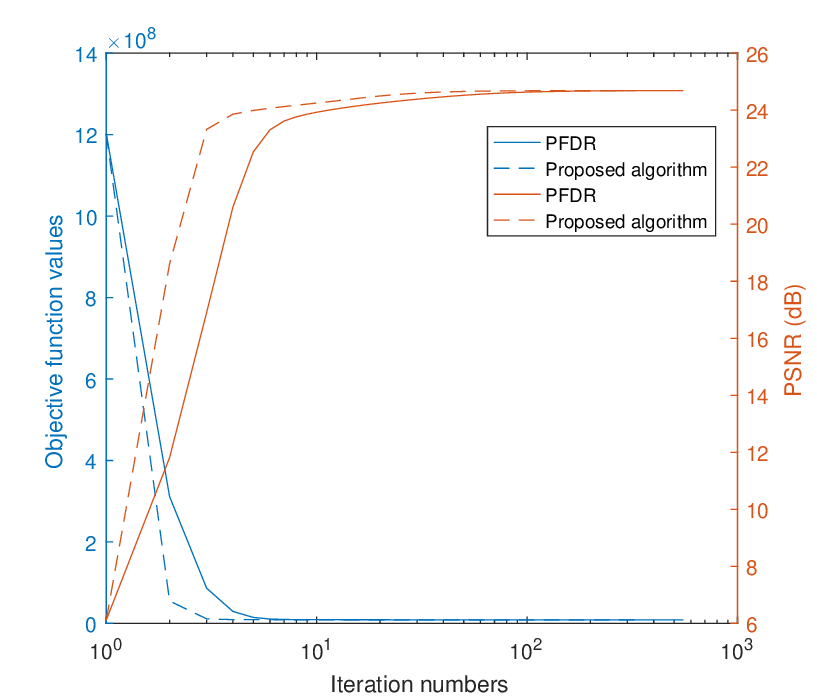}}
    } \hspace{-10pt}
      \subfigure[Gaussian, $\sigma_g = 20$]{
        \scalebox{0.4}{\includegraphics{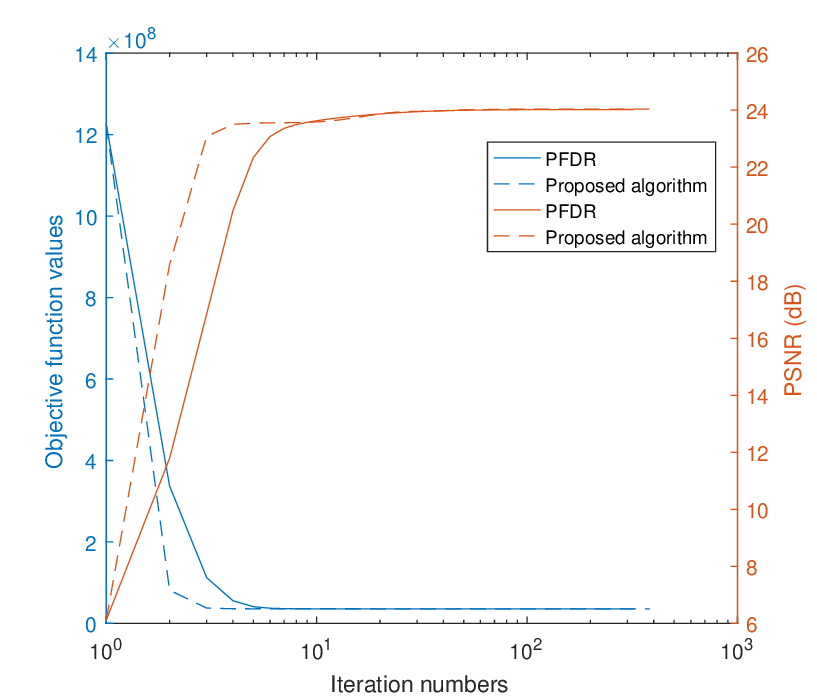}}
    }\\
    \caption{Objective function values and PSNR versus the number of iterations for the test images ``Castle".}
    \label{psnr3}
\end{figure}

\begin{figure}[H]
     \setlength{\abovecaptionskip}{0pt}
  \centering
  \makeatletter
    \subfigure[Uniform, $\sigma_g = 10$]{
        \scalebox{0.35}{\includegraphics{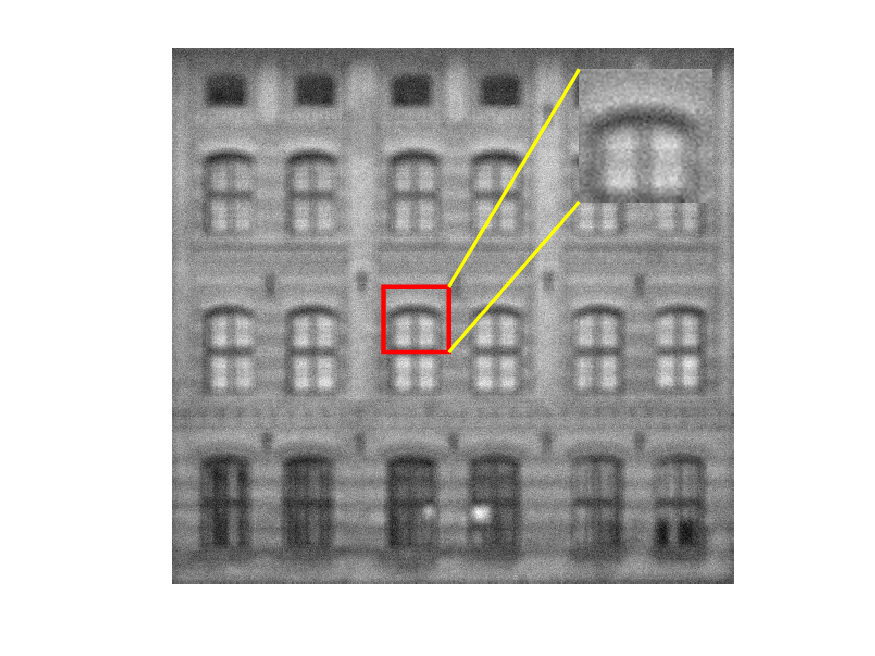}}
    } \hspace{-30pt}
      \subfigure[PFDR]{
        \scalebox{0.35}{\includegraphics{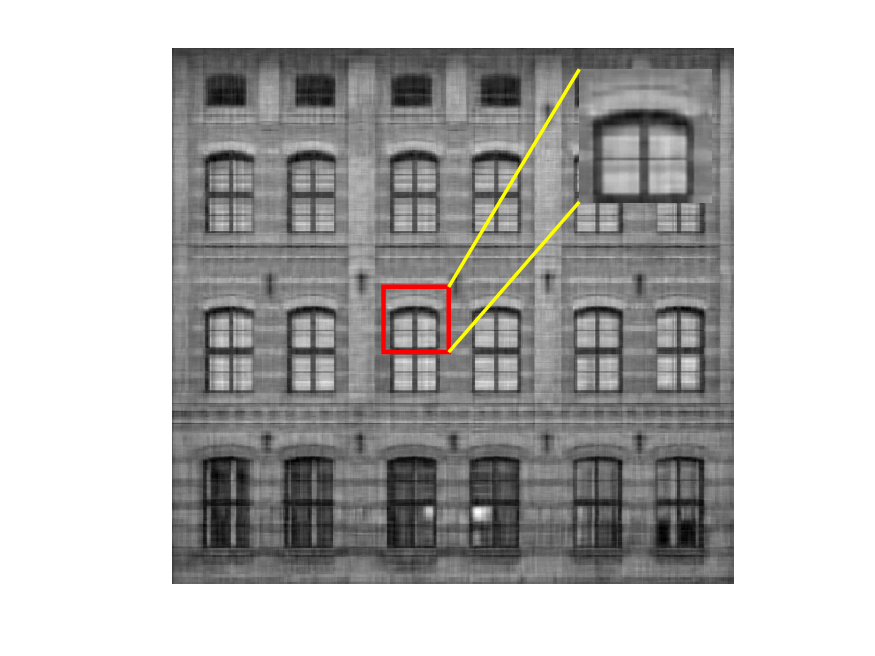}}
    } \hspace{-30pt}
    \subfigure[Proposed algorithm]{
        \scalebox{0.35}{\includegraphics{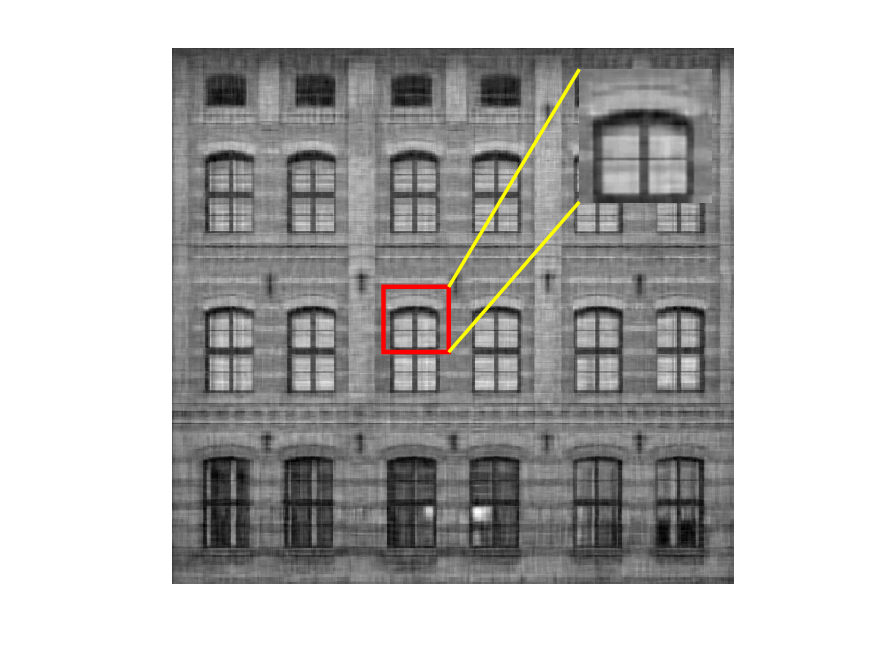}}
    }\\
    \subfigure[Uniform, $\sigma_g = 20$]{
        \scalebox{0.35}{\includegraphics{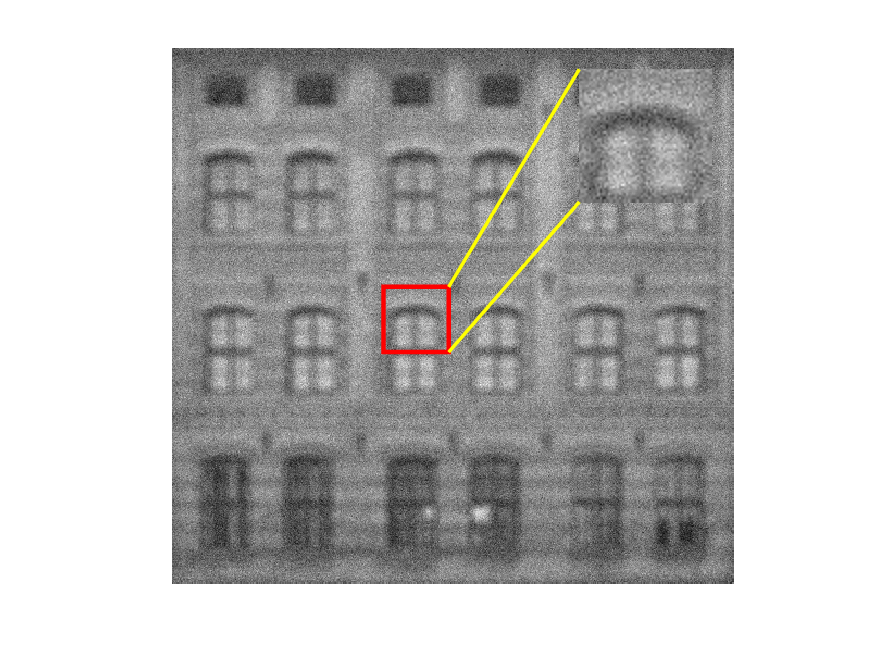}}
    } \hspace{-30pt}
      \subfigure[PFDR]{
        \scalebox{0.35}{\includegraphics{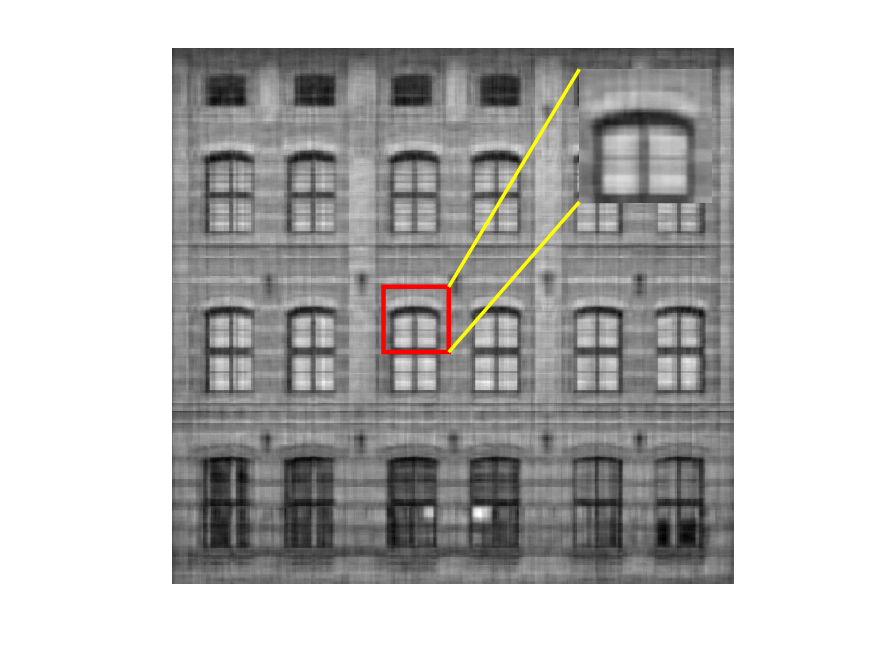}}
    } \hspace{-30pt}
    \subfigure[Proposed algorithm]{
        \scalebox{0.35}{\includegraphics{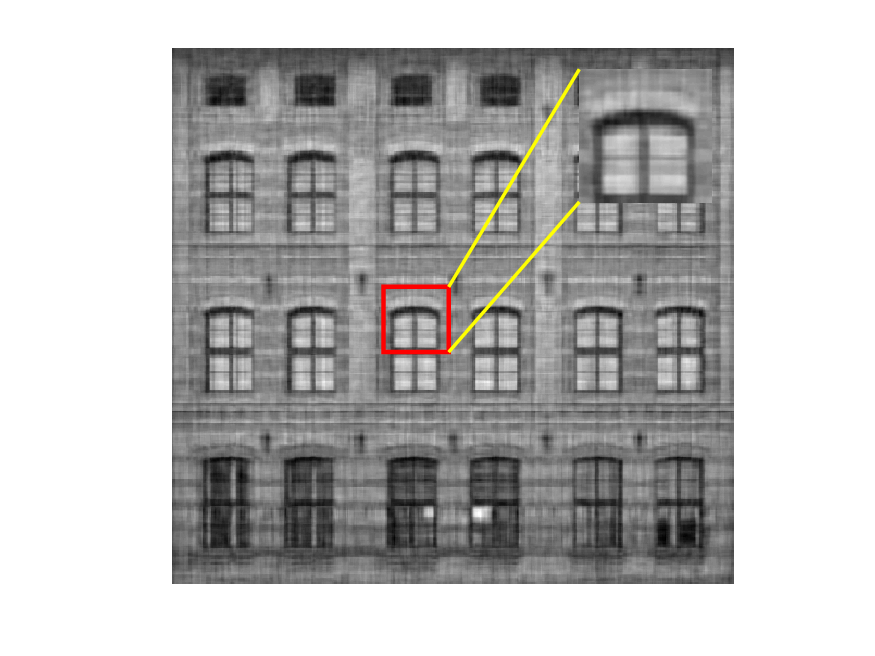}}
    }\\
     \subfigure[Gaussian, $\sigma_g = 10$]{
        \scalebox{0.35}{\includegraphics{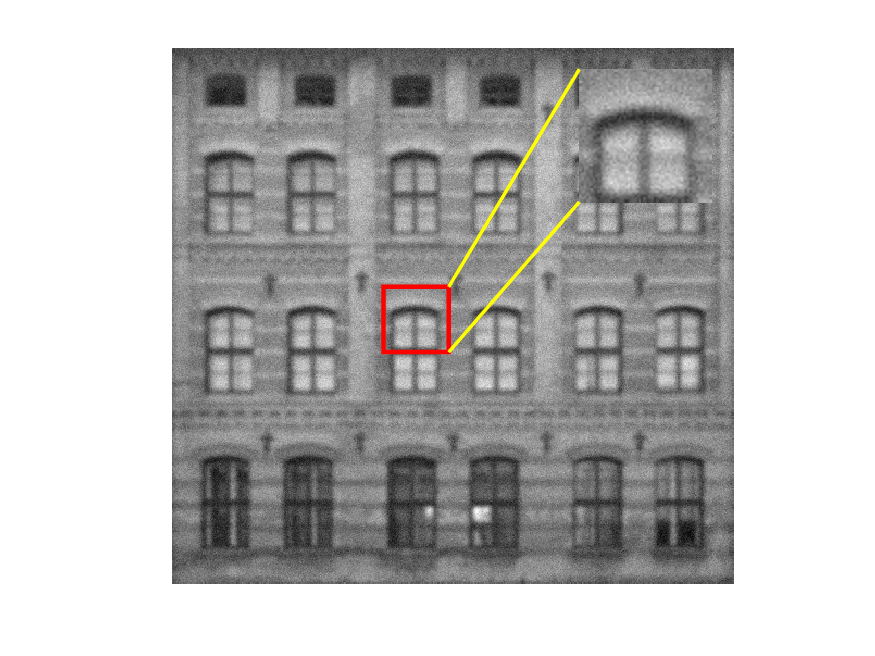}}
    } \hspace{-30pt}
      \subfigure[PFDR]{
        \scalebox{0.35}{\includegraphics{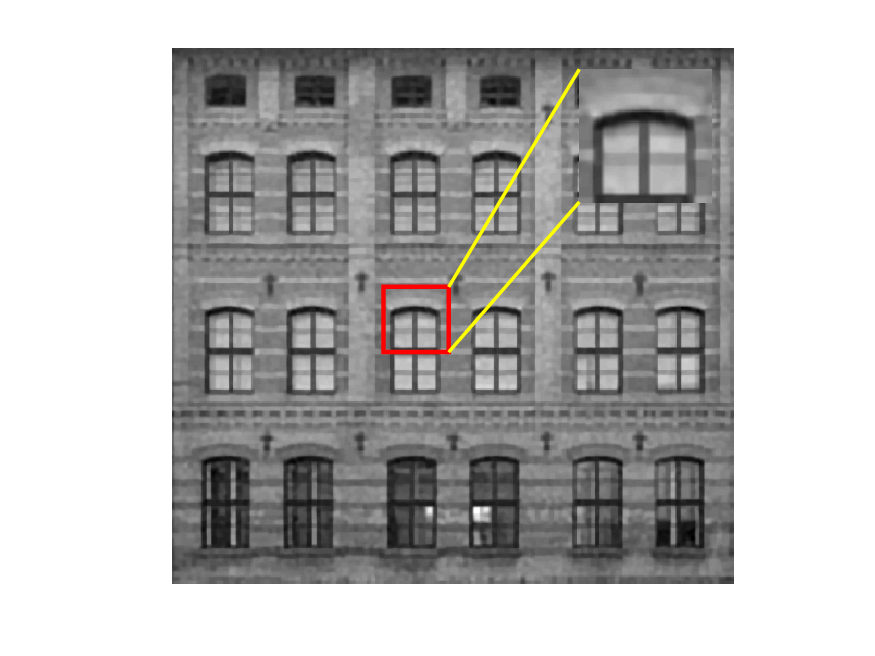}}
    } \hspace{-30pt}
    \subfigure[Proposed algorithm]{
        \scalebox{0.35}{\includegraphics{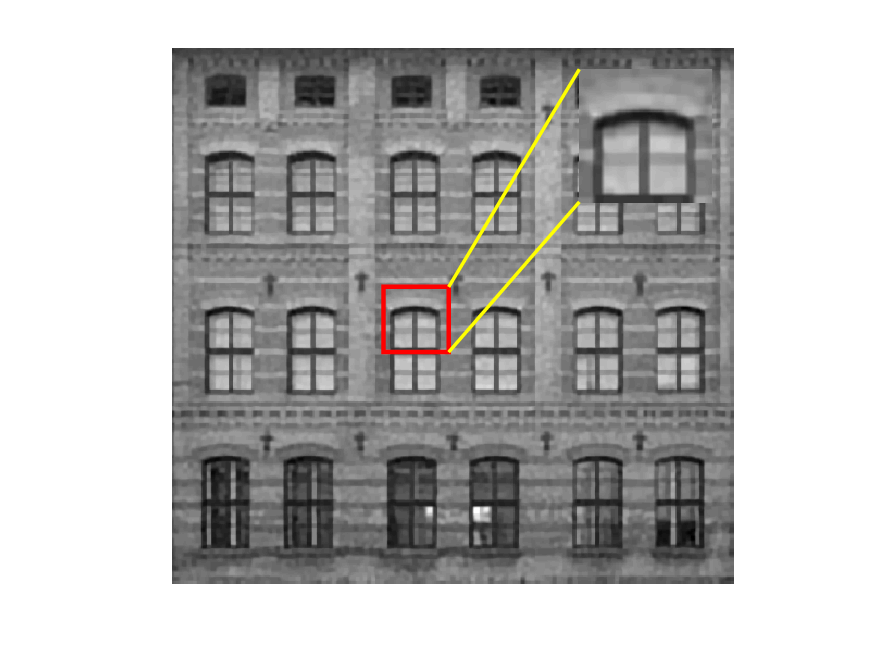}}
    }\\
    \subfigure[Gaussian, $\sigma_g = 20$]{
        \scalebox{0.35}{\includegraphics{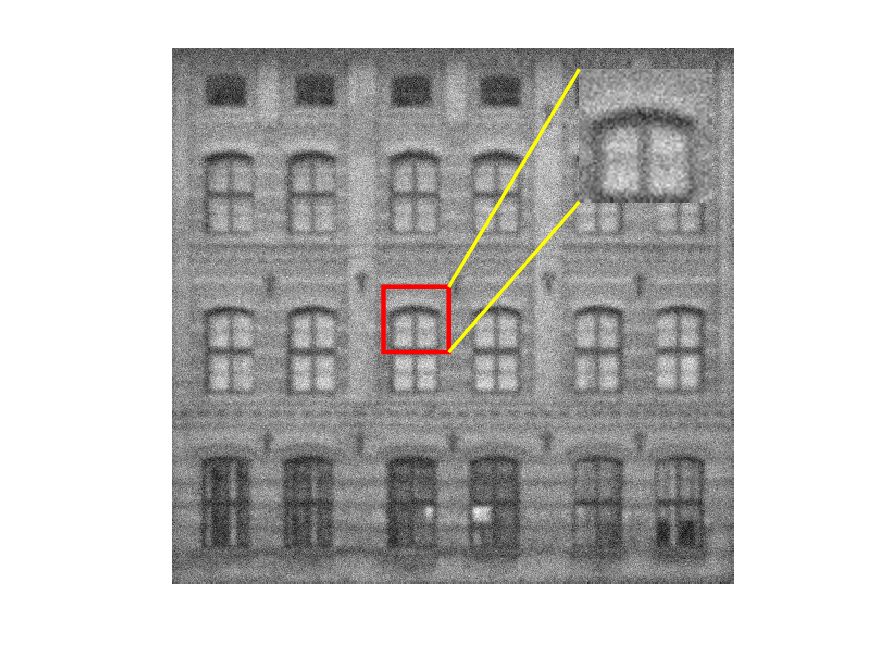}}
    } \hspace{-30pt}
      \subfigure[PFDR]{
        \scalebox{0.35}{\includegraphics{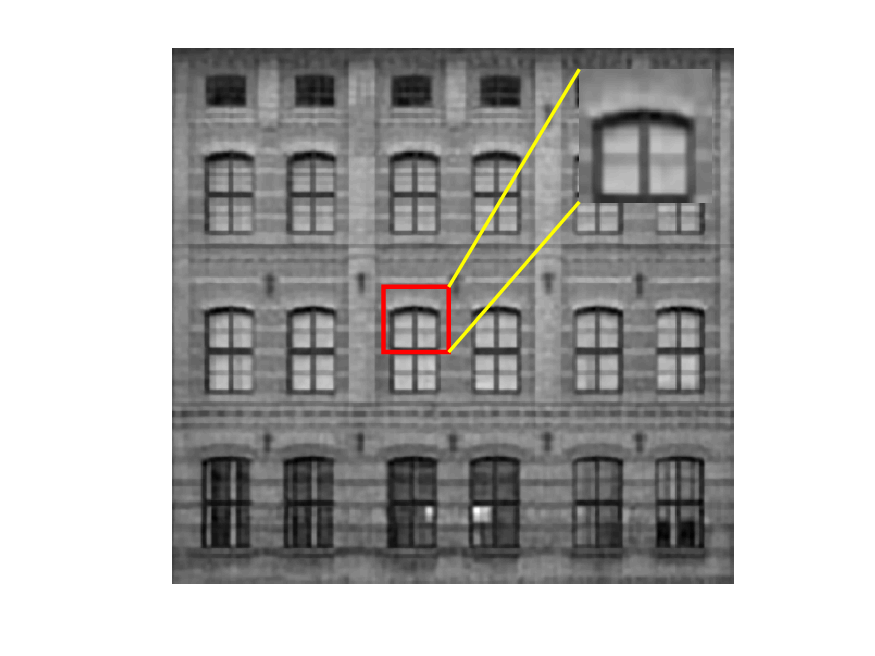}}
    } \hspace{-30pt}
    \subfigure[Proposed algorithm]{
        \scalebox{0.35}{\includegraphics{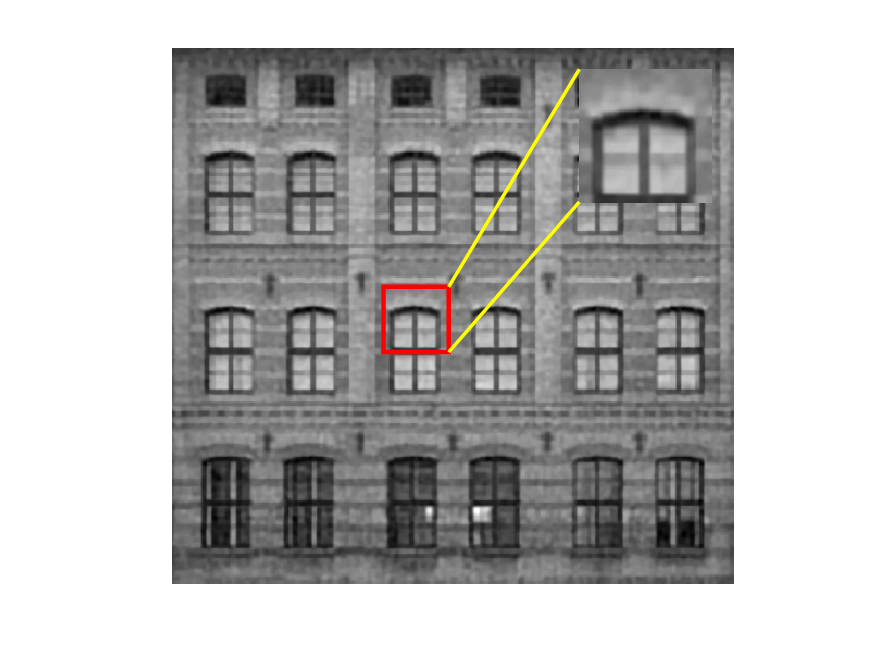}}
    }\\
    \caption{Corrupted and restored results of  the ``Building" image. The first column presents the corrupted images, whereas the second and third columns display the images restored using PFDR \cite{Tang2022JSC} and the proposed algorithm, respectively. }
    \label{restored1}
\end{figure}

\begin{figure}[H]
     \setlength{\abovecaptionskip}{0pt}
  \centering
  \makeatletter
    \subfigure[Uniform, $\sigma_g = 10$]{
        \scalebox{0.35}{\includegraphics{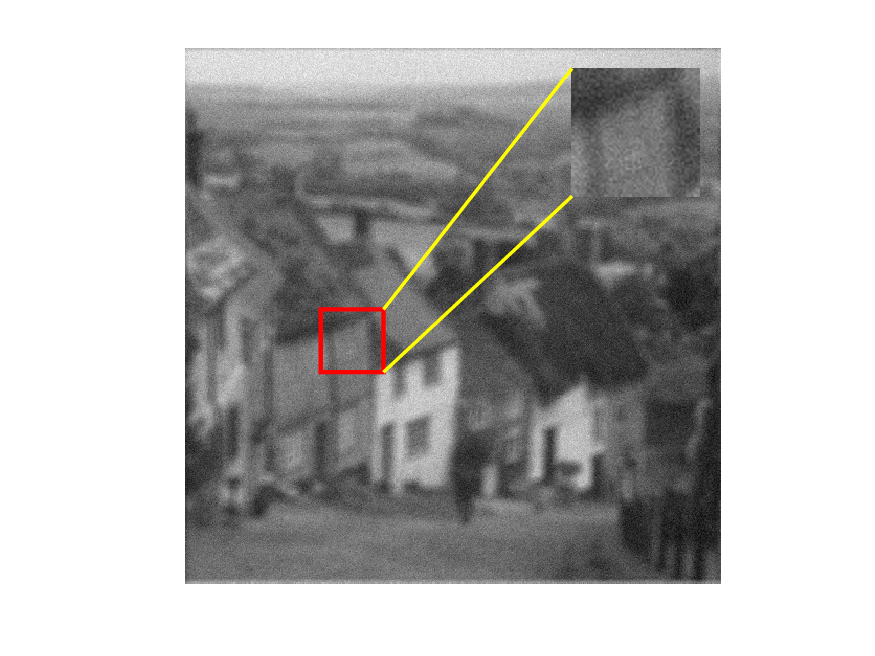}}
    } \hspace{-30pt}
      \subfigure[PFDR]{
        \scalebox{0.35}{\includegraphics{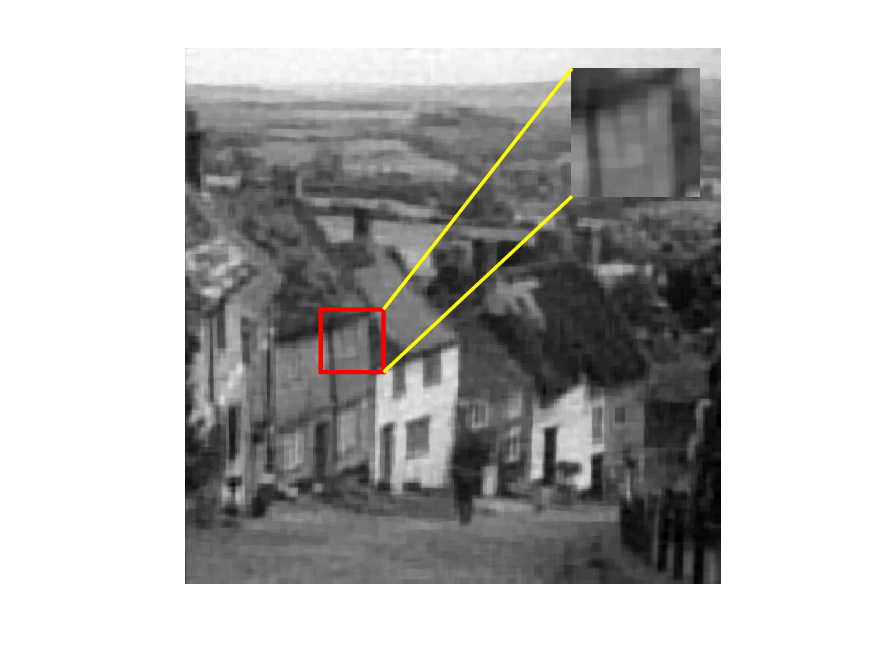}}
    } \hspace{-30pt}
    \subfigure[Proposed algorithm]{
        \scalebox{0.35}{\includegraphics{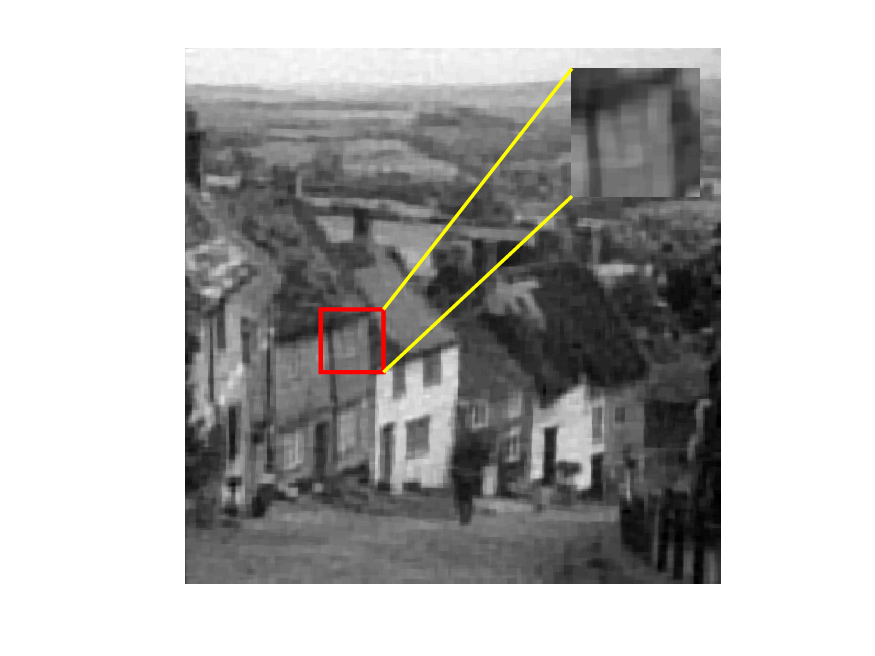}}
    }\\
    \subfigure[Uniform, $\sigma_g = 20$]{
        \scalebox{0.35}{\includegraphics{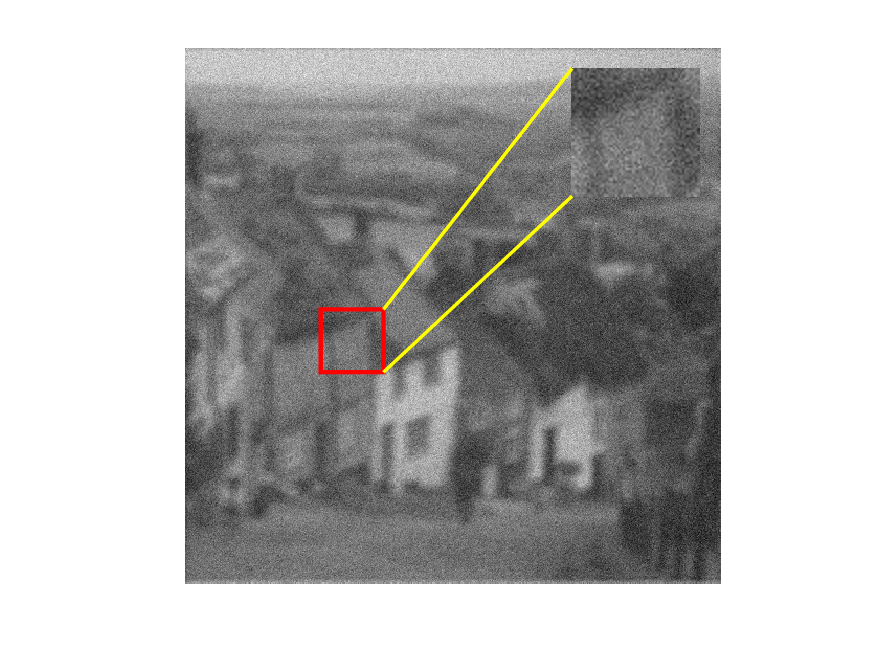}}
    } \hspace{-30pt}
      \subfigure[PFDR]{
        \scalebox{0.35}{\includegraphics{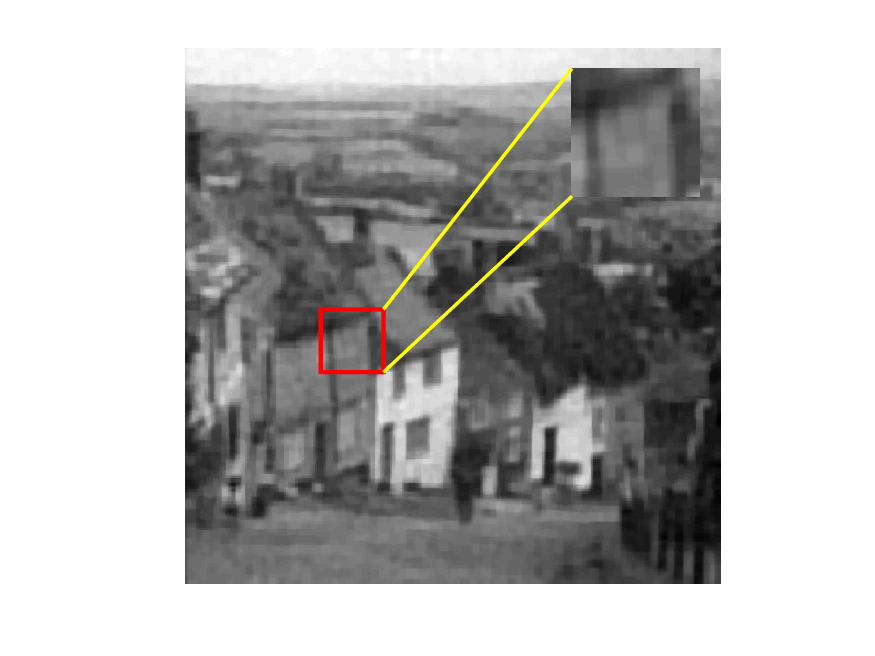}}
    } \hspace{-30pt}
    \subfigure[Proposed algorithm]{
        \scalebox{0.35}{\includegraphics{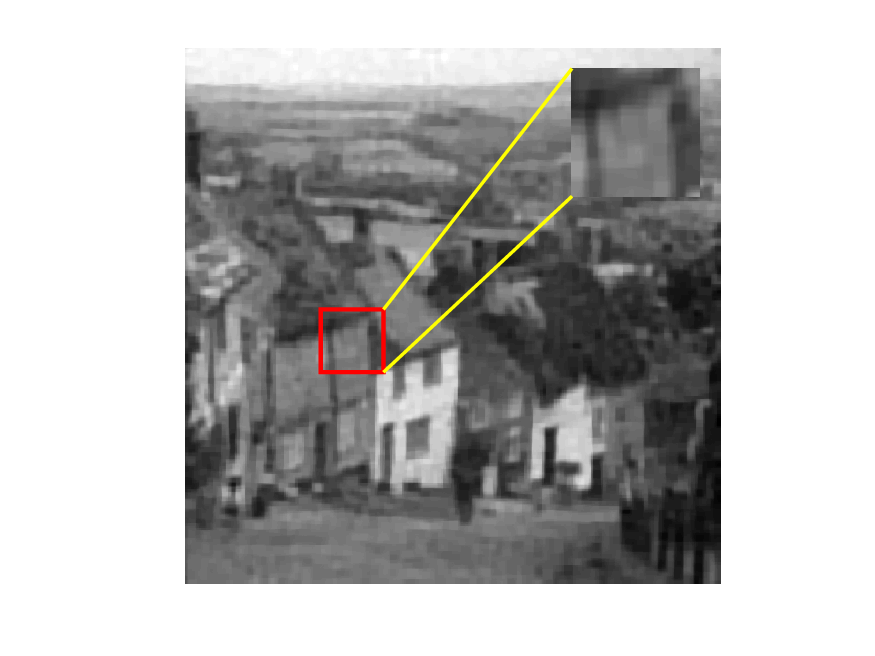}}
    }\\
     \subfigure[Gaussian, $\sigma_g = 10$]{
        \scalebox{0.35}{\includegraphics{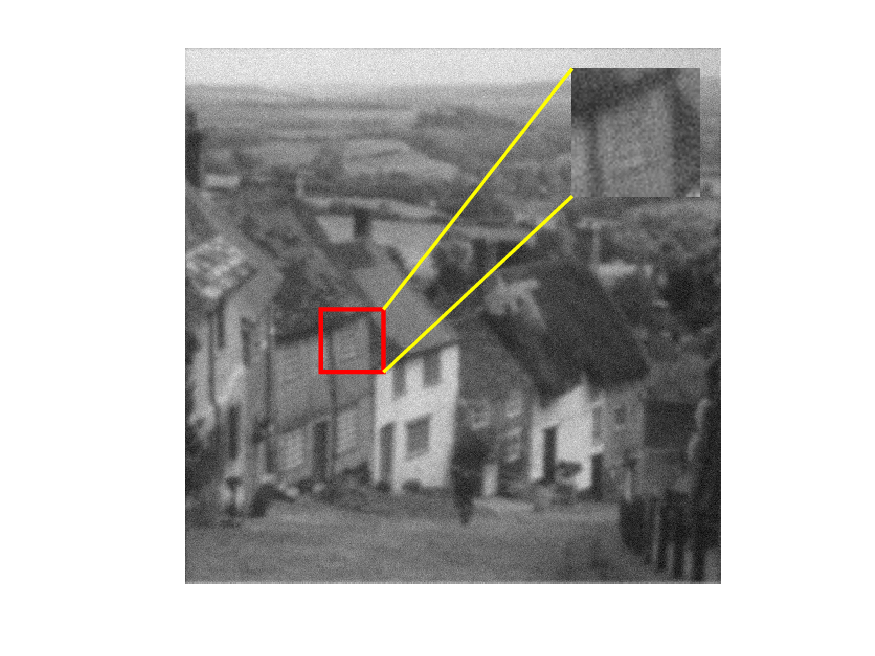}}
    } \hspace{-30pt}
      \subfigure[PFDR]{
        \scalebox{0.35}{\includegraphics{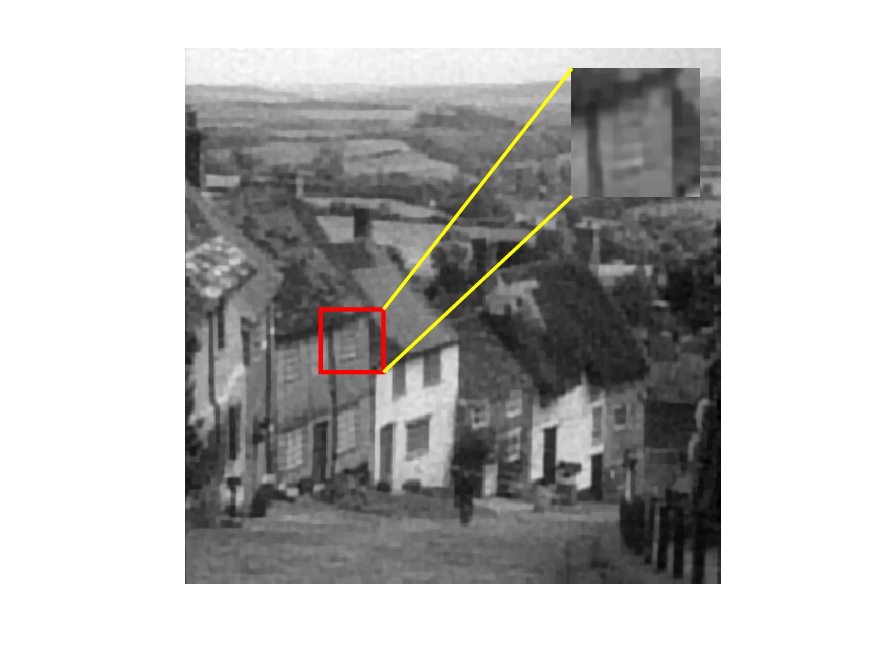}}
    } \hspace{-30pt}
    \subfigure[Proposed algorithm]{
        \scalebox{0.35}{\includegraphics{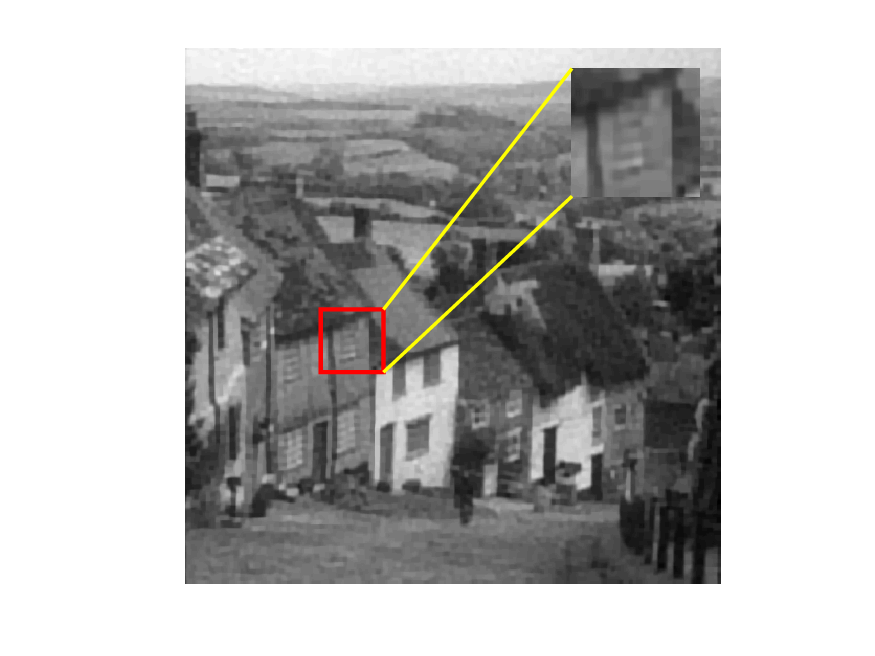}}
    }\\
    \subfigure[Gaussian, $\sigma_g = 20$]{
        \scalebox{0.35}{\includegraphics{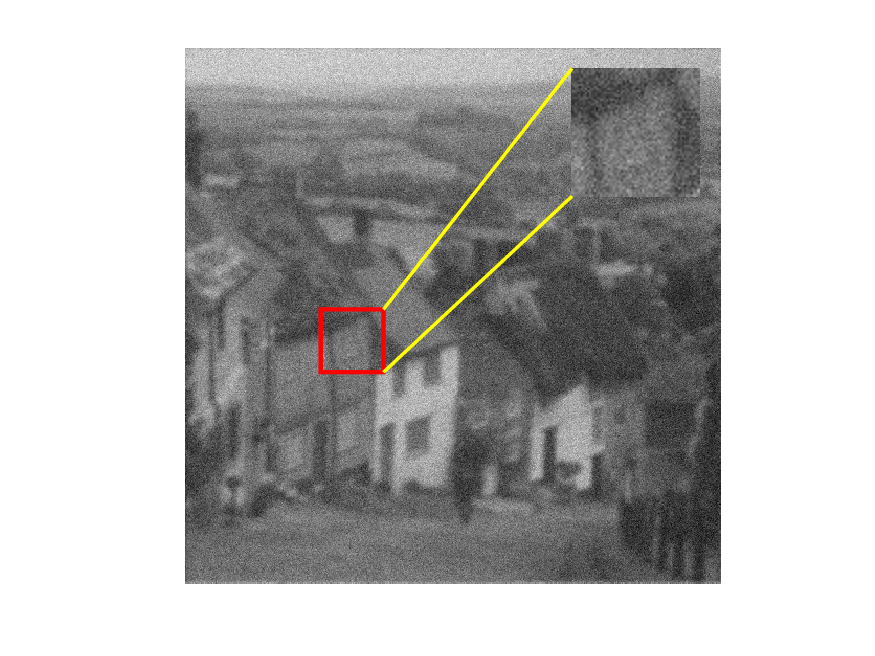}}
    } \hspace{-30pt}
      \subfigure[PFDR]{
        \scalebox{0.35}{\includegraphics{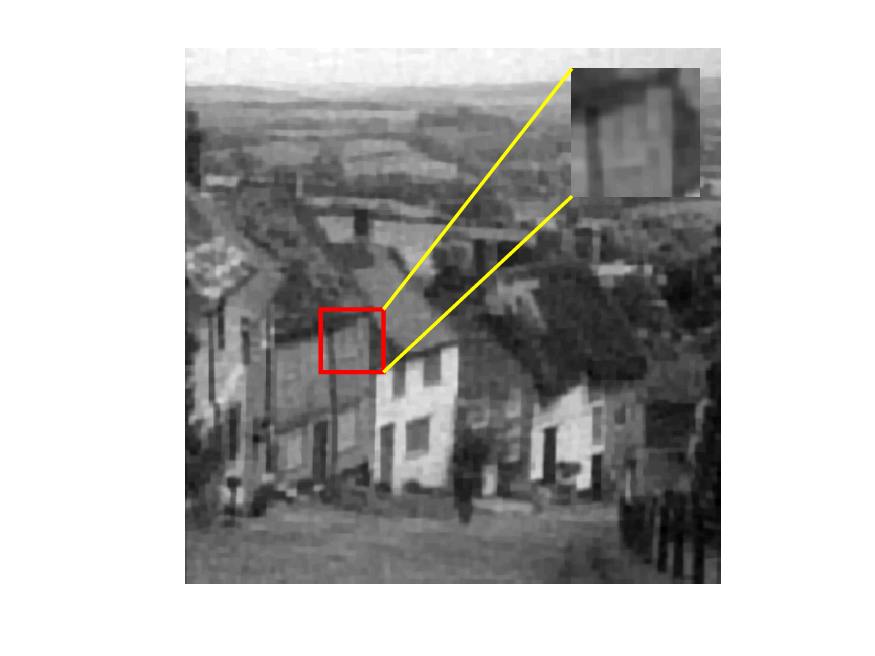}}
    } \hspace{-30pt}
    \subfigure[Proposed algorithm]{
        \scalebox{0.35}{\includegraphics{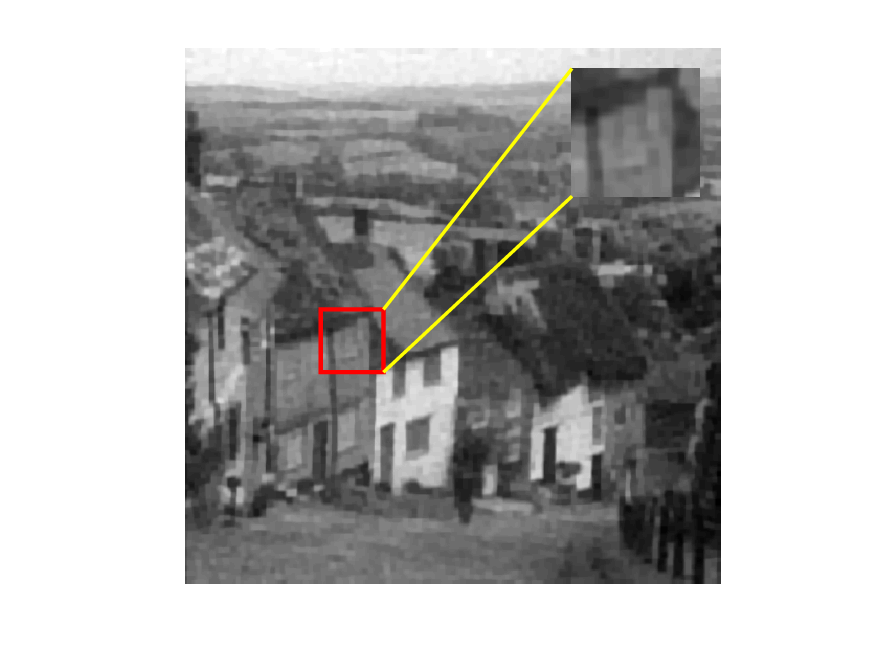}}
    }\\
    \caption{Corrupted and restored results of the ``Goldhill" image. The first column presents the corrupted images, whereas the second and third columns display the images restored using PFDR \cite{Tang2022JSC} and the proposed algorithm, respectively.}
    \label{restored2}
\end{figure}

\begin{figure}[H]
     \setlength{\abovecaptionskip}{0pt}
  \centering
  \makeatletter
    \subfigure[Uniform, $\sigma_g = 10$]{
        \scalebox{0.35}{\includegraphics{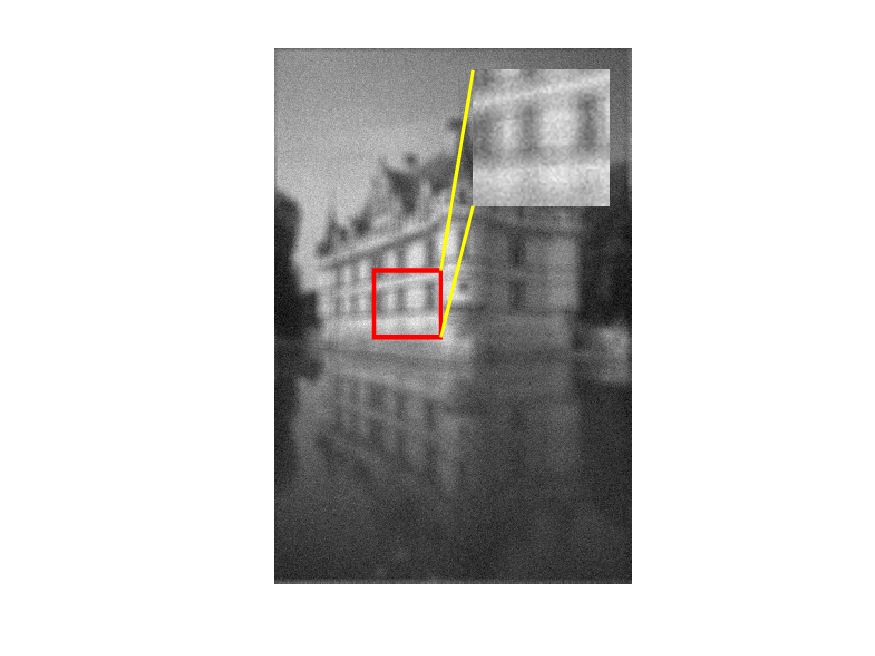}}
    } \hspace{-35pt}
      \subfigure[PFDR]{
        \scalebox{0.35}{\includegraphics{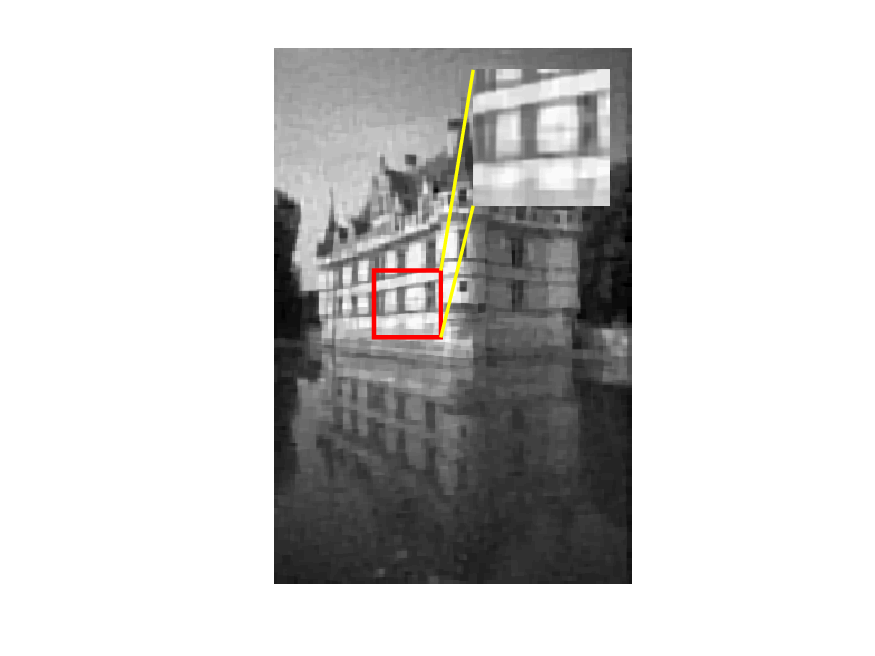}}
    } \hspace{-35pt}
    \subfigure[Proposed algorithm]{
        \scalebox{0.35}{\includegraphics{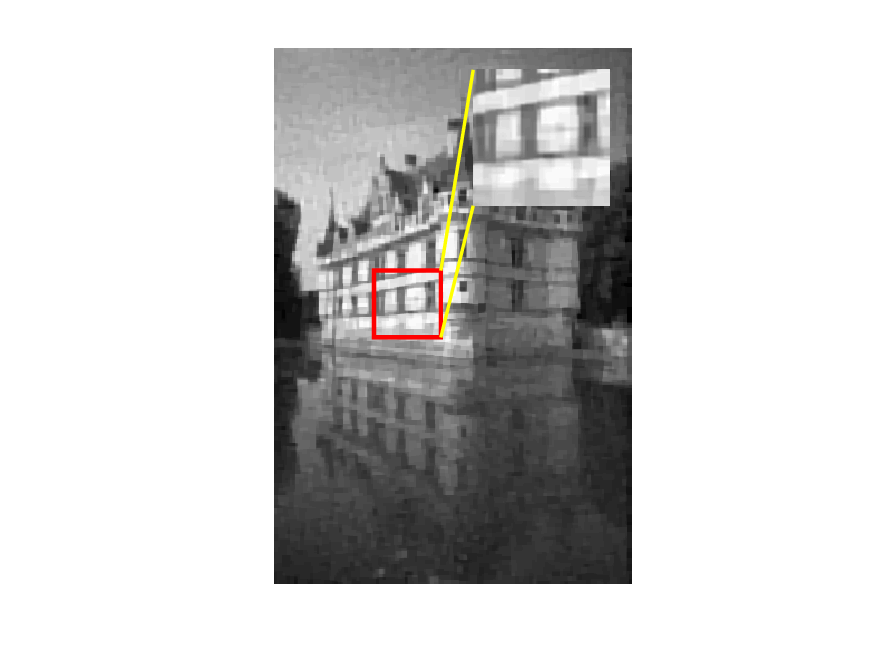}}
    }\\
    \subfigure[Uniform, $\sigma_g = 20$]{
        \scalebox{0.35}{\includegraphics{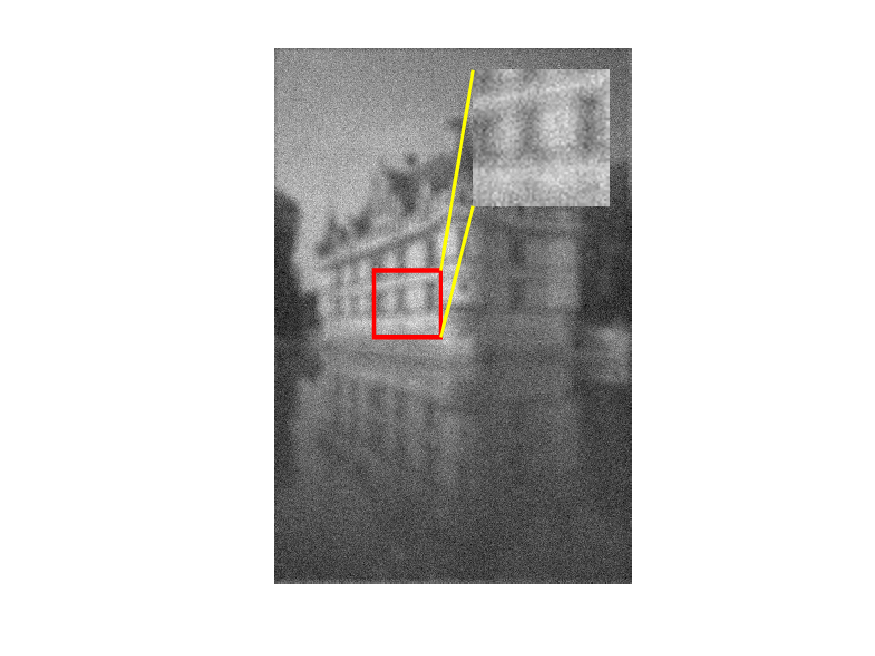}}
    } \hspace{-35pt}
      \subfigure[PFDR]{
        \scalebox{0.35}{\includegraphics{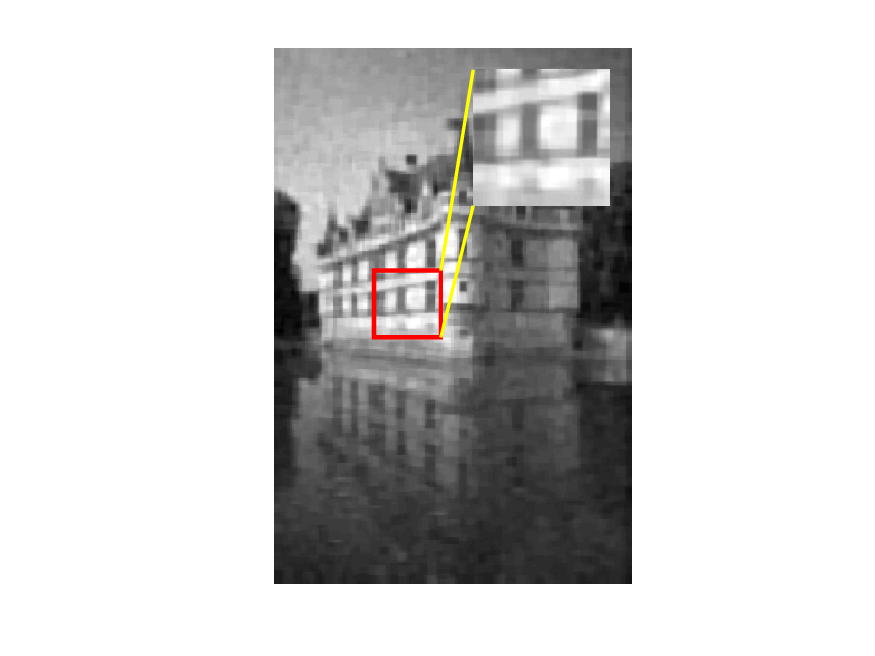}}
    } \hspace{-35pt}
    \subfigure[Proposed algorithm]{
        \scalebox{0.35}{\includegraphics{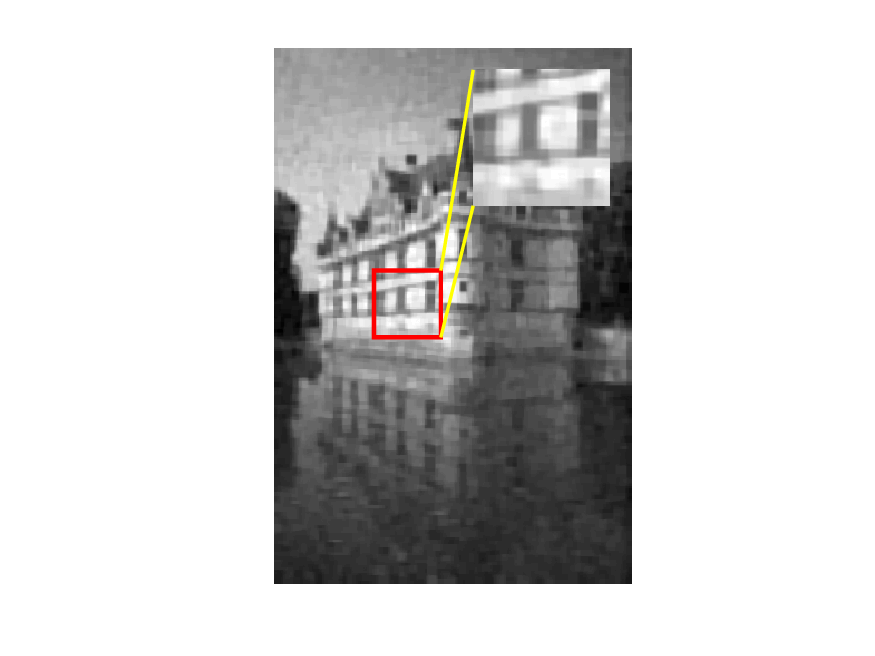}}
    }\\
     \subfigure[Gaussian, $\sigma_g = 10$]{
        \scalebox{0.35}{\includegraphics{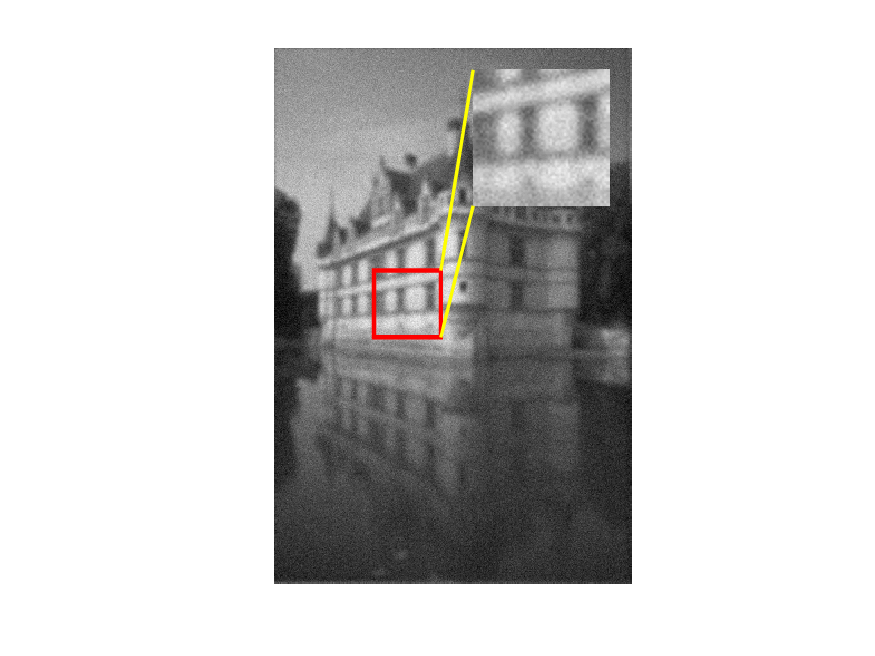}}
    } \hspace{-35pt}
      \subfigure[PFDR]{
        \scalebox{0.35}{\includegraphics{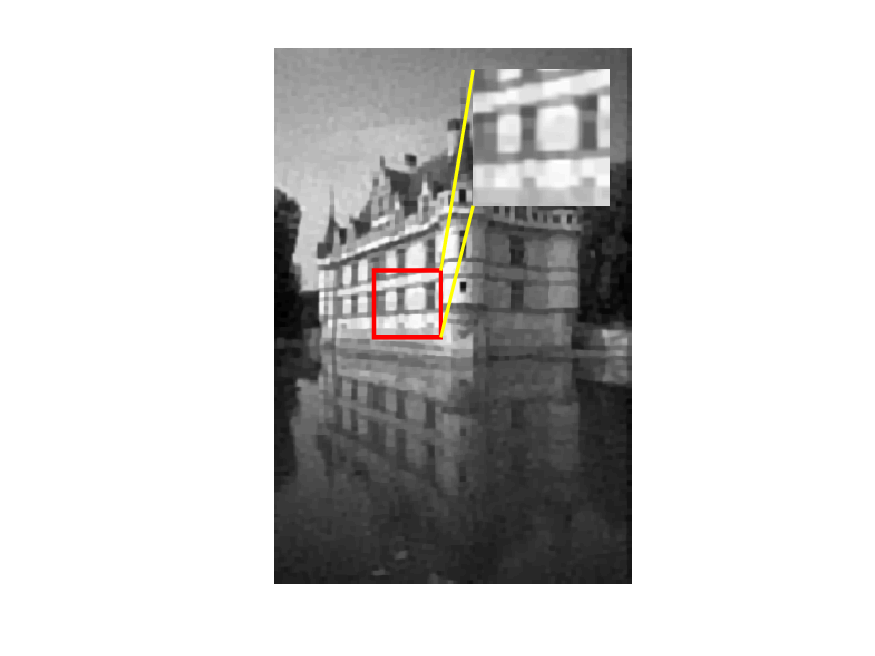}}
    } \hspace{-35pt}
    \subfigure[Proposed algorithm]{
        \scalebox{0.35}{\includegraphics{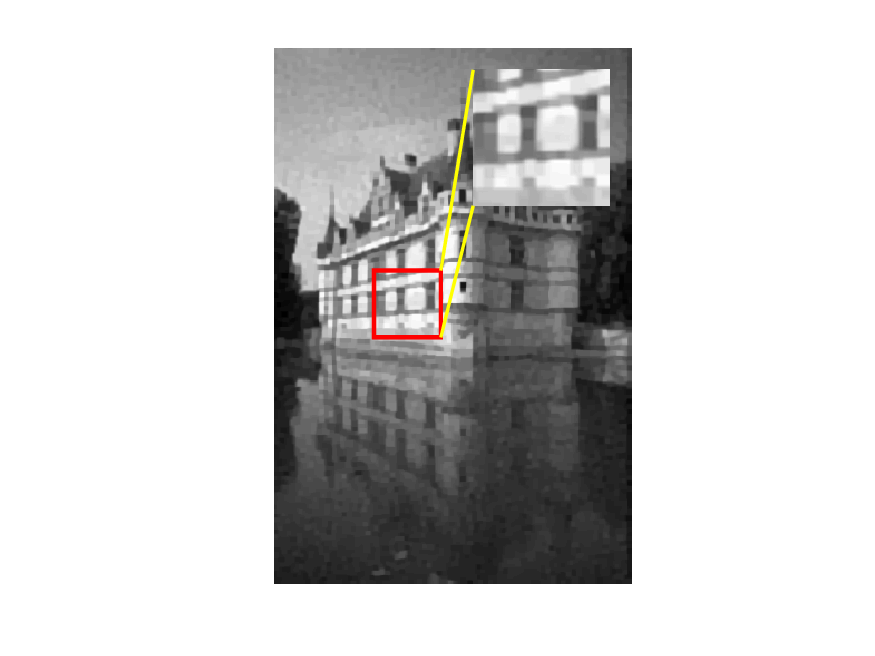}}
    }\\
    \subfigure[Gaussian, $\sigma_g = 20$]{
        \scalebox{0.35}{\includegraphics{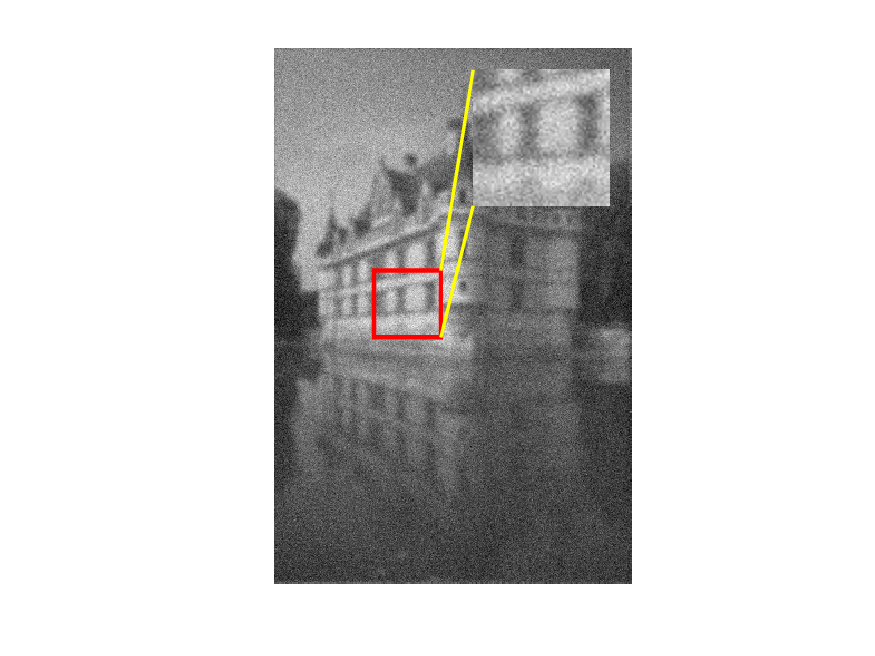}}
    } \hspace{-35pt}
      \subfigure[PFDR]{
        \scalebox{0.35}{\includegraphics{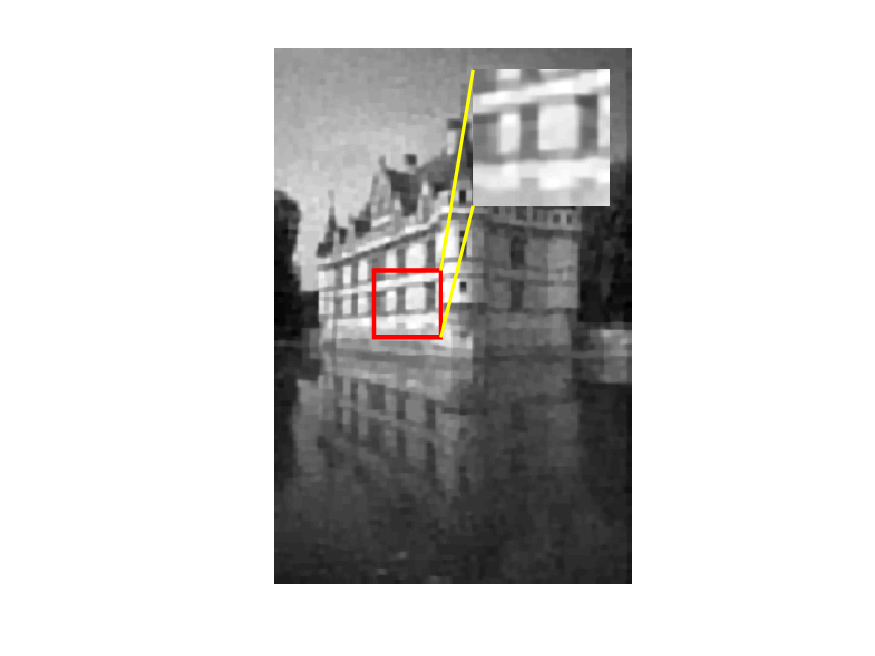}}
    } \hspace{-35pt}
    \subfigure[Proposed algorithm]{
        \scalebox{0.35}{\includegraphics{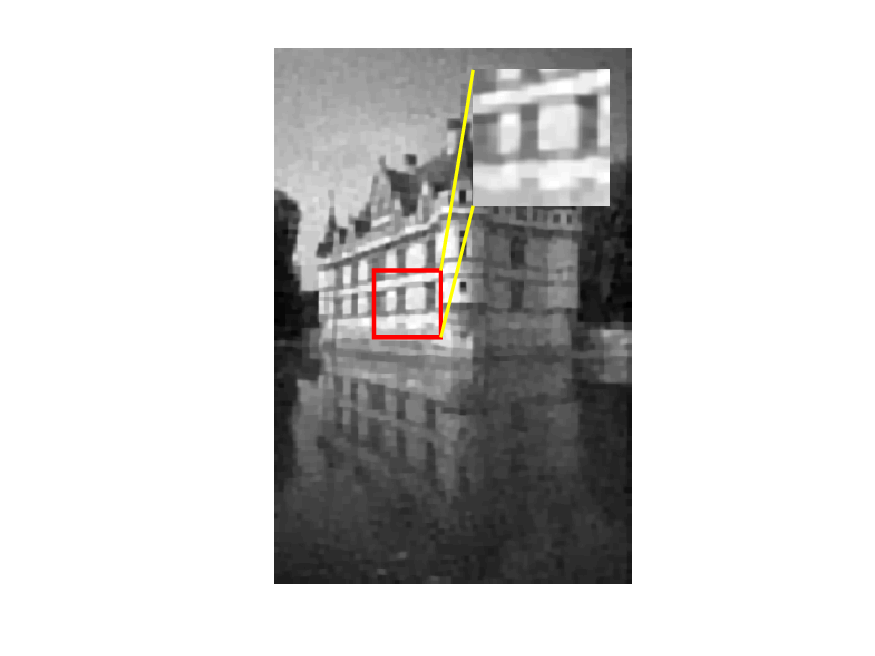}}
    }\\
    \caption{Corrupted and restored results of  the ``Castle" image. The first column presents the corrupted images, whereas the second and third columns display the images restored using PFDR \cite{Tang2022JSC} and the proposed algorithm, respectively. }
    \label{restored3}
\end{figure}


\subsection{Image denoising task}

In this subsection, we consider a general MC-TV image denoising problem, which is defined by
\begin{equation}\label{MC-denoising}
\begin{aligned}
& \min_{x\in R^{m\times n}}\, \frac{1}{2}\| x -b \|_{F}^{2} + \mu_1 \Psi_{a}^{MC}(Dx) + \mu_2 \|x\|_{*}, \\
& s.t. \, x \in C,
\end{aligned}
\end{equation}
where $b\in R^{m\times n}$ is the observed noisy image, $\Psi_{a}^{MC}(\cdot) = \sum \psi_{a}(\cdot)$ ($\psi_{a}$ denotes the MC penalty), $D$ denotes the first-order difference operator, $\|x\|_{*}$ denotes the nuclear norm, $C = \{x_{ij}\in R^{m\times n}: 0 \leq x_{ij} \leq 255, i= 1, \cdots, m,j =1,\cdots, n\}$, and $\mu_1 >0, \mu_2 >0$ are two regularization parameters. When the constraint set $C =R^{m\times n}$ and $\mu_2 =0$, (\ref{MC-denoising}) reduces to the original MC-TV, which was proposed by Selesnick et al. \cite{Selesnick-2020-JMIV}. Let the Moreau envelope of the function $f$ is defined as $f^{M}(x) = inf_{v} \{f(v) + \frac{1}{2}\|x-v\|^2\}$. It follows from the convex-nonconvex method, the general MC-TV problem (\ref{MC-denoising}) can be transformed into the following formulation,
\begin{equation}\label{MC-denoising2}
\min_{x\in R^{m\times n}}\, \frac{1}{2}\| x -b \|_{F}^{2} + \mu_1 \|Dx\|_{1} - a \mu_1 (\frac{1}{a}\|\cdot\|_{1})^{M}(Dx)  + \mu_2 \|x\|_{*} + \delta_{C}(x), \\
\end{equation}
which is convex under the condition of $0\leq a \leq \frac{1}{8\mu_1}$. Let $f_{1}(x)= \frac{1}{2}\| x -b \|_{F}^{2} - a \mu_1 (\frac{1}{a}\|\cdot\|_{1})^{M}(Dx) $, $g_{1}(x)=\mu_2 \|x\|_{*}$, $g_{2}(x)=\delta_{C}(x)$,
$h_{1}(x) = \mu_1 \|x\|_{1}$, $L_1 = D$. Then (\ref{MC-denoising2}) is a special case of (\ref{convex-composite}). According to \cite{Selesnick-2020-JMIV}, $\nabla f_{1}(x) = x -b - a \mu_1 D^{T}(Dx - soft_{\frac{1}{a}}(Dx))$ and $\nabla f_{1}$ has a Lipschitz constant of $1$, where $soft_{\frac{1}{a}}$ denotes the soft thresholding operator.

In the subsequent experiments, we select images (a)–(c) in Figure \ref{test-images} as the test images. Gaussian noise with zero mean and standard deviation $\sigma_g$ is added to each image. To achieve better denoising performance, for images with different noise levels, we search and tune the regularization parameters $\mu_1, \mu_2$ and the nonconvexity parameter $a$, while ensuring that problem (\ref{MC-denoising}) remains convex, so as to obtain the best possible denoising quality. The resulting parameter settings are summarized in Table \ref{regularizaton-2}.

\begin{table}[h]
\centering
\caption{Selected values of the regularization parameters $\mu_1, \mu_2$ and the nonconvexity parameter $a$ for the nonconvex image denoising model (\ref{MC-denoising}).}
\begin{tabular}{c|ccccccc}
\hline
\multirow{2}[1]{*}{Image}  & \multicolumn{3}{c}{$\sigma_g = 15$}  & & \multicolumn{3}{c}{$\sigma_g = 25$}   \\
\cline{2-4} \cline{6-8}
 &  $\mu_1$ & $\mu_2$ &  $a$ & & $\mu_1$ & $\mu_2$  & $a$  \\
\hline
Building
  & $4.8$ & $204.5$  & $0.0234$  & &  $6.8$   & $399.4$ & $0.0165$\\
Goldhill
  & $7.5$ & $62$  & $0.0083$ & & $14$ & $90.2$ & $0.0045$ \\
 Castle
  & $8.1$ & $58$  & $0.0139$ & & $14.2$ & $85.2$ & $0.0079$ \\
\hline
\end{tabular}
\label{regularizaton-2}
\end{table}

\subsubsection{Numerical results and discussions}

We compare the algorithm proposed in this paper with the PFDR algorithm from \cite{Tang2022JSC} in solving problem (\ref{MC-denoising}); the results are shown in Table \ref{denoising-results}. The results in Table \ref{denoising-results} clearly demonstrate the superiority of the proposed algorithm over the PFDR method in terms of reconstruction quality. Across all test images and noise levels, the proposed algorithm consistently achieves the highest PSNR and SSIM values, indicating its enhanced capability to recover fine details and preserve structural information. Although the numerical improvements may appear marginal, their consistency across diverse scenarios highlights the robustness and reliability of the proposed approach. Notably, the performance gains become more meaningful under higher noise levels, where accurate recovery is more challenging. While the proposed method requires a larger number of iterations and higher computational cost, this trade-off is justified by the improved restoration accuracy. Figure \ref{denoising-psnr-func} presents the evolution of the objective function values and PSNR curves with respect to the number of iterations under different noise levels. Furthermore, Figures \ref{Building-denoising}, \ref{Goldhill-denoising}, and \ref{Castle-denoising} present the images restored by the two algorithms, providing a visual comparison of their denoising performance.

\begin{table}[h]
\centering
\footnotesize
\caption{Numerical results for solving (\ref{MC-denoising}) in terms of PSNR (dB), SSIM, number of iterations, and CPU time (seconds).}
\begin{tabular}{c|c|ccc}
\hline
 \multirow{2}[1]{*}{Image}&  \multirow{2}[1]{*}{Noise level} & Input  & PFDR \cite{Tang2022JSC}   & Proposed algorithm  \\
& & PSNR/SSIM & PSNR/SSIM/Iter/CPU &   PSNR/SSIM/Iter/CPU  \\
\hline
\hline
 \multirow{2}[1]{*}{Building}
 & $\sigma_g =15$  & $24.6145/0.7236$ & $29.4727$/$0.8681$/$91$/$27.1$ &  $29.4755$/$0.8682$/$298$/$114.8$ \\
 & $\sigma_g = 25$   & $20.1804/0.5341$ & $27.0207$/$0.7930$/$160$/$50.2$  &    $27.0236$/$0.7931$/$375$/$88.6$   \\
\hline
\multirow{2}[1]{*}{Goldhill}
 & $\sigma_g =15$ & $24.6091/0.5273$ & $30.7700$/$0.8078$/$77$/$25.5$ &  $30.7796$/$0.8084$/$434$/$102.7$ \\
 & $\sigma_g = 25$ & $20.1844/0.3227$ & $28.6664$/$0.7297$/$127$/$54.9$  &    $28.6777$/$0.7306$/$658$/$169.7$   \\
\hline
\multirow{2}[1]{*}{Castle}
 & $\sigma_g =15$ & $24.6112/0.4647$ & $30.9838$/$0.8600$/$87$/$18.4$ &  $30.9898$/$0.8609$/$467$/$57.5$ \\
 & $\sigma_g = 25$  & $20.1742/0.3005$ & $28.2860$/$0.7971$/$128$/$27.4$  &    $28.2909$/$0.7984$/$654$/$92.6$   \\
\hline
\end{tabular}\label{denoising-results}

\end{table}

\begin{figure}[H]
     \setlength{\abovecaptionskip}{0pt}
  \centering
  \makeatletter
    \subfigure[Building, $\sigma_g = 15$]{
        \scalebox{0.4}{\includegraphics{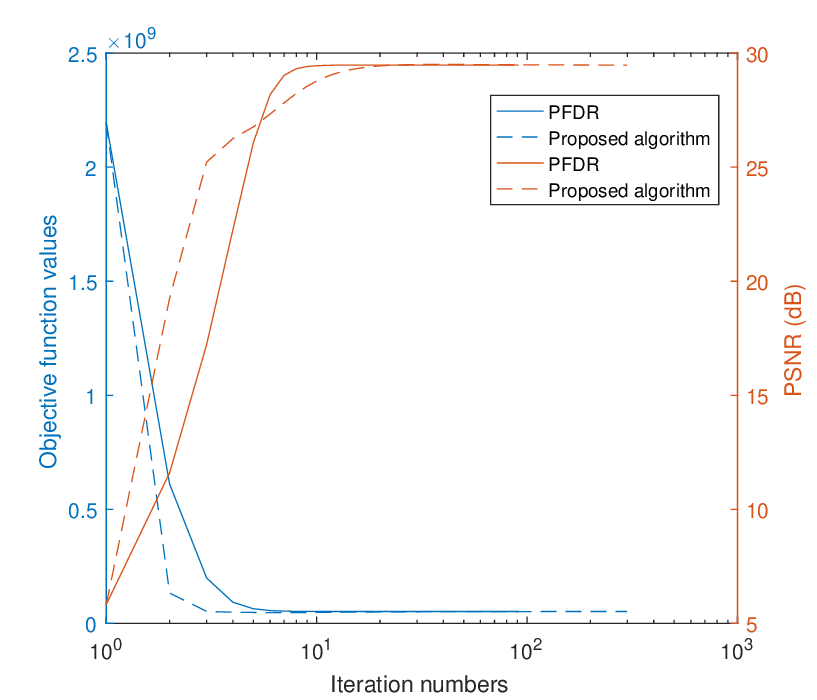}}
    } \hspace{-10pt}
      \subfigure[Building, $\sigma_g = 25$]{
        \scalebox{0.4}{\includegraphics{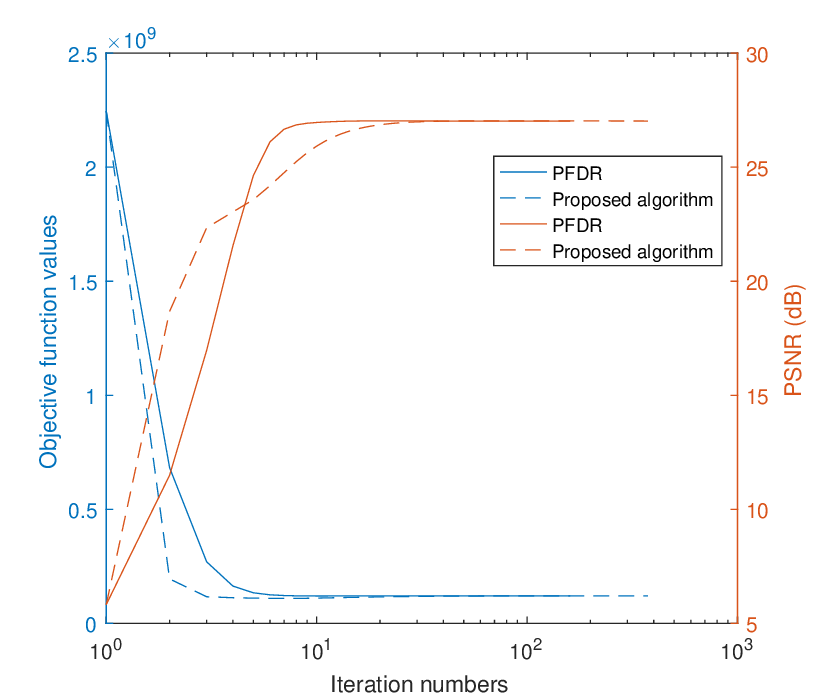}}
    } \\
    \subfigure[Goldhill, $\sigma_g = 15$]{
        \scalebox{0.4}{\includegraphics{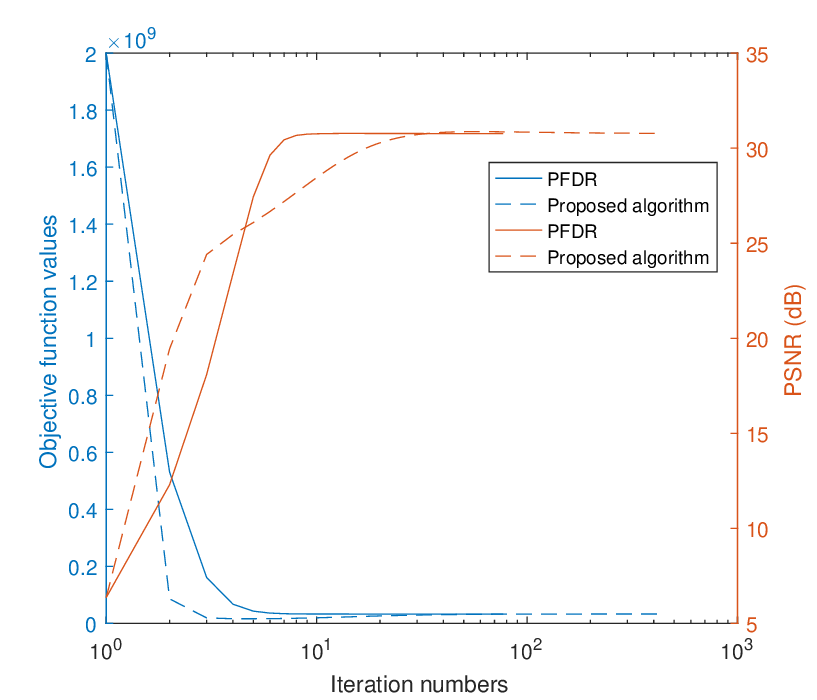}}
    } \hspace{-10pt}
      \subfigure[Goldhill, $\sigma_g = 25$]{
        \scalebox{0.4}{\includegraphics{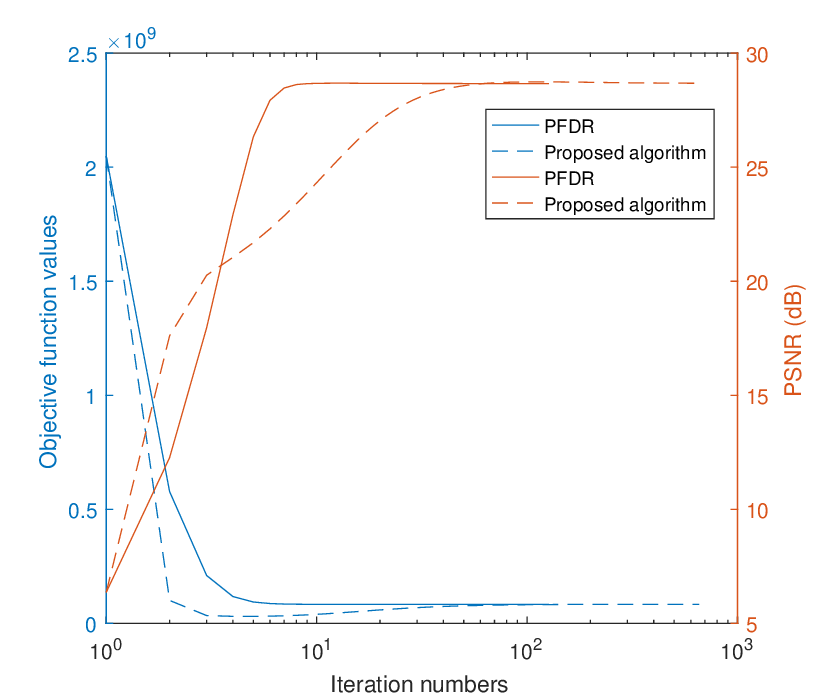}}
    }\\
     \subfigure[Castle, $\sigma_g = 15$]{
        \scalebox{0.4}{\includegraphics{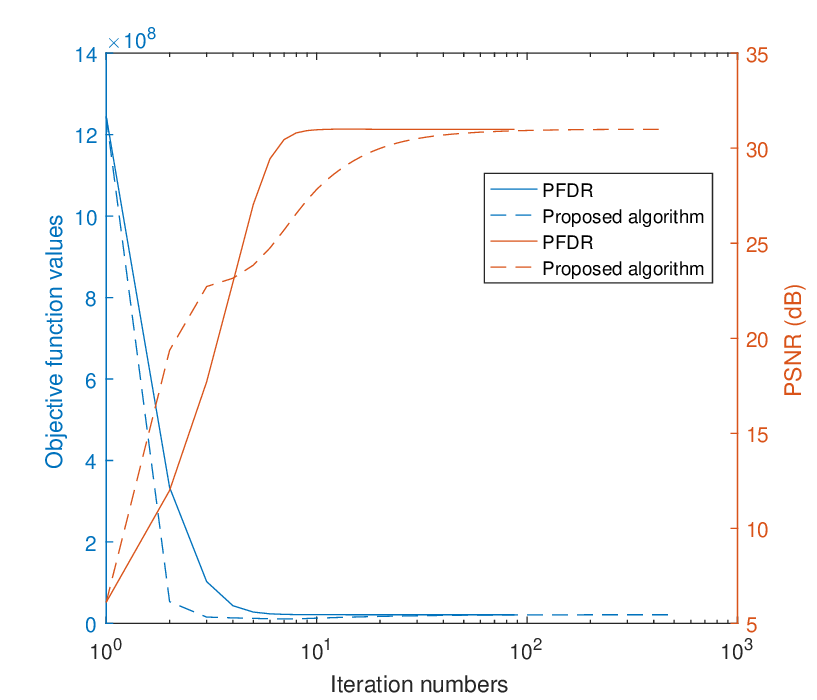}}
    } \hspace{-10pt}
      \subfigure[Castle, $\sigma_g = 25$]{
        \scalebox{0.4}{\includegraphics{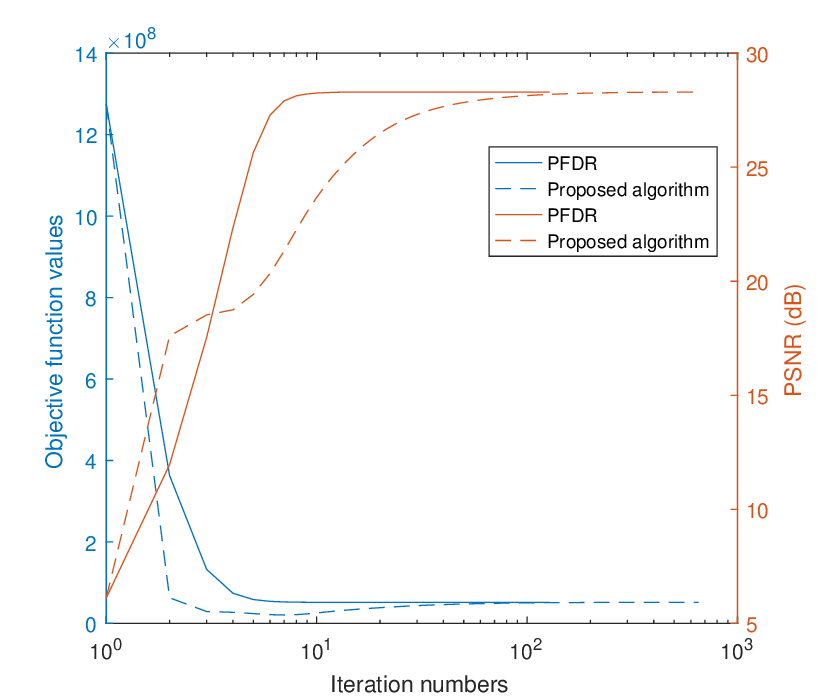}}
    }\\
    \caption{Objective function values and PSNR versus the number of iterations for the test images ``Building", ``Goldhill", and ``Castle" under different noise levels.}
    \label{denoising-psnr-func}
\end{figure}

\begin{figure}[H]
     \setlength{\abovecaptionskip}{0pt}
  \centering
  \makeatletter
    \subfigure[$\sigma_g = 15$]{
        \scalebox{0.35}{\includegraphics{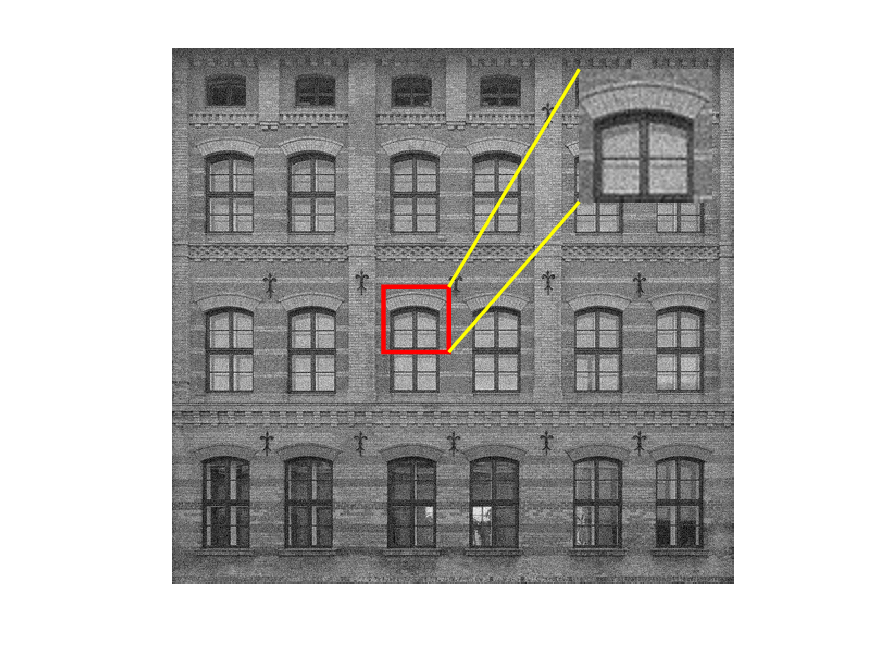}}
    } \hspace{-30pt}
      \subfigure[PFDR]{
        \scalebox{0.35}{\includegraphics{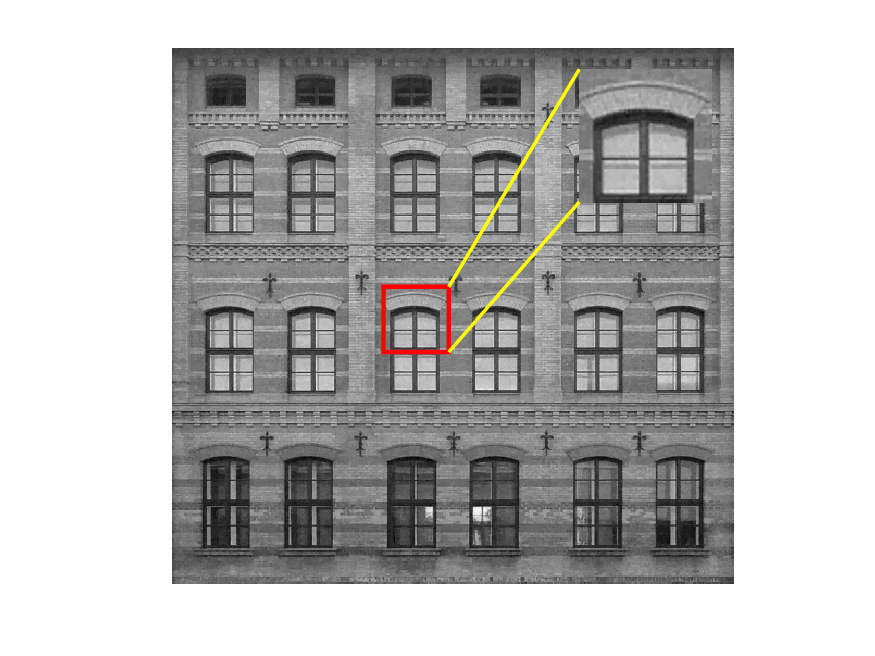}}
    } \hspace{-30pt}
    \subfigure[Proposed algorithm]{
        \scalebox{0.35}{\includegraphics{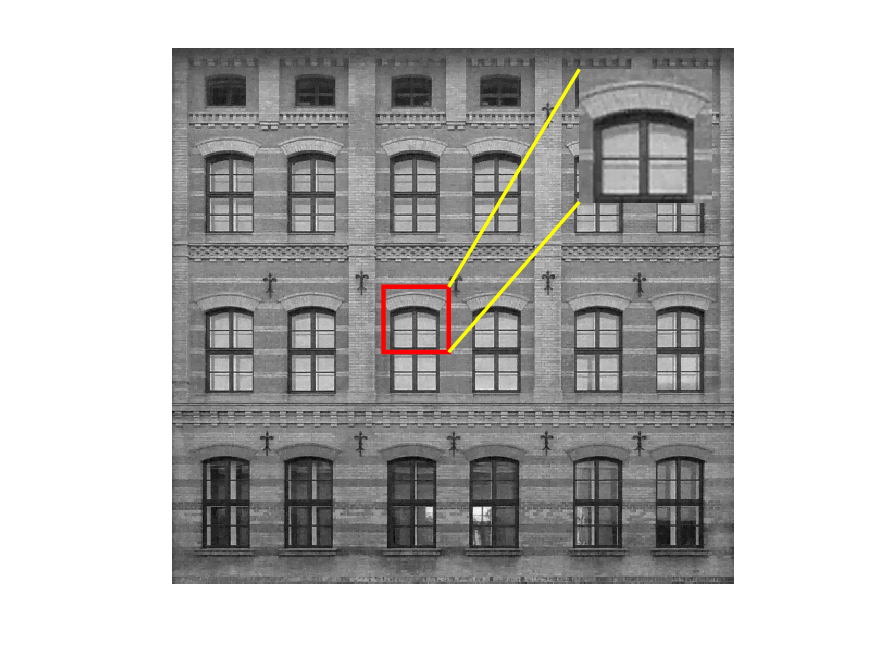}}
    }\\
    \subfigure[$\sigma_g = 25$]{
        \scalebox{0.35}{\includegraphics{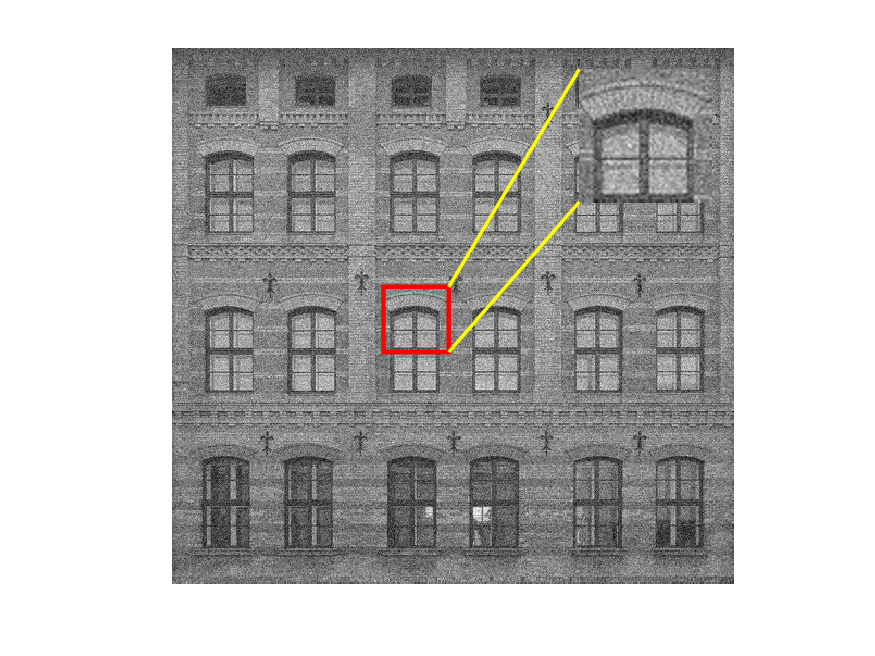}}
    } \hspace{-30pt}
      \subfigure[PFDR]{
        \scalebox{0.35}{\includegraphics{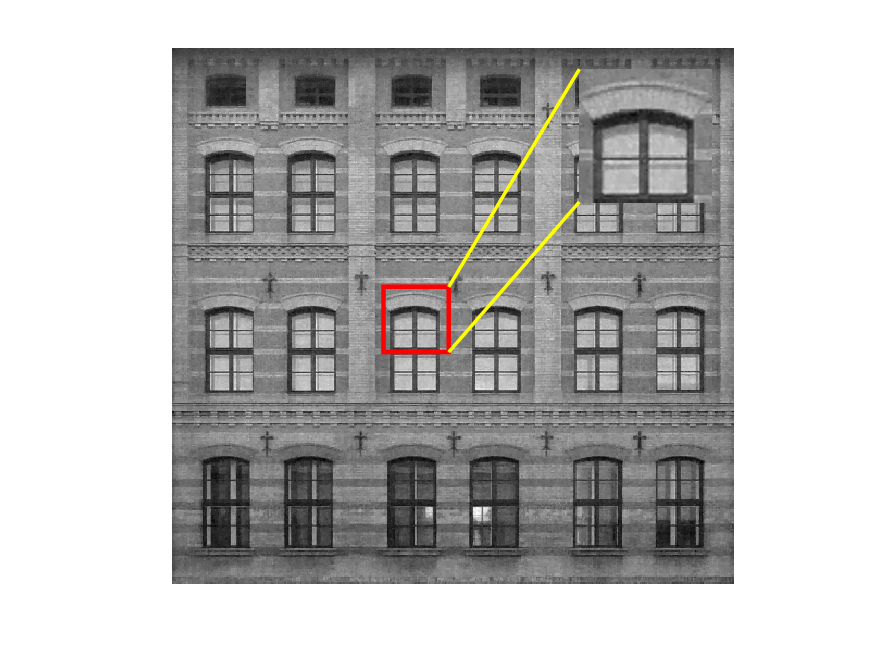}}
    } \hspace{-30pt}
    \subfigure[Proposed algorithm]{
        \scalebox{0.35}{\includegraphics{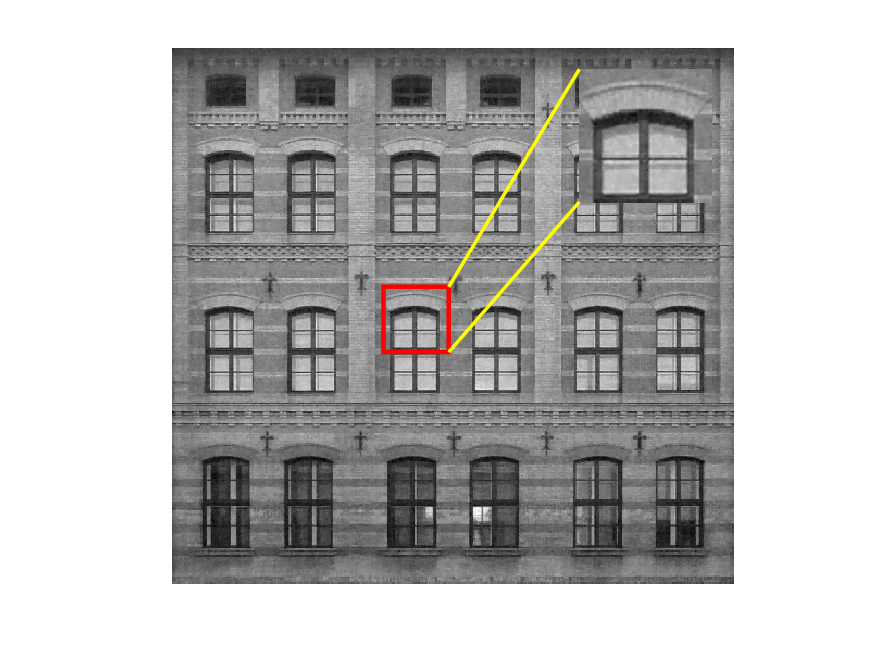}}
    }\\
    \caption{Corrupted and restored results for image denoising of  the ``Building" image. The first column presents the corrupted images, whereas the second and third columns show the images restored using PFDR \cite{Tang2022JSC} and the proposed algorithm, respectively. }
    \label{Building-denoising}
\end{figure}

\begin{figure}[H]
     \setlength{\abovecaptionskip}{0pt}
  \centering
  \makeatletter
    \subfigure[$\sigma_g = 15$]{
        \scalebox{0.35}{\includegraphics{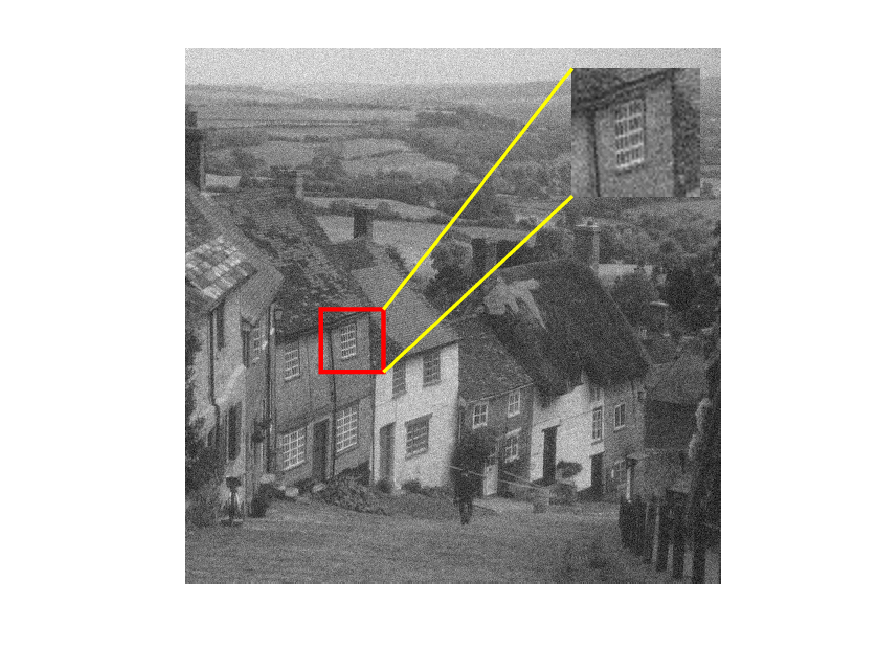}}
    } \hspace{-30pt}
      \subfigure[PFDR]{
        \scalebox{0.35}{\includegraphics{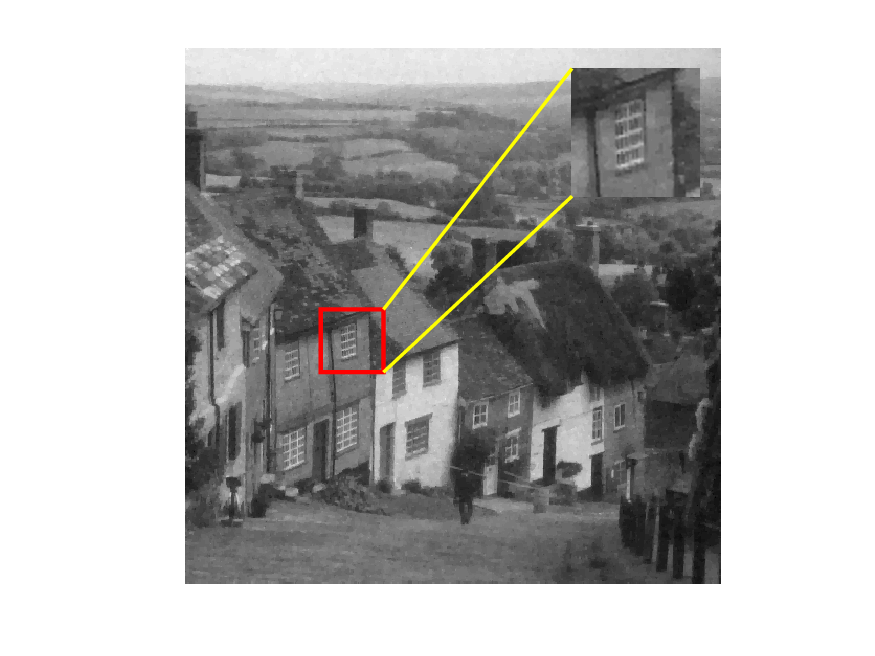}}
    } \hspace{-30pt}
    \subfigure[Proposed algorithm]{
        \scalebox{0.35}{\includegraphics{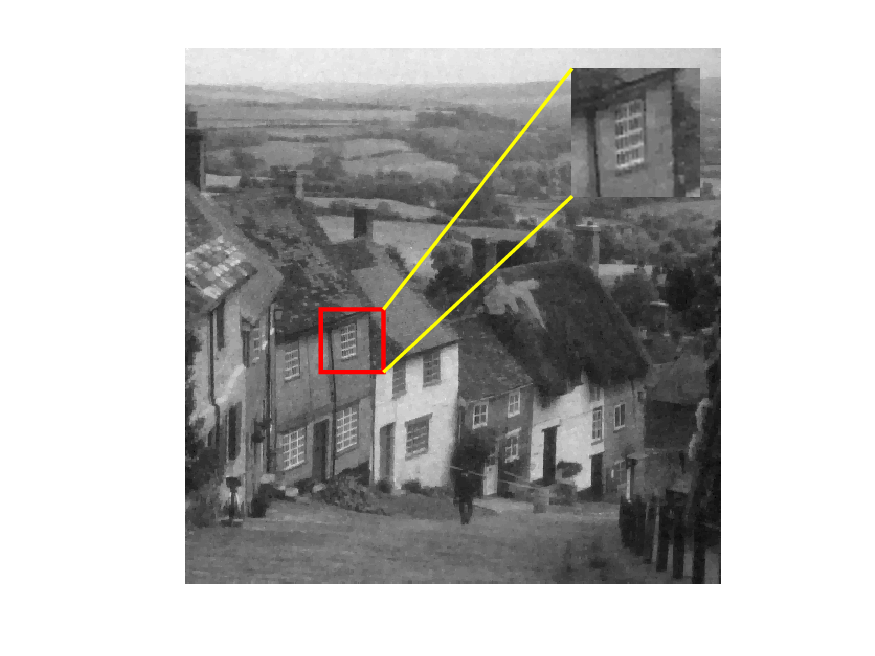}}
    }\\
    \subfigure[$\sigma_g = 25$]{
        \scalebox{0.35}{\includegraphics{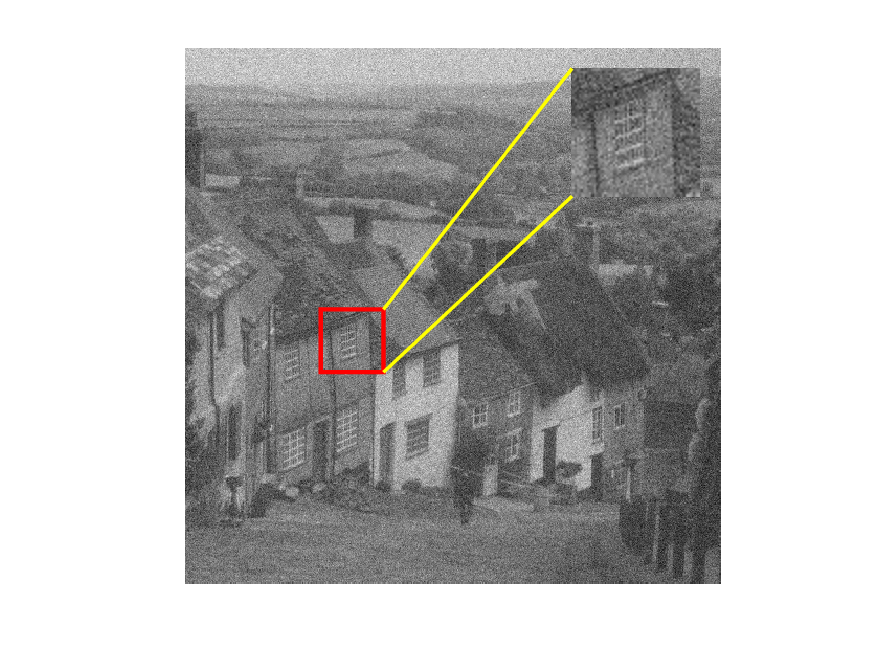}}
    } \hspace{-30pt}
      \subfigure[PFDR]{
        \scalebox{0.35}{\includegraphics{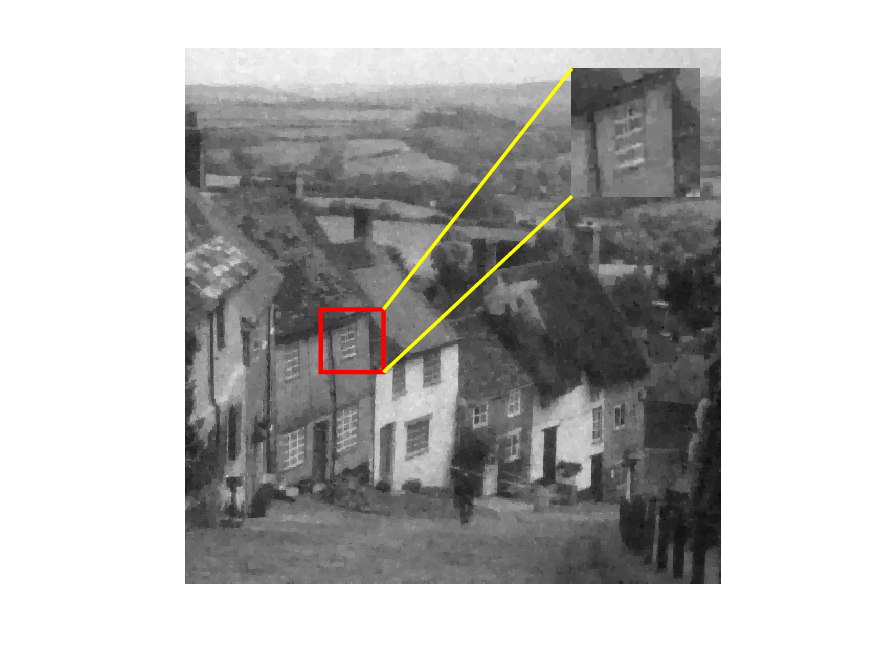}}
    } \hspace{-30pt}
    \subfigure[Proposed algorithm]{
        \scalebox{0.35}{\includegraphics{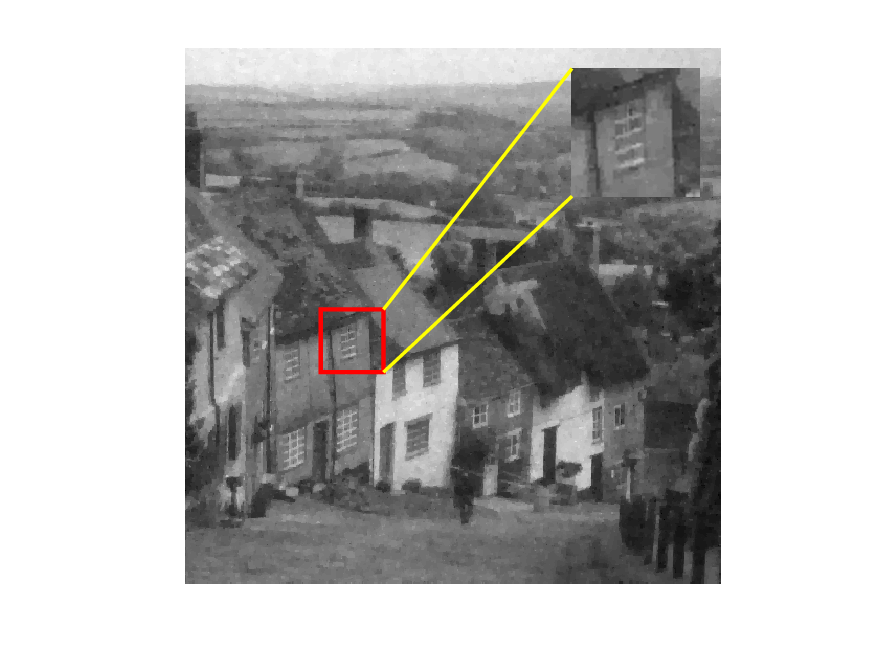}}
    }\\
    \caption{Corrupted and restored results for image denoising of  the ``Goldhill" image. The first column presents the corrupted images, whereas the second and third columns show the images restored using PFDR \cite{Tang2022JSC} and the proposed algorithm, respectively. }
    \label{Goldhill-denoising}
\end{figure}

\begin{figure}[H]
     \setlength{\abovecaptionskip}{0pt}
  \centering
  \makeatletter
    \subfigure[$\sigma_g = 15$]{
        \scalebox{0.35}{\includegraphics{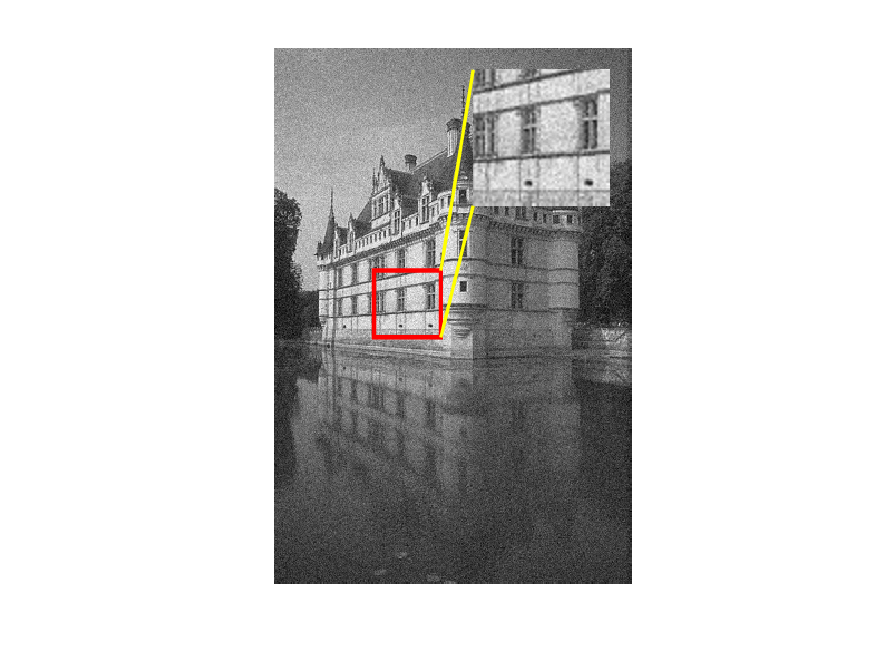}}
    } \hspace{-30pt}
      \subfigure[PFDR]{
        \scalebox{0.35}{\includegraphics{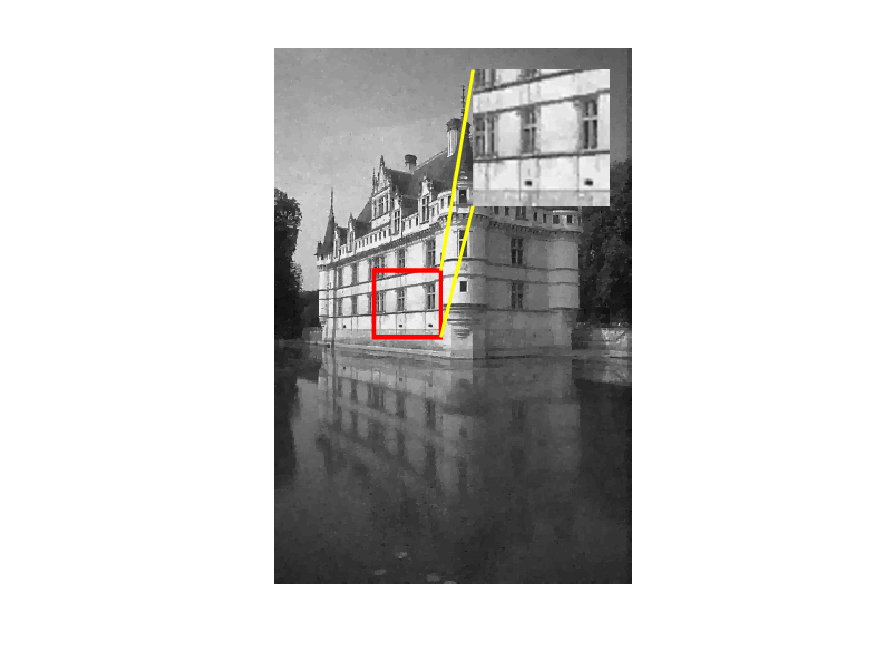}}
    } \hspace{-30pt}
    \subfigure[Proposed algorithm]{
        \scalebox{0.35}{\includegraphics{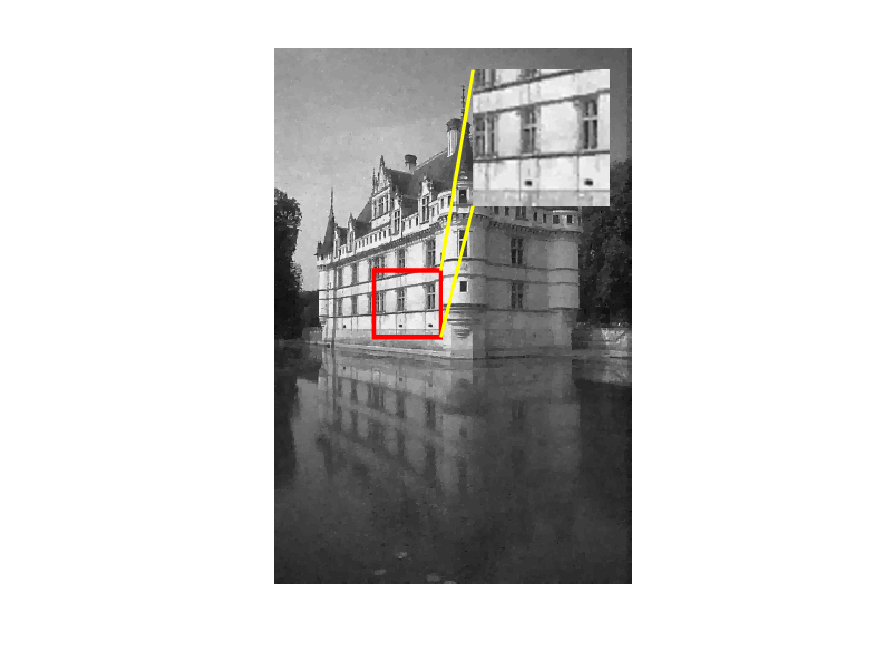}}
    }\\
    \subfigure[$\sigma_g = 25$]{
        \scalebox{0.35}{\includegraphics{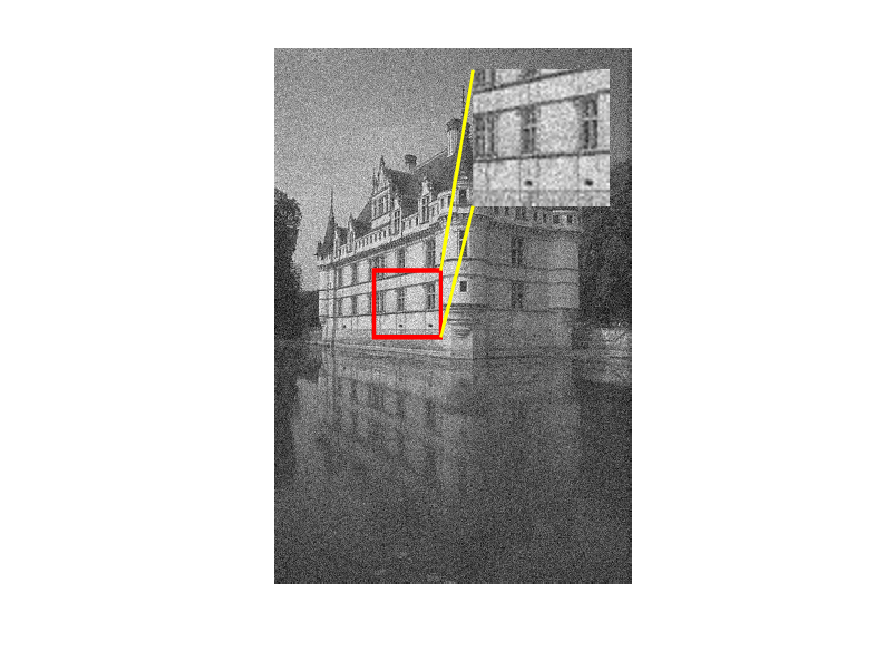}}
    } \hspace{-30pt}
      \subfigure[PFDR]{
        \scalebox{0.35}{\includegraphics{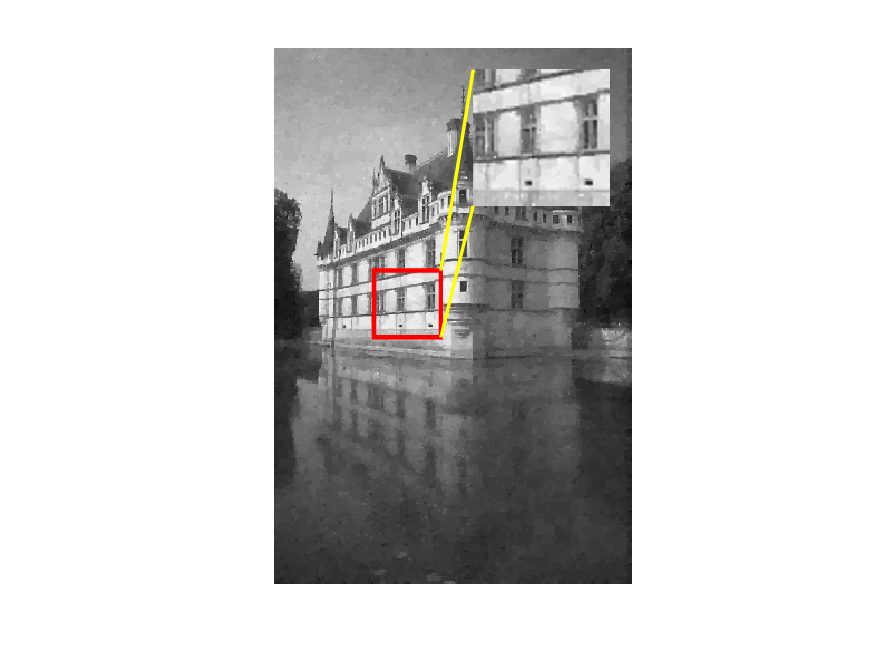}}
    } \hspace{-30pt}
    \subfigure[Proposed algorithm]{
        \scalebox{0.35}{\includegraphics{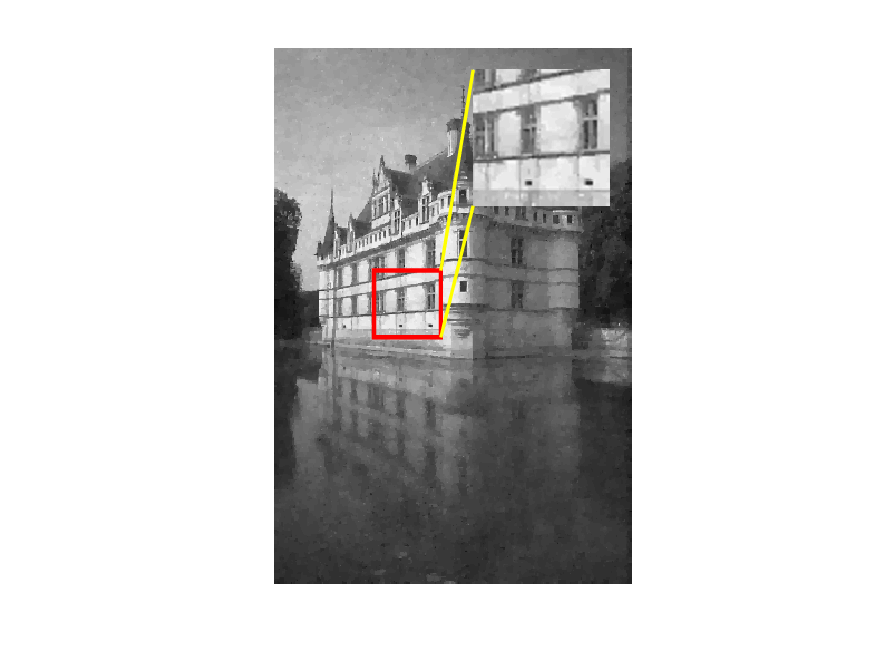}}
    }\\
    \caption{Corrupted and restored results for image denoising of  the ``Castle" image. The first column presents the corrupted images, whereas the second and third columns show the images restored using PFDR \cite{Tang2022JSC} and the proposed algorithm, respectively. }
    \label{Castle-denoising}
\end{figure}

\section{Conclusions}

In this work, we addressed a general class of monotone inclusion problems involving the sum of multiple maximally monotone operators, several cocoercive operators, and a composite term defined by the composition $L^{*}BL$, where $L$ is a bounded linear operator and $B$ is maximally monotone. This framework encompasses a wide range of structured convex optimization and saddle-point problems frequently encountered in imaging, signal processing, and variational analysis. To solve this class of problems, we proposed a novel primal–dual splitting algorithm that extends and unifies several well-known schemes.  We established the weak convergence of the algorithm under standard assumptions on monotonicity and cocoercivity, and demonstrated strong convergence under additional regularity conditions such as uniform monotonicity. Numerical results on image restoration tasks showed that the proposed method is competitive with existing approaches. Overall, our results extend the scope of operator splitting methods for solving composite monotone inclusions and contribute new theoretical and algorithmic tools to the field of monotone operator theory and convex optimization.
Future work may explore stochastic variants, inertial extensions, and applications to deep unfolding settings.

\section*{Acknowledgement}
We sincerely thank the editor and the anonymous reviewers for their valuable comments and constructive suggestions, which have greatly contributed to improving the quality of this work.


\section*{Funding}

This work was supported  by the National Natural Science Foundations of China (12031003, 12571491, 12571558), the Guangzhou Education Scientific Research Project 2024 (202315829), and the Jiangxi Provincial Natural Science Foundation (20224ACB211004).

\section*{Competing Interests}

The authors declare no competing interests.

\section*{Data Availability Statement}

The data that support the findings of this study are publicly available at the following GitHub repository: https://github.com/hhaaoo1331/A-primal-dual-splitting-algorithm-for-monotone-inclusions-with-applications.


\end{document}